\RequirePackage{fix-cm}
\documentclass[smallextended,numbook,envcountsame]{svjour3}       

\smartqed  
\usepackage{latexsym}
\usepackage{dsfont}
\usepackage{bbm}
\usepackage{amssymb}
\usepackage{amsmath}
\usepackage{amsfonts}
\usepackage{mathtools}
\usepackage{graphicx}
\usepackage{subcaption} 
\usepackage{multicol}
\usepackage{tikz}
\numberwithin{equation}{section}

\newcommand{\defeq}{\vcentcolon=}
\newcommand{\eqdef}{=\vcentcolon}

\newcommand{\N}{\mathbb{N}}
\newcommand{\Z}{\mathbb{Z}}
\newcommand{\Q}{\mathbb{Q}}
\newcommand{\R}{\mathbb{R}}
\newcommand{\B}{\mathfrak{B}}
\newcommand{\C}{\mathbb{C}}
\newcommand{\V}{\mathbb{V}}
\newcommand{\G}{\mathbb{G}}
\renewcommand{\O}{\mathbb{O}}
\newcommand{\U}{\mathbb{U}}
\newcommand{\Rnn}{\R_{\geq}}
\newcommand{\Rp}{\R_{>}}
\newcommand{\imag}{\mathrm{i}}

\newcommand{\Real}{\mathrm{Re}}
\newcommand{\Imag}{\mathrm{Im}}
\renewcommand{\S}{\mathcal{S}}
\newcommand{\W}{\mathcal{W}}
\newcommand{\m}{\mathfrak{m}}
\newcommand{\cZ}{\mathcal{Z}}

\newcommand{\Prob}{\mathbb{P}}
\DeclareMathOperator{\Var}{\mathbb{V}\mathrm{ar}}
\DeclareMathOperator{\Cov}{\mathrm{Cov}}
\newcommand{\E}{\mathbb{E}}
\newcommand{\A}{\mathcal{A}}
\renewcommand{\L}{\mathcal{L}}
\newcommand{\F}{\mathcal{F}}
\newcommand{\1}{\mathbbm{1}}

\newcommand{\Id}{\mathit{I}_{\mathit{d}}}

\newcommand{\Sd}{\mathbb{S}^{d-1}}
\renewcommand{\S}{\mathbb{S}}
\newcommand{\Orth}{\mathbb{O}\mathit{(d)}}
\newcommand{\SOd}{\mathrm{SO}\mathit{(d)}}
\newcommand{\Sim}{\mathbb{S}\mathit{(d)}}

\newcommand{\eqdist}{\stackrel{\mathrm{law}}{=}}

\newcommand{\comp}{\mathsf{c}}
\newcommand{\transp}{\mathsf{T}}
\DeclareMathOperator{\trace}{\mathrm{tr}}

\newcommand{\da}{\mathrm{d} \mathit{a}}
\newcommand{\ddo}{\mathrm{d} \mathit{o}}
\newcommand{\dr}{\mathrm{d} \mathit{r}}
\newcommand{\ds}{\mathrm{d} \mathit{s}}
\newcommand{\dt}{\mathrm{d} \mathit{t}}

\newcommand{\dx}{\mathrm{d} \mathit{x}}
\newcommand{\dy}{\mathrm{d} \mathit{y}}

\newcommand{\bC}{\mathbf{C}}

\newcommand{\bT}{\mathbf{T}}
\newcommand{\bX}{\mathbf{X}}

\newcommand{\AU}{\mathit{A}_{\U}}
\newcommand{\HAU}{\mathit{H}_{\!\mathit{A}_{\U}}}
\newcommand{\CU}{\mathit{C}_{\U}}
\newcommand{\SU}{\mathit{S}_{\U}}

\newcommand{\Au}{\mathit{A}_{U}}

\newcommand{\Cu}{\mathit{C}_{U}}
\newcommand{\Su}{\mathit{S}_{U}}

\newcommand{\abs}[1]{\left| #1 \right|}
\newcommand{\norm}[1]{\left\| #1 \right\|}
\newcommand{\scalar}[1]{\langle #1 \rangle}

\newcommand{\Step}{\noindent \textit{Step}\ }

\begin{document}

\title{Solutions to complex smoothing equations\thanks{Research supported by short visit grant 6172 from the European Science Foundation (ESF)
for the activity entitled `Random Geometry of Large Interacting Systems and Statistical Physics'.}
}


\author{Matthias Meiners       \and	Sebastian Mentemeier 
}

\authorrunning{M. Meiners and S. Mentemeier} 

\institute{M. Meiners	 \at
             Fachbereich Mathematik, Technische Universit\"at Darmstadt	\\
		Schlossgartenstra\ss e 7, 64289 Darmstadt, Germany	\\
              \email{meiners@mathematik.tu-darmstadt.de}           
           \and
           S. Mentemeier \at
           Fakult\"at f\"ur Mathematik, Technische Universit\"at Dortmund\\
           Vogelpothsweg 87, 44227 Dortmund, Germany	\\
 \email{sebastian.mentemeier@tu-dortmund.de}          
}

\date{Submitted to the arXiv.org: July 29, 2015}

\maketitle

\begin{abstract}
We consider smoothing equations of the form
\begin{equation*}
X	~\eqdist~	\sum_{j \geq 1} T_j X_j + C
\end{equation*}
where $(C,T_1,T_2,\ldots)$ is a given sequence of random variables and $X_1,X_2,\ldots$ are independent copies of $X$
and independent of the sequence $(C,T_1,T_2,\ldots)$.
The focus is on complex smoothing equations, i.e.,
the case where the random variables $X, C,T_1,T_2,\ldots$ are complex-valued,
but also more general multivariate smoothing equations are considered,
in which the $T_j$ are similarity matrices.
Under mild assumptions on $(C,T_1,T_2,\ldots)$, we describe the laws of all random variables $X$ solving the above smoothing equation.
These are the distributions of randomly shifted and stopped L\'evy processes satisfying a certain invariance property
called $(U,\alpha)$-stability, which is  related to operator (semi)stability.
The results are applied to various examples from applied probability and statistical physics.
\end{abstract}

\section{Introduction and main results}	\label{sec:intro}

\subsection{Smoothing equations}	\label{subsec:Smoothing equations}

In the paper at hand, we consider complex smoothing equations of the form
\begin{equation}	\label{eq:FP of ST inhom}
X	~\eqdist~	\sum_{j \geq 1} T_j X_j + C
\end{equation}
where $\eqdist$ denotes equality in law, $(C,T_1,T_2,\ldots)$ is a given sequence of complex random variables
and $X_1, X_2, \ldots$ are independent copies of the complex random variable $X$ and independent of $(C,T_1,T_2,\ldots)$.
The law of $X$ is called a solution\footnote{In
slight abuse of language, we will sometimes call a random variable $X$ a solution
if the distribution of $X$ is a solution to \eqref{eq:FP of ST inhom}.}
to \eqref{eq:FP of ST inhom}.
Under suitable assumptions on $(C,T_1,T_2,\ldots)$, we provide a complete description of the set of all solutions to \eqref{eq:FP of ST inhom}.

This kind of problem has a long history going back at least to \cite{Durrett+Liggett:1983,Holley+Liggett:1981,Kahane+Peyriere:1976}.
Recently, there has been progress leading to a complete description of the set of all solutions to \eqref{eq:FP of ST inhom}
in the case where all the $T_j$, $j \in \N$ are real and $C$ and $X$ are random vectors \cite{Alsmeyer+Meiners:2013,Iksanov+Meiners:2015a}.
We also refer to \cite{Alsmeyer+Biggins+Meiners:2012} for an overview over the early results in the field.

The complex case can naturally be embedded into the more general multivariate case,
in which $C$ and $X$ are $d$-dimensional and the $T_j$ are random $d \times d$ matrices, $d \in \N$.
In this setup, \eqref{eq:FP of ST inhom} has been solved under two different sets of assumptions,
namely, in the homogeneous case (i.e., $C=0$ a.s.) under assumptions that guarantee
that all solutions are essentially scale mixtures of multivariate normal distributions \cite{Bassetti+Matthes:2014},
and in the case that $C, T_1, T_2, \ldots$ have positive entries only
and attention is restricted to solutions $X$ on $[0,\infty)^d$ \cite{Mentemeier:2015}.
Yet, these results either cannot be applied to the complex case \cite{Mentemeier:2015}
or cover only a very special situation \cite{Bassetti+Matthes:2014}.
We will solve \eqref{eq:FP of ST inhom} in a multivariate setup that comfortably covers the complex case
and thereby address open questions in papers by Barral \cite[Remark 4]{Barral:2001}
(posed for $C=0$ and $T_1, T_2, \dots$ taking values in a Banach algebra); Chauvin et al.\;\cite[Remark 4.5]{Chauvin+Liu+Pouyanne:2014} 
and Madaule et al.\;\cite[Section 1.2]{Madaule+Rhodes+Vargas:2015} (in the complex case).

\subsection{Motivation}	\label{subsec:motivation}

We believe that solving \eqref{eq:FP of ST inhom} in a setup as general as possible
is of high theoretical value for probability theory.
Smoothing equations appear naturally in various fields of probability and statistical physics.
We mention here two classes of examples that will also motivate the choice of our setup.

\subsubsection{Complex smoothing equations in models of Applied Probability}	\label{subsubsec:Complex smoothing equations}

In various models of Applied Probability
including $b$-ary search trees \cite{Chauvin+Liu+Pouyanne:2014,Chauvin+Pouyanne:2004,Fill+Kapur:2004,Janson:2004},
P\'olya urn models \cite{Janson:2004,Knape+Neininger:2014,Pouyanne:2005}, the conservative fragmentation model \cite{Janson+Neininger:2008},
and B-urns \cite{Chauvin+Gardy+Pouyanne+Ton-That:2014}
phase transitions were shown for the limiting behavior of the distributions of quantities of interest.
All these phase transitions follow very similar patterns suggesting that they are particular instances of one universal phenomenon.

Before we sketch the general phenomenon, we describe it in the context of $b$-ary search trees.
The space requirement (i.e., the total number of nodes) in a $b$-ary search tree with $n$ keys inserted under the random permutation model,
centered by its mean and scaled by $\sqrt{n}$,
is known \cite{Lew+Mahmoud:1994} to be asymptotically normal when $b \leq 26$.
On the other hand, it exhibits stable periodic fluctuations \cite{Chauvin+Pouyanne:2004,Fill+Kapur:2004} around its mean when $b > 26$.

The general scheme is as follows. For each of the models mentioned above,
a characteristic equation can be formulated with several complex roots
(for $b$-ary search trees, $b$ appears as a parameter in the equation).
The root with largest real part always is $1$.
The asymptotic behavior of the quantity of interest is determined by the root with second largest real part, $\xi + \imag \eta$, say.
Roughly speaking,
if $\xi \leq \frac12$, the fluctuations around the expected size of the quantity are of order $\sqrt{n}$
(times possibly a slowly varying correction term),
where $n$ denotes the ``size'' of the model in an appropriate sense,
and the limiting distribution is normal.
When $\xi > \frac12$, the fluctuations around the mean are of order $n^{\xi}$
and there is no convergence but a periodic limiting behavior
the precise description of which involves solutions to complex smoothing equations.
Our results reflect the phase transition between normal and periodic limiting behavior with stable fluctuations.
Concrete examples are considered in detail in Section \ref{sec:examples}.

\subsubsection{Kinetic models}	\label{subsubsec:Kac caricature}

A main objective in the kinetic theory of gas is to understand the distribution $\mu_t$ on $\R^3$
of particle velocities and its evolution in time, given by the equation 
\begin{equation}	\label{eq:evolution equation}
\frac{\partial}{\partial t} \mu_t + \mu_t = S(\mu_t).
\end{equation}
Eq. \eqref{eq:evolution equation} has to be understood in the weak sense, i.e.,
$$ \frac{\partial}{\partial t} \int f \mathrm{d}\mu_t + \int f \mathrm{d} \mu_t = \int f \mathrm{d} S(\mu_t)$$
for all bounded continuous functions $f$ on $\R^3$.
Here $S$  is a mapping on probability measures, called collisional-gain operator, that describes the change in $\mu_t$ due to the collision of two uncorrelated particles. Its form can be deduced from the equations of energy and momentum conservation (see \cite{Villani:2002} for a detailed survey on the topic).

In order to understand qualitative aspects of \eqref{eq:evolution equation}, simplifications of $S$ have been considered, the most important one being the Kac caricature of the Boltzmann equation \cite{Kac:1956}, where $\mu_t$ is restricted to the real line, and 
\begin{equation*}
S : \mu \mapsto\ \textrm{law of}\,\big( \sin (\Theta) V_1 + \cos(\Theta) V_2 \big)
\end{equation*}
where  $V_1, V_2$ are i.i.d.\ with law $\mu$ and assumed to be independent of the random arc length $\Theta$
which has the uniform distribution on $[0, 2\pi)$. In this case, only energy is conserved, but not momentum.
Any steady state distribution $\mu_\infty$, i.e., satisfying $\frac{\partial}{\partial t} \mu_\infty =0$,
is a solution of the smoothing equation $S(\mu_{\infty})=\mu_{\infty}$
with real valued weights. Such equations were investigated in \cite{Iksanov+Meiners:2015a}.

Recently, there has been a lot of interest in generalizations of the Kac caricature \cite{Bassetti:2012,Bassetti:2011,Bobylev+al:2009}, also with applications in economic theory \cite{Chakraborti+al:2011,Cordier+al:2005,Matthes+Toscani:2008}, to mention just a few.

Of particular interest for us is a generalization of the Kac model in $\R^3$, considered by Bassetti and Matthes \cite[Section 6.2]{Bassetti+Matthes:2014}.
Here, a random vector $V$ the law of which is the steady state distribution, satisfies
\begin{equation}	\label{eq:Bassetti+Matthes}
V ~\eqdist~ L V_1 + R V_2,
\end{equation}
where $V, V_1, V_2$ are i.i.d.\ and independent of the random pair $(L,R)$ of similarity matrices\footnote{
By a $d \times d$ \emph{similarity matrix} we mean a $d \times d$ matrix that can be written as a scale multiple of an orthogonal matrix.}
which is subject to the condition $\E [\norm{L}^2 + \norm{R}^2] =1$ and density assumptions,
which guarantee that there is a unique solution (up to scaling) to \eqref{eq:Bassetti+Matthes}.
It is shown that this solution is a mixture of multivariate Gaussian distributions.
We solve \eqref{eq:Bassetti+Matthes} in much greater generality.
In particular, dropping their density assumption,
we show that larger classes of Gaussian solutions appear.

\subsection{Setup and main results}	\label{subset:Main results}

We aim for solving \eqref{eq:FP of ST inhom} in a multivariate setup covering both,
the complex equation and the multivariate equation as outlined in Section \ref{subsubsec:Kac caricature}.

Fix $d \in \N$, denote by $\abs{x}=(x_1^2 + \dots x_d^2)^{1/2}$ for $x=(x_1, \dots, x_d) \in \R^d$
the Euclidean norm and by $\norm{a}\defeq\sup \{\abs{ax} \, : \, \abs{x}=1\}$ for a real $d \times d$-matrix $a$ the associated matrix norm.
Assume that we are given a sequence $(C, T_1, T_2, \ldots) \eqdef (C,T)$
where $C$ is a $d$-dimensional random vector and $T = (T_j)_{j \in \N}$ is a sequence of
random $d \times d$ similarity matrices, i.e., for each $j \in \N$, $T_j = \norm{T_j} O_j$
where $O_j \in \Orth$,
with $\Orth$ denoting the group of orthogonal $d \times d$ matrices.
Notice that $\norm{T_j}$ and $O_j$ are random elements of $\Rnn=[0,\infty)$ and $\Orth$, respectively, which in general depend on each other. 
Further, we point out that no assumptions on the dependence structure of the sequence $(C, T)$ are imposed.

This setup includes the case where $X, C, T_1, T_2, \ldots$ are complex random variables,
henceforth referred to as the \emph{complex case},
since $\C$ can be identified with $\R^2$
and multiplication by a complex number $r e^{\imag \theta}$, $r \geq 0$, $\theta \in [0,2\pi)$
corresponds to multiplication in $\R^2$ from the left by the similarity matrix
\begin{equation*}
\begin{pmatrix}
r \cos \theta & -r \sin \theta	\\
r\sin \theta & r \cos \theta
\end{pmatrix}.
\end{equation*}

We proceed by introducing the assumptions and some basic concepts needed for the statement of our main result.

\subsubsection{Assumptions}	\label{subsubsec:Assumptions}

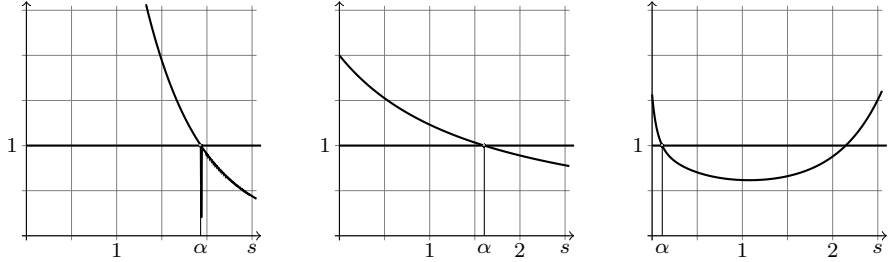
\begin{figure}[b]
		\begin{subfigure}[t]{0.3\textwidth}	
		\begin{center}
		\begin{tikzpicture}[scale=0.6]
		\draw [help lines] (-0.1,-0.1) grid (5.1,5.1); 
		\draw [->] (0,-0.1) -- (0,5.2);
		\draw [->] (-0.1,0) -- (5.2,0);
		\draw [-,thick] (5.2,2) -- (0,2) node[left] {$1$};
		\draw (2,0) node[below]{$1$};
		\draw (5,0) node[below]{$s$};
		\draw [thick, domain=2.65:5.1, samples=500] plot (\x, {2*(1/(0.5*0.517*(\x/2)+0.5))*(1/((1/3)*0.517*(\x/2)+2/3))*(1/(0.25*0.517*(\x/2)+3/4))*(1/(0.2*0.517*(\x/2)+4/5))*(1/(1/6*0.517*(\x/2)+5/6))*(1/(1/7*0.517*(\x/2)+6/7))*(1/(1/8*0.517*(\x/2)+7/8))*(1/(1/9*0.517*(\x/2)+8/9))*(1/(0.1*0.517*(\x/2)+9/10))*(1/(1/11*0.517*(\x/2)+10/11))*(1/(1/12*0.517*(\x/2)+11/12))*(1/(1/13*0.517*(\x/2)+12/13))*(1/(1/14*0.517*(\x/2)+13/14))*(1/(1/15*0.517*(\x/2)+14/15))*(1/(1/16*0.517*(\x/2)+15/16))*(1/(1/17*0.517*(\x/2)+16/17))*(1/(1/18*0.517*(\x/2)+17/18))*(1/(1/19*0.517*(\x/2)+18/19))*(1/(1/20*0.517*(\x/2)+19/20))*(1/(1/21*0.517*(\x/2)+20/21))*(1/(1/22*0.517*(\x/2)+21/22))*(1/(1/23*0.517*(\x/2)+22/23))*(1/(1/24*0.517*(\x/2)+23/24))*(1/(1/25*0.517*(\x/2)+24/25))*(1/(1/26*0.517*(\x/2)+25/26))*(1/(1/27*0.517*(\x/2)+26/27))});
		\draw [-,color=black] (3.86,2) -- (3.86,0) node[below] {$\alpha$}; 
		\shade[ball color=black] (3.86,2) circle(0.3ex);
		\end{tikzpicture}
		\end{center}
		\caption{$m(s)$ for Eqs. \eqref{eq:Fill+Kapur} and \eqref{eq:Chauvin+Liu+Pouyanne} with $b=27$}
		\label{subfig:m(s) b-ary}
	\end{subfigure}
\hfill
	\begin{subfigure}[t]{0.3\textwidth}
		\begin{center}
		\begin{tikzpicture}[scale=0.6]
		\draw [help lines] (-0.1,-0.1) grid (5.1,5.1); 
		\draw [->] (0,-0.1) -- (0,5.2);
		\draw [->] (-0.1,0) -- (5.2,0);
		\draw [-,thick] (5.2,2) -- (0,2) node[left] {$1$};
		\draw (2,0) node[below]{$1$};
		\draw (4,0) node[below]{$2$};
		\draw (5,0) node[below]{$s$};
		\draw [thick, domain=0:5.1, samples=500] plot (\x, {4/(1+cos(deg(2*pi/7))*(\x/2))});
		\draw [-,color=black] (3.21,2) -- (3.21,0) node[below] {$\alpha$}; 
		\shade[ball color=black] (3.21,2) circle(0.3ex);
		\end{tikzpicture}
		\end{center}
		\caption{$m(s)$ for Eq.\,\eqref{eq:cyclic Polya} with $b=7$}
		\label{subfig:m(s) Polya}
	\end{subfigure}
\hfill
	\begin{subfigure}[t]{0.3\textwidth}
		\begin{center}
		\begin{tikzpicture}[scale=0.6]
		\draw [help lines] (-0.1,-0.1) grid (5.1,5.1); 
		\draw [->] (0,-0.1) -- (0,5.2);
		\draw [->] (-0.1,0) -- (5.2,0);
		\draw [-,thick] (5.2,2) -- (0,2) node[left] {$1$};
		\draw (2,0) node[below]{$1$};
		\draw (4,0) node[below]{$2$};
		\draw (5,0) node[below]{$s$};
		\draw [thick, domain=0:5.1, samples=500] plot (\x, {exp(0.75*exp(-20*(\x/3)))+exp(-\x)+0.1*exp(\x-2)});
		\draw [-,color=black] (0.225,2) -- (0.225,0) node[below] {$\alpha$};
		\shade[ball color=black] (0.225,2) circle(0.3ex);
		\end{tikzpicture}
		\end{center}
		\caption{$m(s)$ may even\-tually increase if 
$\sup_j |T_j| > 1$ with positive probability}
		\label{subfig:m(s) increasing}
	\end{subfigure}
\caption{Possible shapes of $s \mapsto m(s)$}
\label{fig:plots of m}
\end{figure}
Throughout the paper, we assume that the number of non-zero $T_j$ is a.s.\ finite,
i.e., $N \defeq \#\{j \in \N: \norm{T_j} > 0\} < \infty$ a.s.
We suppose without loss of generality that $T_1, \ldots, T_N \not = 0$ and $T_j = 0$ for $j>N$.
For the formulation of further assumptions, we define, for $s \geq 0$,
\begin{equation}	\label{eq:m}
m(s)	~\defeq~	\E \bigg[ \sum_{j=1}^N \norm{T_j}^{s} \bigg].
\end{equation}
Notice that $m(0) = \E[N]$.
Throughout the paper the following assumptions will be in force.
\begin{align*}
\E[N] > 1.	\tag{A1}	\label{eq:A1}	\\
m(\alpha) = 1 \text{ for some } \alpha > 0.	\tag{A2}	\label{eq:A2}
\end{align*}
(Notice that we do not exclude the case that $m(s)=\infty$ for all $s \not = \alpha$.)
Then $W_1 \defeq \sum_{j=1}^N \norm{T_j}^{\alpha}$ is a nonnegative random variable with unit mean.
In our main results, we assume that
\begin{equation*}	\tag{A3}	\label{eq:A3}
m'(\alpha) \defeq \E \bigg[\sum_{j=1}^N \norm{T_j}^{\alpha} \log\norm{T_j} \! \bigg] \in (-\infty,0)	\,\text{ and }\,	\E [W_1 \log^+ W_1] < \infty.
\end{equation*}
Notice that $m'(\alpha)$ as defined in \eqref{eq:A3} is indeed the derivative of $s \mapsto m(s)$ at $\alpha$
if the latter exists. $\alpha$ can be thought of as a generalized index of stability.
Indeed, strictly $\alpha$-stable random variables solve the particular instance of \eqref{eq:FP of ST inhom}
with $C=0$ and $T_j = N^{-1/\alpha}$, $N \ge 2$ fixed.

While the assumptions imposed up to now are very mild and suffice when $\alpha \not= 1$, it might be not surprising that additional assumptions are needed 
in the case $\alpha=1$ in order to overcome severe technical obstacles. We recommend to skip the following part at first reading.

\subsubsection{Additional assumptions for the case $\alpha=1$}

We solve \eqref{eq:FP of ST inhom} in the case $\alpha=1$ under two sets of assumptions,
which cover the most relevant cases. 
Let $\ell = 1+\dim_\R \Orth = 1+ d(d-1)/2$.
Then the first set of additional assumptions for $\alpha=1$ is 
\begin{align*}
\E\bigg[\sum_{j=1}^N \norm{T_j}^{\alpha} \delta_{T_j}(\cdot) \bigg]
\text{ is spread-out (w.r.t.\ the Haar measure on $\Sim$),}	\\  
\E \bigg[\sum_{j=1}^N \norm{T_j}^{\alpha} |\log^-(\norm{T_j})|^{ \ell+\delta+1} \bigg] <  \infty 
\,\text{ and }\,
\E [h_{ 2 \ell + \delta + 1} (W_1)] < \infty   \tag{A4} \label{eq:A4}
\end{align*}
for some $\delta>0$ and $h_r(x) \defeq x (\log^+ (x))^{r} \log^+(\log^+(x))$. 
Notice that in the complex case, for the first condition in \eqref{eq:A4} to hold,
it is sufficient that, for some $j \in \N$,
the law of $T_j$ is spread-out on $\C$.
Let $\O$ denote the smallest closed subgroup of $\Orth$
which contains the random set $\{T_j/\norm{T_j}: j=1,\ldots,N\}$ with probability one. Assumption \eqref{eq:A4} will imply that $\O \supseteq \SOd$, the subgroup of $\Orth$ of matrices with determinant 1. In the second set of assumptions, we assume $\O$ to be a finite group:
\begin{align*}
\O \text{ is finite, } \E\bigg[\sum_{j=1}^N \norm{T_j}^{\alpha} \delta_{-\log (\norm{T_j})}(\cdot) \bigg]
\text{ is spread-out,}	\notag	\\  
\E \bigg[\sum_{j=1}^N \norm{T_j}^{\alpha} (\log^-(\norm{T_j}))^2\bigg] <  \infty 
\,\text{ and }\,
\E [h_3 (W_1)] < \infty.	\label{eq:A4'}	\tag{A4'}
\end{align*}
For a probability measure $\mu$ on $\Orth \times \R$, we say that $\mu$ satisfies the minorization condition \eqref{eq:minorization}
if there is a nonempty open interval $I \subseteq \R$ and $\gamma >0$ such that
\begin{equation*}	\tag{M}	\label{eq:minorization}
\mu(\ddo, \dx) ~\geq~ \gamma \1_{\SOd \times I}(o,x) \, H_{\Orth}(\ddo) \, \dx
\end{equation*}
where $H_{\Orth}$ is the normalized Haar measure on the compact group $\Orth$.
We will show in Lemma \ref{Lem:embedded ladder process} that we may assume the following stronger property instead of \eqref{eq:A4}:
\begin{equation}\label{eq:A5}\tag{A5}
\text{\eqref{eq:A4} holds and } \E\bigg[\sum_{j=1}^N \norm{T_j}^{\alpha} \delta_{(T_j/\norm{T_j}, - \log \norm{T_j})}(\cdot) \bigg]
\text{ satisfies \eqref{eq:minorization}}.
\end{equation}

\subsubsection{Weighted branching}	\label{subsubsec:Weighted branching}

The weighted branching process is a natural tool in the study of equations of the form \eqref{eq:FP of ST inhom}.
In our context, this process is defined as follows.

Let $\V = \bigcup_{n \in \N_0} \N^n$ denote the infinite Ulam-Harris tree where $\N^0 \defeq \{\varnothing\}$ is the set that contains the empty tuple only.
For $u,v \in \V$, $u=(u_1,\ldots,u_m)$, $v = (v_1,\ldots,v_n)$, we write $uv$ for $(u_1,\ldots,u_m,v_1,\ldots,v_n)$.
We say that $v$ is in generation $n$, in short: $|v|=n$, if $v \in \N^n$.
The restriction of $v$ to its first $k$ components is denoted by $v|_k$.
 
Let $(\bC,\bT) \defeq ((C(v),T(v)))_{v \in \V} = ((C(v),T_1(v),T_2(v),\ldots))_{v \in \V}$ be a family of i.i.d.\ copies of the sequence $(C,T)$.
For notational simplicity, we assume $(C(\varnothing),T(\varnothing)) = (C,T)$.
Let $\Id$ denote the $d \times d$ identity matrix and let $L(\varnothing)=\Id$.
Recursively, for $v \in \V$ and $j \in \N$, we define $L(vj) = L(v) T_j(v)$.
Hence, if $v=v_1\ldots v_n \in \N^n$, then
\begin{equation*}
L(v)	~=~	T_{v_{1}}(\varnothing) \cdot \ldots \cdot T_{v_n}(v|_{n-1}). 
\end{equation*}
Notice that the order of multiplication matters.
Each of the matrices $T_j(v)$ and $L(v)$ (if nonzero) can be written as the product of a positive scaling factor and an orthogonal matrix:
\begin{equation*}
T_j(v)	=	e^{-S_j(v)} O_j(v)	\quad	\text{ and }	\quad	L(v)	=	\norm{L(v)} O(v)
\end{equation*}
where
\begin{align*}
S_j(v)	&= -\log \norm{T_j(v)} \in \R,	\quad	& O_j(v)	& =	\norm{T_j(v)}^{-1} T_j(v) \in \Orth	\\
S(v) & = -\log \norm{L(v)} \in \R,		\quad 	& O(v)		& = \norm{L(v)}^{-1} L(v)	\in \Orth
\end{align*}
whenever $\norm{T_j(v)} > 0$ or $\norm{L(v)}>0$, respectively.
We make the convention that whenever we quantify over the $|v|=n$ as in $\sum_{|v|=n}$ or $\prod_{|v|=n}$,
this has to be understood as a quantification over the $|v|\!=\!n$ with $\norm{L(v)}>0$ only.

These definitions imply that, for $v = v_1 \ldots v_n$, when $\norm{L(v)} > 0$,
\begin{equation*}
S(v)	~=~	\sum_{k=1}^n S_{v_k}(v|_{k-1})	\quad	\text{and}	\quad	O(v)	~=~	O_{v_{1}}(\varnothing) \cdot \ldots \cdot O_{v_n}(v|_{n-1}).
\end{equation*}
Here, of course, the order of summation does not matter while the order of matrix multiplication does (in general).
We further point out that conditions \eqref{eq:A1}--\eqref{eq:A2}
imply that
\begin{equation}	\label{eq:Biggins:1998}
\lim_{n \to \infty} \ \sup_{|v|=n} \norm{L(v)}	~=~	0	\quad	\text{a.s.}
\end{equation}
(with the convention that $\sup \emptyset = 0$),
see \cite[Theorem 3]{Biggins:1998} for a reference.

For $u \in \V$ and a function $\Psi=\Psi((\bC,\bT))$ of the weighted branching process,
let $[\Psi]_u$ be defined as $\Psi(((C(uv),T(uv)))_{v \in \V})$, that is,
the same function but applied to the weighted branching process rooted at $u$.
The $[\cdot]_u$, $u \in \V$ are called {\em shift operators}.

\subsubsection{Special solutions to smoothing equations}

The solutions to \eqref{eq:FP of ST inhom} are connected with the solutions
of two related distributional identities, namely,
the tilted homogeneous equation for nonnegative random variables
\begin{equation}	\label{eq:FP tilted}
W	~\eqdist~	\sum_{j \geq 1} \norm{T_j}^{\alpha} W_j
\end{equation}
where $W_1, W_2, \ldots$ are i.i.d.\ copies of the nonnegative random variable $W$ and independent of $T$,
and the homogeneous equation
\begin{equation}	\label{eq:FP of ST hom}
X	~\eqdist~	\sum_{j \geq 1} T_j X_j
\end{equation}
where $X_1,X_2,\ldots$ are i.i.d.\ copies of $X$ and independent of $T$.
Special solutions to \eqref{eq:FP tilted} and \eqref{eq:FP of ST hom}
can be constructed using the weighted branching process.
Indeed, \eqref{eq:A2} implies that
\begin{equation}	\label{eq:Biggins' martingale}
W_n	~\defeq~	\sum_{|v|=n} \norm{L(v)}^{\alpha}	~=~	\sum_{|v|=n} e^{-\alpha S(v)},	\quad	n \in \N_0
\end{equation}
defines a nonnegative mean-one martingale. Let $W = \lim_{n \to \infty} W_n$ a.s.
It is known \cite{Lyons:1997} that \eqref{eq:A1}--\eqref{eq:A3} guarantee $\E[W]=1$.
It can be checked that
\begin{equation}	\label{eq:FP tilted a.s.}
W	~=~	\sum_{|v|=n} \norm{L(v)}^{\alpha} [W]_v	\quad	\text{a.s.}
\end{equation}
for every $n \in \N_0$. In particular, $W$ is a solution to \eqref{eq:FP tilted}.
The set of solutions to \eqref{eq:FP tilted} is $\{\mathrm{Law}(cW): c \geq 0\}$, see \cite{Alsmeyer+Biggins+Meiners:2012,Iksanov:2004}.
The description of the set of solutions to \eqref{eq:FP of ST hom} is more delicate
and most of the analysis in this paper is concerned with solving it.
One aspect of this equation is that special solutions may arise due to balancing effects
in the sum $Z_1 \defeq \sum_{j \geq 1} T_j$.
This may happen when the matrix $\E[Z_1]$ has eigenvalue $1$.
Then, for any eigenvector $w$ corresponding to the eigenvalue $1$ and with $Z_n \defeq \sum_{|v|=n} L(v)$,\label{eq:Z_n}
the sequence $(Z_nw)_{n \in \N_0}$ defines a martingale.
If it converges in probability, we denote its limit by $Z^w$. $Z^w$ is a function of $\bT$ and a solution to the identity
\begin{equation}	\label{eq:FP of ST hom a.s.}
Z	~=~	\sum_{j \geq 1} T_j [Z]_j	\quad	\text{a.s.}
\end{equation}
The solutions to \eqref{eq:FP of ST hom a.s.} are described by the following result.
Note that $Z =0$ a.s.\ always satisfies \eqref{eq:FP of ST hom a.s.}, this we call the trivial solution.

\begin{proposition}	\label{Prop:convergence Z_n w}
Assume that \eqref{eq:A1}--\eqref{eq:A3} hold and let $Z$ satisfy \eqref{eq:FP of ST hom a.s.}.
\begin{itemize}
\item[(a)]	If $0\!<\!\alpha\!<\!1$,
			then $Z=0$ a.s.
\item[(b)]	If $\alpha\!=\!1$ and \eqref{eq:A4} holds, then $Z=0$ a.s.
			If \eqref{eq:A4'} holds, then a nontrivial solution $Z$ exists
			iff $w = \E[Z]$ is an eigenvector corresponding to the eigenvalue $1$ of $\E[Z_1]$,
			and then $Z=Ww$.  
\item[(c)]	If $1\!<\!\alpha\!<\!2$,
			then a nontrivial solution $Z$  exists
			iff $w = \E[Z]$ exists and is an eigenvector corresponding to the eigenvalue $1$ of $\E[Z_1]$
			and $(Z_nw)_{n \in \N_0}$ is uniformly integrable. If these conditions hold, then $Z=Z^w$
			and the sequence $(Z_n w)_{n \in \N_0}$ is bounded in $\mathcal{L}^s$ for all $1 < s < \alpha$.
	
			If $1$ is an eigenvalue of $\E[Z_1]$ associated with the eigenvector $w$,
			then a sufficient condition for $\mathcal{L}^\beta$-boundedness (and hence uniform integrability)
			of $(Z_n w)_{n \in \N_0}$ for $\beta \in (\alpha,2]$
			is $\E[|Z_1w|^{\beta}]<\infty$ and $m(\beta)<1$.
\item[(d)]	If $\alpha \! \geq \! 2$, then a nontrivial solution $Z$  exists
			iff there is a deterministic $w \not = 0$ such that $w$ is an eigenvector corresponding to the eigenvalue $1$ of the random matrix $Z_1$ a.s.,
			and then $Z \!=\! w$ a.s.
\end{itemize}
\end{proposition}

Besides $W$ and $Z$, the class of $(U,\alpha)$-stable L\'evy processes,
with $U$ being a closed subgroup of the group of similarity matrices of $\R^d$,
will be relevant for the description of solutions, and is introduced next.
This notion is a particular case of $(U,\bar \alpha)$-stability,
a concept introduced in \cite[p.\,338]{Buraczewski+Damek+Guivarch:2010}.

Let $(Y_t)_{t \geq 0}$ be a L\'evy process on $\R^d$.
We say that $(Y_t)_{t \geq 0}$ is \emph{$(U,\alpha)$-stable}
if there exists a mapping $b: U \to \R^d$ such that, for all $u \in U$, $t >0$,
\begin{equation}	\label{eq:(U,alpha)-stable process}
u Y_t ~\eqdist~ Y_{\norm{u}^\alpha t} + tb(u).
\end{equation} 
This implies in particular that each $Y_t$ is operator semistable \cite{Hazod+Siebert:2001},
see Section \ref{subsect:operator semistable} for more details.
Property \eqref{eq:(U,alpha)-stable process} is equivalent to
\begin{equation}	\label{eq:(u,alpha)-stable chf}
\Psi(u^\transp x) =  \norm{u}^\alpha \Psi(x) + \imag \scalar{x,b(u)}
\end{equation}
for all $u \in U$ for the characteristic exponent $\Psi$ of $Y_1$.
$(Y_t)_{t \geq 0}$ is said to be \emph{strictly $(U,\alpha)$-stable} if $b(u) \equiv 0$.
We say that a probability measure $P$ on $\R^d$ is (strictly) $(U,\alpha)$-stable
if it is the law of a (strictly) $(U,\alpha)$-stable L\'evy process at time $1$.

The general formula of the L\'evy measure of a $(U, \alpha)$-stable law is given in Proposition \ref{Prop:(U,alpha)-stable laws} below. 
Let us just point out that there are no $(U,\alpha)$-stable L\'evy processes for $\alpha > 2$ unless $U \subseteq \Orth$ and that the larger the group $U$, the smaller the class of $(U, \alpha)$-stable L\'evy processes.
For example, if $U \supseteq \R_> \times \SOd$, then $\Psi(x)=- c \norm{x}^{\alpha}$
for some $c \geq 0$, i.e.\;$(Y_t)_{t \geq 0}$ is strictly $\alpha$-stable and rotation-invariant.
It will be shown (in Remark \ref{Rem:U under A4} below)  that assumption \eqref{eq:A4} implies $U \supseteq \R_> \times \SOd$.
Below, we consider $(\U,\alpha)$-stable L\'evy processes, where $\U$ denotes the smallest closed subgroup of the similarity group that covers $\{T_j: j=1,\ldots,N\}$ with probability one.

There are further technicalities to deal with before the main result can be formulated
but the complex and homogeneous case,
the most important special case in view of applications,
is now given as an illustration.

\subsubsection{Solutions to complex smoothing equations}	\label{subsubsec:solutions to complex smoothing equations}

In the complex case, each $Z_n$ is a complex random variable,
and the condition that $\E[Z_1]$ has eigenvalue $1$ is equivalent to $\E[Z_1]=1$,
which in turn is equivalent to $(Z_n)_{n \in \N_0}$ being a (complex) martingale.
We write $Z$ for the a.s.\ limit of this martingale if it exists, and define $Z=0$, otherwise. Note that $\U$ is now the smallest closed subgroup of the multiplicative group $\C^*$ that covers $\{T_j \, : \, j=1, \dots, N\}$ with probability one.


\begin{theorem}	\label{Thm:solutions to complex smoothing equations}
Consider \eqref{eq:FP of ST hom} in the complex case and assume that \eqref{eq:A1}--\eqref{eq:A3} hold.
Additionally suppose that \eqref{eq:A4} or \eqref{eq:A4'} holds if $\alpha=1$.

Then a probability distribution on $\C$ is a solution to \eqref{eq:FP of ST hom}
if and only if it is the law of a random variable of the form
\begin{equation}	\label{eq:solutions to complex smoothing equations}
Y_{W} + aZ
\end{equation}
for some $a \in \C$ and a complex strictly $(\U,\alpha)$-stable L\'evy process $(Y)_{t \geq 0}$ independent of $(W,Z)$.
\end{theorem}

Theorem \ref{Thm:solutions to complex smoothing equations} is a special case of the more general
Theorem \ref{Thm:solutions to multivariate smoothing equations}.
Therefore, we do not give a separate proof of Theorem \ref{Thm:solutions to complex smoothing equations}.

We have tried to present the result of Theorem \ref{Thm:solutions to complex smoothing equations}
as concise as possible. This should not hide the fact that it hosts a multitude of different cases.
On the one hand, there are several qualitatively different possibilities for the group $\U$
two of which, coming from examples discussed in Section \ref{subsec:Applications complex},
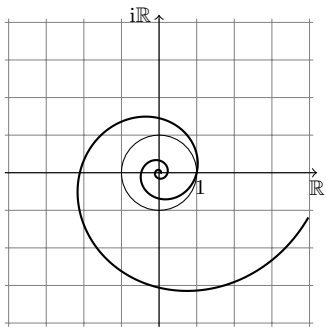
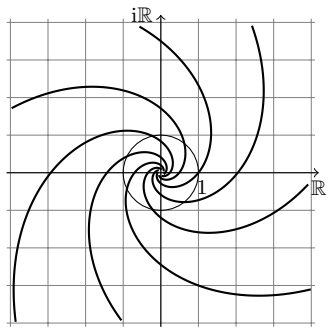
\begin{figure}[h]
\hfill
	\begin{subfigure}[h]{0.45\textwidth}	
		\begin{tikzpicture}[scale=0.5]
		\draw [help lines] (-4.1,-4.1) grid (4.1,4.1);
		\draw (1.1,0) node[below]{$1$};
		\draw [->] (-4.1,0) -- (4.2,0) node[below] {\footnotesize $\R$}; 
		\draw [->] (0,-4.1) -- (0,4.2) node[left] {$\imag\R$}; ;
		\draw (0,0) circle (1);
\draw [thick, domain=-6:2.75, samples=1000] plot ({exp(0.517*\x)*cos(deg(2.179*\x))},
{exp(0.517*\x)*sin(deg(2.179*\x))});
		\end{tikzpicture}
		\caption{The group $\U$ associated with Eq.\ \eqref{eq:Chauvin+Liu+Pouyanne}
		appearing in the context of $b$-ary search trees (here, $b=27$)}
		\label{subfig:snail}
	\end{subfigure}
\hfill
	\begin{subfigure}[h]{0.45\textwidth}
		\begin{tikzpicture}[scale=0.5]
		\draw [help lines] (-4.1,-4.1) grid (4.1,4.1);
		\draw (1.1,0) node[below]{$1$};
		\draw [->] (-4.1,0) -- (4.2,0) node[below] {$\R$}; 
		\draw [->] (0,-4.1) -- (0,4.2) node[left] {$\imag\R$}; ;
		\draw (0,0) circle (1);
\draw [thick, domain=-4:2.2, samples=1000] plot ({exp(cos(deg(2*pi/7))*\x)*cos(deg(sin(deg(2*pi/7))*\x))},
{exp(cos(deg(2*pi/7))*\x)*sin(deg(sin(deg(2*pi/7))*\x))});
\draw [thick, domain=-4:2.35, samples=1000] plot ({exp(cos(deg(2*pi/7))*\x)*cos(deg(2*pi/7+sin(deg(2*pi/7))*\x))},
{exp(cos(deg(2*pi/7))*\x)*sin(deg(2*pi/7+sin(deg(2*pi/7))*\x))});
\draw [thick, domain=-4:2.75, samples=1000] plot ({exp(cos(deg(2*pi/7))*\x)*cos(deg(2*2*pi/7+sin(deg(2*pi/7))*\x))},
{exp(cos(deg(2*pi/7))*\x)*sin(deg(2*2*pi/7+sin(deg(2*pi/7))*\x))});
\draw [thick, domain=-4:2.25, samples=1000] plot ({exp(cos(deg(2*pi/7))*\x)*cos(deg(3*2*pi/7+sin(deg(2*pi/7))*\x))},
{exp(cos(deg(2*pi/7))*\x)*sin(deg(3*2*pi/7+sin(deg(2*pi/7))*\x))});
\draw [thick, domain=-4:2.6, samples=1000] plot ({exp(cos(deg(2*pi/7))*\x)*cos(deg(4*2*pi/7+sin(deg(2*pi/7))*\x))},
{exp(cos(deg(2*pi/7))*\x)*sin(deg(4*2*pi/7+sin(deg(2*pi/7))*\x))});
\draw [thick, domain=-4:2.2, samples=1000] plot ({exp(cos(deg(2*pi/7))*\x)*cos(deg(5*2*pi/7+sin(deg(2*pi/7))*\x))},
{exp(cos(deg(2*pi/7))*\x)*sin(deg(5*2*pi/7+sin(deg(2*pi/7))*\x))});
\draw [thick, domain=-4:2.45, samples=1000] plot ({exp(cos(deg(2*pi/7))*\x)*cos(deg(6*2*pi/7+sin(deg(2*pi/7))*\x))},
{exp(cos(deg(2*pi/7))*\x)*sin(deg(6*2*pi/7+sin(deg(2*pi/7))*\x))});
		\end{tikzpicture}
		\caption{The group $\U$ associated with Eq.\ \eqref{eq:cyclic Polya}
		appearing in the context of cyclic P\'olya urns (with $b=7$)}
		\label{subfig:snails}
	\end{subfigure}
\hfill
\caption{The group $\U$ in examples}
\label{fig:snail and snails}
\end{figure}

\noindent
The reader should notice that $\U$ may also be a family of discrete points aligned on a ``snail graph''
or consist of a finite number of copies of $\Rp \defeq (0,\infty)$ obtained by multiplication with roots of unity as depicted in Figure \ref{fig:examples}.
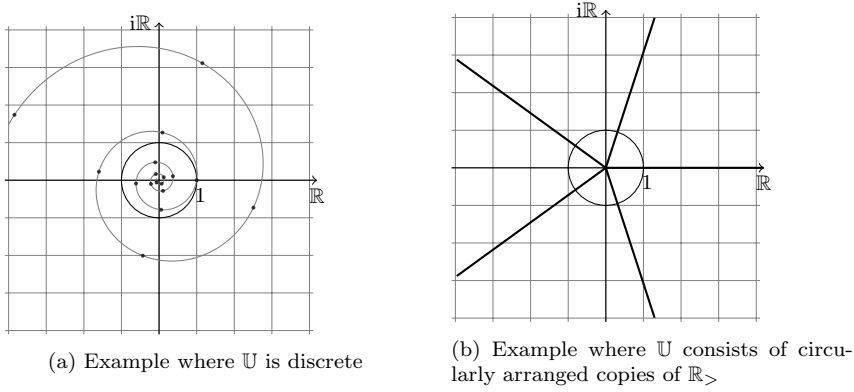
\begin{figure}[h] 
\hfill
	\begin{subfigure}[h]{0.45\textwidth}
		\begin{tikzpicture}[scale=0.5]
		\draw [help lines] (-4.1,-4.1) grid (4.1,4.1);
		\draw (1.1,0) node[below]{$1$};
		\draw [->] (-4.1,0) -- (4.2,0) node[below] {$\R$}; 
		\draw [->] (0,-4.1) -- (0,4.2) node[left] {$\imag\R$}; ;
		\draw (0,0) circle (1);
		\draw [color=black!50, thin, domain=-11:6.05, samples=500] plot ({exp(0.24*\x)*cos(deg(1.5*\x))},
		{exp(0.24*\x)*sin(deg(1.5*\x))});
		\foreach \exp in{-10,-9,-8,-7,-6,-5,-4,-3,-2,-1,0,1,2,3,4,5,6}
		\fill[black,opacity=.75] ({exp(0.24*\exp)*cos(deg(1.5*\exp))},
		{exp(0.24*\exp)*sin(deg(1.5*\exp))}) circle (1.5pt);
		\end{tikzpicture}
		\caption{Example where $\U$ is discrete}
	\end{subfigure}
\hfill
		\begin{subfigure}[h]{0.45\textwidth}
		\begin{tikzpicture}[scale=0.5]
		\draw [help lines] (-4.1,-4.1) grid (4.1,4.1); 
		\draw (1.1,0) node[below]{$1$};
		\draw [->] (-4.1,0) -- (4.2,0) node[below] {$\R$}; 
		\draw [->] (0,-4.1) -- (0,4.2) node[left] {$\imag\R$}; ;
		\draw (0,0) circle (1);
		\draw [thick, domain=-0:4.2, samples=1000] plot ({cos(deg(0*pi/5))*\x},
{sin(deg(0*pi/5))*\x});
\draw [thick, domain=-0:4.2, samples=1000] plot ({cos(deg(2*pi/5))*\x},
{sin(deg(2*pi/5))*\x});
\draw [thick, domain=-0:4.9, samples=1000] plot ({cos(deg(4*pi/5))*\x},
{sin(deg(4*pi/5))*\x});
\draw [thick, domain=-0:4.9, samples=1000] plot ({cos(deg(6*pi/5))*\x},
{sin(deg(6*pi/5))*\x});
\draw [thick, domain=-0:4.2, samples=1000] plot ({cos(deg(8*pi/5))*\x},
{sin(deg(8*pi/5))*\x});
		\end{tikzpicture}
		\caption{Example where $\U$ consists of circularly arranged copies of $\Rp$}
	\end{subfigure}
\hfill
\caption{Two possible shapes for $\U$}
\label{fig:examples}
\end{figure}
On the other hand, different values of $\alpha$ give rise to qualitatively different regimes.
A brief discussion of the implications of Proposition \ref{Prop:convergence Z_n w} and the structure of $(\U,\alpha)$-stable L\'evy processes (described in Proposition \ref{Prop:(U,alpha)-stable laws} below) is given in the next remark.

\begin{remark}	\label{Rem:solutions to complex smoothing equations}	
Before we discuss the different cases of Theorem \ref{Thm:solutions to complex smoothing equations},
we point out that \eqref{eq:A2} implies that $\U \not \subseteq \S$.
\begin{itemize}
	\item[(i)]
		$\alpha \! < \! 1$: Then $Z = 0$ a.s., i.e., the second summand in \eqref{eq:solutions to complex smoothing equations} vanishes.
	\item[(ii)]
		$\alpha \! = \! 1$: $Z$ vanishes unless $\E[Z_1]=1$.
		$\E[Z_1]=1$ and $m(1)\!=\!1$ imply $T_j \in [0,\infty)$ a.s.\ for all $j \in \N$.
		In this case, $Z=W$ and $\U \subseteq \Rp$. In fact, Assumptions \eqref{eq:A4} and \eqref{eq:A4'}, resp., imply that $\U=\Rp$.
	\item[(iii)]
		$1\!<\!\alpha\!<\!2$: $Z$ vanishes unless $\E[Z_1]=1$.
		In most applications, $\E[Z_1] = 1$ will hold and the full spectrum of solutions given by \eqref{eq:solutions to complex smoothing equations}
		will arise.
	\item[(iv)]
		$\alpha\!=\!2$: In this case, $(Y_t)_{t \geq 0}$ is a centered Gaussian process (or $0$).
		Moreover, $(\U,2)$-stability necessitates 
		that real and imaginary part of $(Y_t)_{t \geq 0}$ are i.i.d.\ centered one-dimensional Brownian motions (or $0$)
		whenever $\U \setminus \R \not = \emptyset$.
		Further, $Z=1$ a.s.\ iff $Z_1 = 1$ a.s.\ and $Z=0$, otherwise.
	\item[(v)]
		$\alpha\!>\!2$: Here, $Y_t = 0$ a.s.\ for all $t$ and 
		$Z=1$ a.s.\ iff $Z_1 = 1$ a.s.\ and $Z=0$, otherwise.
		Hence, the set of solutions to \eqref{eq:FP of ST hom} is either $\{\delta_a: a \in \C\}$ or $\{\delta_0\}$, respectively,
		where here and throughout the paper, $\delta_a$ denotes the Dirac distribution with a point at $a$.		
\end{itemize}
\end{remark}

\begin{remark}	\label{Rem:tails Y_W}
In many cases, the different solutions to \eqref{eq:solutions to complex smoothing equations}
can be distinguished via their tail behavior.
We discuss here in details the most relevant case $1 < \alpha < 2$.
In applications, typically $\E[|Z|^\beta]<\infty$ for some $\beta > \alpha$
(see the sufficient condition of Proposition \ref{Prop:convergence Z_n w}(c)),
whereas $Y_W$ for a non-trivial $(\U,\alpha)$-stable L\'evy process $(Y_t)_{t \geq 0}$ exhibits heavier tails.
Indeed, the L\'evy measure $\bar\nu$ of $(Y_t)_{t \geq 0}$
satisfies $h_1 r^{-\alpha} \leq \bar \nu (\{|x| \geq r\}) \leq h_2 r^{-\alpha}$
for $0 < h_1 \leq h_2 < \infty$,
which can be derived directly from the $(\U,\alpha)$-stability
and is also implicit in our proofs.
According to \cite[Corollary 25.8]{Sato:1999},
this implies that, for every $t>0$, $Y_t$ has all absolute moments of order $<\alpha$ finite, while $\E[|Y_t|^\alpha]=\infty$.
Now if $\G \defeq \{\norm{u}: u \in \U\} = \Rp$, then, using that $\E[W]=1<\infty$
and the independence of $W$ and $(Y_t)_{t \geq 0}$,
\begin{equation*}
\E[|Y_W|^p] = \E[(W^{1/\alpha} |Y_1|)^p] = \E[W^{p/\alpha}] \E[|Y_1|^p] < \infty
\end{equation*}
iff $p < \alpha$.
In the case where $\G = r^\Z$ for some $r > 1$, one can argue similarly
using the fact that $\E[\sup_{1 \leq s \leq r}|Y_s|^p]<\infty$ iff $\E[|Y_1|^p]<\infty$,
see \cite[Theorem 25.18]{Sato:1999}.

More results on the tail behavior of $Z$ (under stronger conditions than imposed here)
can be found in \cite{Buraczweski+Damek+Mentemeier:2013}. 
\end{remark}

\subsubsection{Solutions to multivariate smoothing equations}	\label{subsubsec:solutions to multivariate smoothing equations}

We continue with the description of the solutions to \eqref{eq:FP of ST inhom} in the general situation.
A special solution can be constructed in terms of the weighted branching process. Define
\begin{equation}	\label{eq:W_n^*}
W_n^*	~\defeq~	\sum_{\abs{v}<n} L(v) C(v),	\quad	n \in \N_0
\end{equation}
and let $W^*$ denote the limit in probability as $n \to \infty$ of $W_n^*$ provided the limit exists.
If it exists, $W^*$ satisfies
\begin{equation}	\label{eq:FP of ST inhom a.s.}
W^*	~=~	\sum_{j \geq 1} T_j [W^*]_j + C	\quad	\text{a.s.}
\end{equation}
and thus constitutes a solution to \eqref{eq:FP of ST inhom}.
Each of the following is a sufficient condition for the convergence in probability of $W_n^*$
taken from \cite[Proposition 2.1]{Iksanov+Meiners:2015a},
which remains valid in the present context:
\begin{itemize}
	\item[(S1)]
		\eqref{eq:A1} and \eqref{eq:A2} hold and there is $\beta \in (0,1]$ with $m(\beta)<1$ and $\E[|C|^{\beta}]<\infty$.
	\item[(S2)]
		For some $\beta \geq 1$, $\sup_{n \in \N_0} \E[|W_n^*|^{\beta}] < \infty$ and
		either $T_j \geq 0$ for all $j \in \N$ or $\E[C] = 0$.
\end{itemize}
Notice that if $C = 0$ a.s., then $W_n^*$ converges trivially to $W^* = 0$ a.s.

We now state the main result for the general case.

\begin{theorem}	\label{Thm:solutions to multivariate smoothing equations}
Assume that \eqref{eq:A1}--\eqref{eq:A3} hold and that $W_n^* \to W^*$ in probability as $n \to \infty$.
If $\alpha=1$, assume that \eqref{eq:A4} or \eqref{eq:A4'} holds in addition.

Then a probability distribution  on $\R^d$ is a solution to \eqref{eq:FP of ST inhom}
if and only if it is the law of a random variable of the form
\begin{equation}	\label{eq:solutions to multivariate smoothing equations}
W^* + Y_W + Z
\end{equation}
where $Z$ is a solution of \eqref{eq:FP of ST hom a.s.} and $(Y_t)_{t \geq 0}$ is a strictly $(\U,\alpha)$-stable L\'evy process on $\R^d$ independent of $(W^*,W,Z)$.
\end{theorem}

Denote by $E_{1} \subseteq \R^d$ the eigenspace corresponding to the eigenvalue $1$ of the matrix $\E[Z_1]$
and let $E_1 = \{0\}$ if $1$ is not an eigenvalue of $\E[Z_1]$.
Then Proposition \ref{Prop:convergence Z_n w} yields that
the set of solutions to \eqref{eq:FP of ST hom a.s.} is either empty or parametrized by $E_1$.
Besides, 
the solutions are parametrized by the different strictly $(\U,\alpha)$-stable distributions, which always include $Y_t \equiv 0$, i.e., $W^*+Z$ is also a fixed point.
In previous works, the solutions to smoothing equations on the nonnegative halfline \cite{Alsmeyer+Meiners:2012},
on $\R$ \cite{Alsmeyer+Meiners:2013} and on $\R^d$ in the case where $\U \subseteq \R^*$ \cite{Iksanov+Meiners:2015a}
have been represented in the form
\begin{equation}	\label{eq:solutions to multivariate smoothing equations classic}
W^* + W^{1/\alpha} Y + Z
\end{equation}
with $(W^*,W,Z)$ defined as here and $Y$ denoting an independent strictly $\alpha$-stable random variable
(where $Y=0$ is allowed). The same representation is possible here when $\Rp \times \{\Id\} \subseteq \U$ since then
$(Y_t)_{t \geq 0}$ is strictly $\alpha$-stable (or zero) and thus $Y_W$ has the same law as $W^{1/\alpha} Y_1$.
In general, solutions to \eqref{eq:FP of ST inhom} do not possess a representation of the form \eqref{eq:solutions to multivariate smoothing equations classic}
as $\U$ need not contain $\Rp \times \{\Id\}$, see e.g.\ the examples depicted in Figure \ref{fig:snail and snails}.

Observe that $W^*$ and $Z$ are measurable functions of $(\bC,\bT)$ and satisfy a.s.\ versions of the inhomogeneous or homogeneous fixed point equation,
i.e., \eqref{eq:FP of ST inhom a.s.} and \eqref{eq:FP of ST hom a.s.}, respectively.
Such fixed points are called {\em endogenous}, a notion coined by Aldous and Bandyopadhyay \cite{AB2005},
see \cite[Section 6]{Alsmeyer+Biggins+Meiners:2012} and \cite[Section 3.5]{Iksanov+Meiners:2015a} for further information.
The fixed point $Y_W$ is not endogenous, for the process $(Y_t)_{t \geq 0}$ introduces additional randomness.
Note, however, that $W$ is an endogenous fixed point of the one-dimensional smoothing transform with scalar weights
$\norm{T_1}^{\alpha}, \dots, \norm{T_N}^{\alpha}$, see \eqref{eq:FP tilted a.s.}.
The tail behavior of $W^*$ is investigated in \cite{Buraczweski+Damek+Mentemeier:2013}, and parallels that of $Z$.

\subsection{The class of strictly $(U,\alpha)$-stable L\'evy processes}	\label{subsec:strictly (U,alpha)-stable Levy processes}

To complement our main result, we determine the form of the characteristic exponent
of strictly $(U,\alpha)$-stable L\'evy processes in this section.
The analysis of $(U,\alpha)$-stable L\'evy processes for $U \subseteq \Orth$ is of no relevance for this paper
and simpler than in the case $U \not \subseteq \Orth$ and therefore omitted here.

Below, denote by $\nu^\alpha$, $0<\alpha <2$, a L\'evy measure satisfying
\begin{equation} \label{eq:nualpha (U,alpha)-invariant}
\nu^\alpha(uB) = \norm{u}^{-\alpha} \nu^\alpha(B) 
\end{equation}
for all $u \in U$ and all Borel sets $B \subseteq \R^d \setminus \{0\}$.
We call such a L\'evy measure $(U, \alpha)$-invariant as well.
The structure of such measures is described in Section \ref{subsec:U-invariant Levy measures}.
Given $\nu^{\alpha}$, define the functions
\begin{align}
\label{eq:Definition eta1}
\eta_1^\alpha(x) ~&\defeq~ \frac{1}{|x|^\alpha} \int \big(1 - \cos (\scalar{x,y}) \big) \nu^\alpha(\dy), \\
\eta_2^\alpha(x) ~&\defeq~ \frac{1}{|x|^\alpha} \int \big(\sin (\scalar{x,y}) - \1_{\{\alpha >1\}}\scalar{x,y} \big) \nu^\alpha(\dy).
\label{eq:Definition eta2}
\end{align}
It is proved in Lemma \ref{Lem:I(tx) evaluated} that $\eta_i^\alpha$ are bounded functions,
satisfying $\eta_i(u^\transp x) = \eta_i(x)$ for all $u \in U$, $x \in \R^d \setminus\{0\}$,  $i=1,2$. 

Now consider the case where the image of $U$ under the homomorphism $U \ni u \mapsto \norm{u}$ is $\R_>$.
Notice that the groups depicted in Figure \ref{fig:snail and snails} are of this type.
It is proved in \cite[Proposition C.1 and Theorem D.13]{Buraczewski+al:2009},
see also Proposition \ref{Prop:Buraczewski et al}, that then $U = \{ t^Q \, : \, t \in \Rp\} \times C$
for a suitable $d \times d$-matrix $Q$ and $C \defeq U \cap \Orth$.
Here, $t^Q \defeq e^{(\ln t) Q}$ and we scale $Q$ in such a way that $\norm{t^Q}=t$.
Observe that \eqref{eq:(U,alpha)-stable process} then implies that any $(U,\alpha)$-stable law is also operator stable
with exponent $\frac{1}{\alpha}Q$, see Section \ref{subsect:operator semistable} for details.

Further, let $\rho$ be a measure on $\Sd \defeq \{x \in \R^d \, : \, \abs{x}=1\}$
satisfying $\rho(oB)=\rho(B)$ for all $o \in C$ and Borel sets $B \subseteq \Sd$.
We call such a measure $C$-invariant and define
\begin{align}
\label{eq:Definition eta}
\eta^1 (x)~\defeq~ \int_{\Rp} \int_{\Sd} \bigg( e^{\imag\scalar{x,t^{Q}y}} - 1 - \frac{\imag \scalar{x,t^{Q}y}}{1+t^2} \bigg) t^{-2} \, \rho(\dy) \, \dt, \\
(Q-\Id) \, \gamma^1 ~\defeq~  \int_{\Rp} \int_{\Sd} (t^Q y) \, \frac{2t \scalar{Qy,y}}{(1+ t^2)^2} \, \rho(\dy) \, \dt,
\label{eq:Definition gamma}
\end{align}
given the right-hand side in \eqref{eq:Definition gamma} is in the range of $(Q-\Id)$.

Below, we write $O\defeq\{u/\norm{u} \, : \, u \in U\}$ for the projection of $U$ onto $\Orth$.

\begin{proposition}	\label{Prop:(U,alpha)-stable laws}	
Let $U \not \subseteq \Orth$ be a closed subgroup of the similarity group and $\alpha > 0$.
Then for a $d$-dimensional L\'evy process $(Y_t)_{t \geq 0}$ with characteristic exponent $\Psi$, the following assertions hold:
\begin{itemize}
	\item[(i)]	Let $1 \not = \alpha \in (0,2)$.
				$(Y_t)_{t \geq 0}$ is strictly $(U,\alpha)$-stable iff $\Psi$ is of the form
				\begin{equation}\label{eq:Psi deterministic alpha not 1}
				\Psi(x) ~=~- |x|^\alpha \eta_1^\alpha(x) + \imag |x|^\alpha \eta_2^\alpha(x)
				\end{equation}
				for a $(U, \alpha)$-invariant L\'evy measure $\nu^\alpha$.
	\item[(ii)]	Let  $U = \{ t^{Q} \, : \, t \in \Rp\} \times C$.  $(Y_t)_{t \geq 0}$ is strictly $(U,1)$-stable iff $\Psi$ is of the form
				\begin{equation}\label{eq:Psi deterministic alpha is 1}
				\Psi(x) ~=~\eta^1(x) + \imag\scalar{\gamma^1+z,x} 
				\end{equation}
				for a $C$-invariant measure $\rho$ on $\Sd$, satisfying 
				$ \int \scalar{x,s} \rho(\ds) =0$ for all $x$ with $Q^\transp x=x$
				(this also guarantees the existence of $\gamma^1$), and any vector $z$ satisfying $u z = \norm{u}z$ for all $u \in U$.
	\item[(iii)]	$(Y_t)_{t \geq 0}$ is strictly $(U,2)$-stable iff $\Psi$ is of the form
				\begin{equation}\label{eq:Psi deterministic alpha is 2}
				\Psi(x)	~=~	- x^{\transp} \Sigma x / 2,	\quad	x \in \R^d
				\end{equation}
				for a positive semi-definite symmetric $d \times d$ matrix $\Sigma$ satisfying
				$o \Sigma o^\transp = \Sigma$ for all $o \in O$.
	\item[(iv)]	Let $\alpha > 2$.
				$(Y_t)_{t \geq 0}$ is strictly $(U,\alpha)$-stable iff $\Psi(x)=0$ for all $x \in \R^d$, equivalently,
				$Y_t=0$ a.s.\ for all $t \geq 0$.
\end{itemize} 
\end{proposition}

Notice that for $\alpha=1$, the proposition excludes the case
where the image of $U$ under the homomorphism $U \ni u \mapsto \norm{u}$
is a discrete subgroup of $\Rp$.
This is because \eqref{eq:A4} and \eqref{eq:A4'}, assumed in the case $\alpha=1$,
imply that $\G=\Rp$ and hence the discrete case is of no relevance here. 

A description of the matrices $\Sigma$ satisfying $o \Sigma o^\transp = \Sigma$ for all $o \in O$
is given in Proposition \ref{Prop:oSigmao^t}. In particular, if there is no proper subspace $V \subseteq \R^d$,
satisfying $oV = V$ for all $o \in O$, then $\Sigma$ is a scalar multiple of the identity matrix.

\subsection{Further organization of the paper}
In Section \ref{sec:examples}, we apply our main results to the examples mentioned in the introduction as well as to further important models.
The rest of the paper is then devoted to the proofs of the main results,
with the major, probabilistic part of the proof being given in Section \ref{sect:proof 1},
at the beginning of which we will also introduce further notation and concepts relevant for the proofs.
More algebraic considerations, concerned with the structure of $(U,\alpha)$-stable L\'evy processes,
are contained in Section \ref{sec:invariant matrices and measures}.
The appendix contains a Choquet-Deny lemma for functions on $U$
and a rate-of-convergence result for Markov renewal processes;
the latter result will be needed only in the case $\alpha=1$.

\section{Applications of the main results}	\label{sec:examples}	
In this section, we discuss the examples from the introduction, as well as further applications of our main results to the study of Biggins' martingale with complex parameter or Gaussian multiplicative chaos.

\subsection{Applications of Theorem \ref{Thm:solutions to complex smoothing equations}}	\label{subsec:Applications complex}

\label{subsubsec:b-ary search trees} \subsubsection{$b$-ary search trees}

$b$-ary search trees are $b$-ary trees
which are basic data structures in computer science used in searching and sorting,
see \cite{Mahmoud:1992} for the definition and background information.
Each node of a $b$-ary search tree can store up to $b-1$ elements from a set of distinct real numbers
$x_1,\ldots,x_n$, called the set of keys.
For the problems considered, it constitutes no loss of generality to assume $\{x_1,\ldots,x_n\} = \{1,\ldots,n\}$.
Denote by $Y_n$ the space requirement of a $b$-ary search tree under the random permutation model,
i.e., $(x_1,\ldots,x_n)$ is a uniform permutation of the set $\{1,\ldots,n\}$
and $Y_n$ is the number of nodes in the resulting $b$-ary search tree.

Let $b \geq 4$. The asymptotic behavior of $Y_n$ is coded in the equation
\begin{equation}	\label{eq:chi}
\chi(z)	~\defeq~	\prod_{j=1}^{b-1} (z+j) - b!	~=~	0,	\quad	z \in \C.
\end{equation}
The root with largest absolute value is $\lambda_1=1$.
Let $\lambda_2$ denote the root with second-largest real part and $\Imag(\lambda_2)>0$.
Then $\Real(\lambda_2) \in (0,1)$.
If $\Real(\lambda_2) \leq 1/2$, equivalently, $b \leq 26$, then after centering and norming,
$Y_n$ is asymptotically normal, see \cite{Lew+Mahmoud:1994}.
By using martingale methods, Chauvin and Pouyanne \cite{Chauvin+Pouyanne:2004} have shown
that if $\Real(\lambda_2) > 1/2$, equivalently, $b \geq 27$, then
\begin{equation}	\label{eq:limit of b-ary search trees embedding}
Y_n	~=~	\mathrm{const} \cdot n + 2 \Real(n^{\lambda_2} X) + o(n^{\Real(\lambda_2)})
\end{equation}
where $o(n^{\Real(\lambda_2)})$ is a term that, after dividing by $n^{\Real(\lambda_2)}$, tends to $0$ a.s.\ and in $\L^2$,
and $X$ is a complex-valued random variable.
Since $n^{\lambda_2}$ is complex, $n^{-\Real(\lambda_2)}(Y_n- \mathrm{const}\cdot n)$
does not converge in distribution, but behaves like $2 \Real(n^{\mathrm{Im}(\lambda_2)} X)$. 

Fill and Kapur \cite{Fill+Kapur:2004} showed that $X$ solves the smoothing equation
\begin{equation}	\label{eq:Fill+Kapur}
X	~\eqdist~	\sum_{j=1}^b V_j^{\lambda_2} X_j
\end{equation}
where $V_1,\ldots,V_b$ are the spacings of $b-1$ i.i.d.\ random variables uniformly distributed over $(0,1)$,
$U_1,\ldots,U_{b-1}$, say.
In other words, $V_j = U_{(j)}-U_{(j-1)}$ for $j=1,\ldots,b$ with $U_{(0)}=0$,
$U_{(b)}=1$ and $(U_{(1)},\ldots,U_{(b-1)})$ being the order statistics of $(U_1,\ldots,U_{b-1})$.
What is more, Fill and Kapur showed that (the law of) $X$ is the unique solution to \eqref{eq:Fill+Kapur}
subject to the additional constraints $\E[X]=\mu$ and $\E[|X|^2]<\infty$ where $\mu \not = 0$ is a given complex constant.

Later, using an embedding into continuous-time multi-type Markov branching processes,
Chauvin et al.\;\cite{Chauvin+Liu+Pouyanne:2014} established a connection between the law of $X$
and the complex smoothing equation
\begin{equation}	\label{eq:Chauvin+Liu+Pouyanne}
X	~\eqdist~	e^{- \lambda_2 T}(X_1+\ldots+X_b)
\end{equation}
where $T$ has the same distribution as the sum $\tau_1+\ldots+\tau_{b-1}$
of independent random variables with $\tau_j$ being exponentially distributed with parameter $j$, $j=1,\ldots,b-1$.
Again, the connection concerns the solution to \eqref{eq:Chauvin+Liu+Pouyanne} with fixed expectation $\E[X]=\mu \not = 0$
and finite second moment $\E[|X|^2]<\infty$.

\smallskip
Here we will consider Eq.\ \eqref{eq:Fill+Kapur} only;
the study of \eqref{eq:Chauvin+Liu+Pouyanne} bears a striking similarity.
One can check that $V_1,\ldots,V_b$ are identically distributed with Lebesgue density $(b-1)(1-x)^{b-2}\1_{(0,1)}(x)$.
Hence, for $z \in \C$,
\begin{align}	\label{eq:bE[V_1^z]}
\E[V_1^z+\ldots+V_b^z]
~&=~ b(b-1) \int_0^1 x^z(1-x)^{b-2} \, \dx
~=~ \frac{b! \, \Gamma(z+1)}{\Gamma(z+b)}	\\
~&=~ \frac{b!}{(z+1)\cdot \ldots \cdot (z+b-1)}		\notag
\end{align}
where $\Gamma$ denotes Euler's gamma function.
We conclude that in the present context the function $s \mapsto m(s)$ defined in \eqref{eq:m} takes the form
\begin{equation}	\label{eq:m for b-ary search trees}
m(s)	~=~	b! \prod_{j=1}^{b-1} \frac{1}{\Real(\lambda_2) s + j},
\end{equation}
Figure \ref{subfig:m(s) b-ary} shows the graph of $s \mapsto m(s)$.
Notice that $m(\alpha)=1$ only for $\alpha = 1/\Real(\lambda_2)$.
In particular, the phase transition at $b=26$ is visible here
since $\alpha \geq 2$ for $b \leq 26$ and $\alpha \in (1,2)$ for $b \geq 27$,
and the transition between normal and stable behavior occurs at $\alpha=2$.
Further, from \eqref{eq:bE[V_1^z]}
for $z=\lambda_2$,
we conclude that $\E[Z_1]=1$ for
$Z_1=V_1^{\lambda_2}+\ldots+V_b^{\lambda_2}$.
It follows from \eqref{eq:m for b-ary search trees} that $m(s) < 1$ for all $s > \alpha$
and it can be checked that $\E[|Z_1|^2] < \infty$.
Thus, the sufficient condition of Proposition \ref{Prop:convergence Z_n w}(c) applies
and $Z_n \to Z$ a.s.\ and in $\L^2$ for a complex random variable $Z$ with $\Prob(Z \not = 0) > 0$.

By Theorem \ref{Thm:solutions to complex smoothing equations}, the set of solutions to \eqref{eq:Fill+Kapur}
is given by all laws of random variables of the form
\begin{equation}	\label{eq:solutions to Fill+Kapur equation}
Y_{W} + a Z
\end{equation}
where $(Y_{t})_{t \geq 0}$ is a strictly $(\U,\alpha)$-stable L\'evy process independent of $(W,Z)$.
Here, $\U=\{e^{\lambda_2 t} \, : \, t \in \R\}$.
For $b=27$, this group is depicted in Figure \ref{subfig:snail}.
The solution of interest can be singled out by moment properties using Remark \ref{Rem:tails Y_W}
and is $X=\mu Z$. 
This gives in particular a positive answer to the question posed in \cite[Remark 4.5]{Chauvin+Liu+Pouyanne:2014}
about the existence of further solutions with infinite second moment.

\subsubsection{Cyclic P\'olya urns}	\label{subsubsec:Cyclic Polya urn}

Consider an urn containing finitely many balls of $b$ different types, $1, \ldots, b$.
At each step, a ball is drawn and placed back into the urn together with an assortment of new balls,
the types of which depend on the type of the ball drawn. Such a scheme is called a generalized P\'olya urn.
If the replacement rule is such that if a ball of type $k$ is drawn,
then it is placed back into the urn together with a ball of type $k+1$ if $k<b$ and of type $1$ if $k=b$,
the urn is called cyclic.

Let $R_{n,k}$ be the number of balls of type $1$ in a cyclic urn after $n$ steps
when starting with exactly one ball of type $k$ and no other ball.
We have $\E [R_{n,k}] = \frac nb + O(1)$ as $n \to \infty$, see e.g.\ \cite[Lemma 6.7]{Knape+Neininger:2014}.

If $b \leq 6$, then $R_{n,k}-\frac nb$, suitably scaled, is asymptotically normal.
If $b \geq 7$, let $\zeta = \exp(2\pi \imag / b) = \xi + \imag \eta$ be a primitive $b$th root of unity.
It has been shown with martingale methods \cite{Janson:2004,Pouyanne:2005}
and via the contraction method \cite[Section 6.3]{Knape+Neininger:2014} that $n^{-\xi}(R_{n,k} - \frac nb)$
has an asymptotic periodic behavior (similar to Eq. \eqref{eq:limit of b-ary search trees embedding}) that is governed by the law of a random variable $X$
which is the unique non-degenerate solution with expectation $2/(b \Gamma(\zeta+1))$ and finite second moment
of the equation
\begin{equation}	\label{eq:cyclic Polya}
X ~\eqdist~ U^\zeta X_1 + \zeta(1-U)^\zeta X_2
\end{equation}
where $X_1, X_2$ are i.i.d.\ copies of $X$ that are independent of $U$ which has the uniform distribution on $[0,1]$.

\smallskip
Therefore,
\begin{equation*}
m(s) = \E[|U^\zeta|^s + |\zeta(1-U)^\zeta|^s]	
= \frac{2}{1+ \xi s},
\end{equation*}
see Figure \ref{subfig:m(s) Polya} for a plot of $s \mapsto m(s)$.
Thus $\alpha=1/\xi$ and $\alpha \in (1,2)$ iff $\xi = \cos(2 \pi/b) > \frac12$ iff $b \geq 7$.
In particular, the phase transition at $b=6$ is visible here as the phase transition between normal and stable behavior
occurs at $\alpha=2$.

\eqref{eq:A1}--\eqref{eq:A3} are readily checked to be valid in the present context.
Further,
\begin{equation*}
\E[Z_1]	 = \E[U^{\zeta}+\zeta (1-U)^{\zeta}] = (1+\zeta) \E[U^\zeta] = 1.
\end{equation*}
Since $m(2) <1$ and $\E[|Z_1|^{2}] \leq 4 < \infty$,
the sufficient condition in Proposition \ref{Prop:convergence Z_n w}(c) is fulfilled and we conclude that $Z_n \to Z$ a.s.\ and in $\L^2$ as $n \to \infty$
for a random variable $Z$ with $\E[Z]=1$ and $\E[|Z|^2]<\infty$.
We have $\U = \{\zeta^k e^{\zeta t} \, : \,  0 \le k < b,\  t \in \R \}$;
Figure \ref{subfig:snails} is a depiction of $\U$ in the case $b=7$.
It follows that the whole spectrum of solutions given in \eqref{eq:solutions to complex smoothing equations} appears.
The special solution $X$ appearing in the description of the limiting behavior of $R_{n,k}-\frac nb$
is the unique solution to \eqref{eq:cyclic Polya} with mean $2/(b \Gamma(\zeta+1))$ and finite variance.
Since $\E[|Y_{W}|^2] = \infty$ for any non-trivial $(\U,\alpha)$-stable L\'evy process $(Y_t)_{t \geq 0}$
by Remark \ref{Rem:tails Y_W}, it is $X=2Z/(b \Gamma(\zeta+1))$.

\subsubsection{Asymptotic size of fragmentation trees}	\label{subsubsec:Size of fragmentation trees}

In Kolmogorov's conservative fragmentation model \cite{Bertoin:2006,Kolmogorov:1941} an object of mass $x=1$, say,
is split into $b$ parts with respective masses $0 \leq V_1,\ldots,V_b < 1$
where $b \geq 2$ is a fixed integer and $V_1,\ldots,V_b$ are random variables with $V_1+\ldots+V_b=1$ a.s.
The splitting procedure is repeated with the resulting objects using independent copies of the splitting vector $(V_1,\ldots,V_b)$
to determine the relative sizes of the emerging objects.
Janson and Neininger \cite{Janson+Neininger:2008} investigated the size $N(\epsilon)$ of the random fragmentation tree
the vertices of which correspond to all objects created in the fragmentation process
that have mass strictly $\geq \epsilon$ for some given $\epsilon > 0$.
They showed that the asymptotics of $N(\epsilon)$ are coded
in the function $\psi$ that maps $z \in \C$ to $\E [\sum_{j=1}^{b} V_j^z]$
(whenever the expectation exists).
To be more precise, denote by $1=\lambda_1, \lambda_2, \lambda_3, \ldots$ the roots of the equation $\psi(z)=1$
with the convention that $1 = \Real(\lambda_1) > \Real(\lambda_2) \geq \Real(\lambda_3) \geq \ldots$. 
Then, under suitable assumptions,
when $\Real(\lambda_2) \leq \frac12$, $N(\epsilon)$ suitably shifted and scaled, converges in distribution to a centered normal.
On the other hand, when $\Real(\lambda_2) > \frac12$, $N(\epsilon)$
exhibits a periodic limiting behavior governed by a complex-valued random variable $X$,
with finite second moment and a fixed expectation $\gamma \in \C$,
satisfying the distributional equation
\begin{equation}	\label{eq:FPE of fragmentation trees}
X	~\eqdist~	\sum_{j=1}^b V_j^{\lambda_2} X_j
\end{equation}
where $X_1, \ldots, X_b$ are i.i.d.\ copies of $X$ and independent of $(V_1,\ldots,V_b)$.

\smallskip
This equation is in the scope of our analysis.
Indeed, letting $T_j \defeq V_j^{\lambda_2}$ for $j=1,\ldots,b$, and $T_j = 0$ for $j>b$, we have
\begin{equation}	\label{eq:m for FPE of fragmentation trees}
m(s)	~=~ \E \bigg[\sum_{j \geq 1} |T_j|^s	\bigg]	~=~	\E \bigg[\sum_{j=1}^b V_j^{\Real(\lambda_2) s}	\bigg]
~=~	\psi(\Real(\lambda_2) s).
\end{equation}
In particular, $m(\alpha)=1$ iff $\alpha = 1/\Real(\lambda_2)$.
Again, the phase transition between normal and stable fluctuations
is reflected in the equation since $\alpha < 2$ iff $\Real(\lambda_2) > \frac12$.
Further, the conditions \eqref{eq:A1}--\eqref{eq:A3} are easily checked to hold.
Since $m(2)<1$ and
\begin{equation*}
\E \bigg[ \bigg| \sum_{j=1}^b V_j^{\lambda_2} \bigg|^2\bigg]
~\leq~	b\, \E \bigg[ \sum_{j=1}^b V_j^{2 \Real(\lambda_2)}\bigg]
~<~	1,
\end{equation*}
the sufficient condition of Proposition \ref{Prop:convergence Z_n w}(c) is fulfilled
and hence
the martingale $(Z_n)_{n \in \N_0}$ converges a.s.\ and in $\L^2$
and the special solution $X$ used to describe the limiting behavior of $N(\epsilon)$ is $\gamma Z$. The general form of solutions is given by \eqref{eq:solutions to complex smoothing equations}.

It is assumed in \cite{Janson+Neininger:2008} 
that each $V_j$ has an absolutely continuous component, hence we are in the continuous case; in fact, $\U=\{e^{\lambda_2 t} \, : \, t \in \R\}$. If the values of $\lambda_2$ are the same, we obtain the same class of $(\U, \alpha)$-stable L\'evy processes as in the case of $b$-ary search trees, but the law of $W$ will depend on the explicit distribution of $V_1, \dots, V_b$, not only on their support.

\subsubsection{Biggins' martingale with complex parameter \\	and complex Gaussian multiplicative chaos}

Eq. \eqref{eq:FP of ST hom a.s.} arises naturally in the context of branching random walks:
Consider an initial ancestor at the origin with children placed on $\R$
according to a point process $\cZ$ on $\R$ with $\E[\cZ(\R)]>1$ (supercritical case).
Each child produces offspring with positions relative to its location given by an independent copy of $\cZ$, and so on.
Denoting by $(\mathcal{S}(v))_{\abs{v}=n}$ the positions of the $n$th generation particles,
consider the Laplace transform with complex parameter $\lambda$ of the random point measure formed by the $n$th generation particles,
\begin{equation*}
\mathcal{M}_n(\lambda) ~\defeq~\sum_{|v|=n} e^{-\lambda \mathcal{S}(v)}.
\end{equation*} 
If $\m(\lambda) \defeq \E [\sum_{|v|=1} e^{-\lambda \mathcal{S}(v)}]$ is finite,
then $\E[ M_n(\lambda)]=\m(\lambda)^n$, and
\begin{equation*}
\W_n(\lambda) ~\defeq~\frac{\mathcal{M}_n(\lambda)}{\m(\lambda)^n}
\end{equation*}
is a complex-valued martingale with $\E[\W_1(\lambda)]=1$,
called {\em Biggins' martingale}, see \cite{Biggins:1977}.
Sufficient conditions for the convergence of $\W_n(\lambda)$ to a nondegenerate limit $\W(\lambda)$ are studied in \cite{Biggins:1992}.
Upon defining $T_j \defeq e^{-\lambda\mathcal{S}(j)}/\m(\lambda)$ and using the same shift notation as for the weighted branching process,
one obtains that
\begin{equation*}
\W(\lambda)
~=~\sum_{j \ge 1} \frac{e^{-\lambda \mathcal{S}(j)}}{\m(\lambda)} [\W(\lambda)]_j
~=~\sum_{j \ge 1} T_j [\W(\lambda)]_j \qquad \text{a.s.}
\end{equation*}
Thus, $\W(\lambda)$ is a solution to \eqref{eq:FP of ST hom a.s.}, in particular, $\W_n(\lambda)=Z_n$ in our notation. 

The sufficient conditions for the $\mathcal{L}^\beta$-convergence of $\W_n(\lambda)$
from \cite[Theorem 1]{Biggins:1992} translate as follows:
(2.1) there is equivalent to $\E[(\sum_{j \geq 1} \abs{T_j})^\gamma] < \infty$ for some $\gamma \in (1,2]$,
while (2.2) equals $m(\beta)<1$ for some $\beta \in (1, \gamma]$.
These imply the sufficient conditions of Proposition \ref{Prop:convergence Z_n w}(c).

\smallskip
It is an important open problem to find equivalent conditions for the convergence of $Z_n$ to a nondegenerate limit,
for it may also provide educated guesses in the theory of complex Gaussian multiplicative chaos.
This is the complex analogue of real Gaussian multiplicative chaos, which was introduced by Kahane \cite{Kahane:1985},
see also \cite{Rhodes+Vargas:2014} for a recent review and more details. 
Complex Gaussian multiplicative chaos is a complex random measure $\mathcal{M}^{\gamma,\beta}$
(with parameters $\beta, \gamma >0$) on $\R$,
which is obtained via a limiting procedure from {\em regularized} measures $\mathcal{M}^{\gamma, \beta}_\epsilon$. 
The question is about the correct renormalization needed to obtain convergence.
Three phases (Phases I, II and III) appear, see Figure 1 in \cite{Lacoin+Rhodes+Vargas:2013}.
The suitable scaling can be guessed from the behavior of Biggins' martingale with complex parameter.
A particular instance,
which was studied by Madaule, Rhodes and Vargas \cite{Madaule+Rhodes+Vargas:2015}
in order to provide intuition for the behavior on the boundary between phases I/II ($\gamma \in (1/2,1)$, $\gamma + \beta=1$),
is
\begin{equation*}
\mathcal{M}_n(\gamma, \beta) ~\defeq~ \sum_{\abs{v}=n} \exp \big( - \gamma \mathcal{S}(v) + \imag \beta \sqrt{2 \ln 2} \mathcal{S}'(v) \big),
\end{equation*}
where $(\mathcal{S}(v))_{v \in \V}$, $(\mathcal{S}'(v))_{v \in \V}$ are independent branching random walks with binary branching
and i.i.d.\ displacements with normal laws with mean $2 \log 2$ and variance ${2 \log 2}$
for $\mathcal{S}(v)$ resp.\ mean $0$ and variance $1$ for $\mathcal{S}'(v)$.

As described above, if $\mathcal{W}_n \defeq (\E[\mathcal{M}_n])^{-1} \mathcal{M}_n$ converges to a limit $\W$,
then this limit is a solution to a smoothing equation, and equal to the particular solution $Z$. 
In the setting of \cite{Madaule+Rhodes+Vargas:2015},
one has $\E [\mathcal{M}_1(\gamma,\beta)]=1$ and $m(s)= \exp\big((s \gamma -1)^2 \log 2 \big)$,
hence $\alpha = 1/\gamma \in (1,2)$. 
The phase transition is reflected in the fact that $m'(\alpha)=0$ on the boundary between phases I/II.
In this case, the sufficient conditions of Proposition \ref{Prop:convergence Z_n w} do not hold,
nevertheless it is proved in \cite[Theorem 1]{Madaule+Rhodes+Vargas:2015} that $\W_n$ converges to a nontrivial limit.

\subsection{Application of Theorem \ref{Thm:solutions to multivariate smoothing equations}}

We end this section with the study of the example described in Section \ref{subsubsec:Kac caricature}. In contrast to the examples above, the relevant solution will  be given by $Y_W$. 

Bassetti and Matthes \cite[Section 6.2]{Bassetti+Matthes:2014} study the equation
\begin{equation}	\label{eq:Bassetti+Matthes2}
V ~\eqdist~ L V_1 + R V_2,
\end{equation}
where $V, V_1, V_2$ are i.i.d.\ random vectors in $\R^3$, independent of the  random pair $(L,R)$ of similarities
which satisfies
\begin{equation*}
m(2) = \E \big[\norm{L}^2 + \norm{R}^2\big] = 1
\text{ and }
m(p) = \E \big[\norm{L}^p + \norm{R}^p\big] < 1
\end{equation*}
for some $p \in (2,3)$.
Thus \eqref{eq:A1}--\eqref{eq:A3} are satisfied with $\alpha=2$,
and all solutions to \eqref{eq:Bassetti+Matthes2} are given by Theorem \ref{Thm:solutions to multivariate smoothing equations}.
Since the equation is homogeneous, $W^*$ vanishes, while $Y_1$ can be any centered multivariate normal random variable
with a covariance matrix that is invariant under conjugation by elements of $\O$.

Bassetti and Matthes further assume that $L=lA$ for independent random variables $l \in \R$ and $A \in \Orth$,
and that the law of $A^\transp x$ dominates the volume measure on $\Sd$ for every $x \in \Sd$.
This implies that $\O$ acts transitively on $\Sd$ and hence that scalar multiples of $\Sigma=\Id$
are the only possible choices for the covariance matrix of $Y_1$.
As Proposition \ref{Prop:oSigmao^t} below shows,
the weaker assumption that there is no $\U$-invariant proper subspace readily implies $\Sigma=\Id$. 

Finally, $Z=Z^w \neq 0$ if and only if $Lw+Rw=w$ a.s., which corresponds to the conservation of momentum. In \cite{Bassetti+Matthes:2014}, only centered solutions were considered, the physical interpretation of which is that the centre of gravity does not move. Here, we see that the centre of gravity may have drift $Z^w$, which corresponds to the validity of Newtons first law in this context.

In recent papers by Dolera and Regazzini \cite{Dolera+Regazzini:2014} and Bassetti et.\;al.\;\cite{Bassetti+Ladelli+Matthes:2013},
solutions of the Boltzmann equation for Maxwellian molecules in $\R^3$ have been studied directly,
rather than its simplifications like the Kac caricature.
These steady state solutions cannot be written directly as solutions to smoothing equations,
but the techniques employed seem to be very similar.
We hope that our results allow for a better understanding; in particular under which conditions rotation invariant solutions appear.

\section{Proofs of the main results}\label{sect:proof 1}

In this section, we prove our main results: Proposition \ref{Prop:convergence Z_n w},
Theorem \ref{Thm:solutions to multivariate smoothing equations} and Proposition \ref{Prop:(U,alpha)-stable laws}.
At the beginning, we collect the relevant notation and introduce tools and concepts which are used in the proofs below. 

\subsection{Notation}	\label{subsec:Notation}

\subsubsection{Vector spaces, sets, matrices, etc}	\label{subsubsec:vector notation}
We work in the $d$-dimensional space $\R^d$.
We think of an element $x \in \R^d$ as a column vector.
We write $x^{\transp}$ for the corresponding row vector.
$e_1,\ldots,e_d$ denote the canonical basis vectors of $\R^d$.
By $\scalar{\cdot,\cdot}$, we denote the standard Euclidean scalar product on $\R^d$,
that is, $\scalar{x,y} = x^\transp y$ for $x,y \in \R^d$.
We write $|x|$ for $\sqrt{\scalar{x,x}}$, the Euclidean norm of $x$.
For a set $B \subseteq \R^d$, $\partial B$ denotes the boundary of $B$,
$B^\perp = \{y \in \R^d: \scalar{x,y} = 0 \text{ for all } x \in B\}$ denotes the orthogonal complement of $B$ in $\R^d$.
$B_r \defeq \{x \in \R^d: |x|<r\}$ denotes the ball of radius $r$ centered around the origin, $r > 0$.
For a given real $d \times d$ matrix $A$, we write $A_{ij}$ for the coefficient in the $i$th row and the $j$th column of $A$
and $\norm{A} \defeq \sup_{|x|=1}\abs{Ax}$ for its norm.
$A^{\transp}, \trace(A)$ and $\det(A)$ denote the transpose, the trace and the determinant of $A$, respectively.
For $\lambda \in \R$, we set $E_{\lambda}(A) = \{x \in \R^d: Ax = \lambda x\}$.
We write $I_{k}$ for the $k \times k$ identity matrix, $k \in \N$.

\subsubsection{Probability spaces, expectations, etc}
Throughout the paper, we fix a probability space $(\Omega,\A,\Prob)$ which is large enough to carry all random variables
appearing in the paper. By $\E[\cdot]$ and $\Var[\cdot]$ 
we  denote expectation resp.\ variance 
with respect to $\Prob$.
We also consider expectations of random vectors  and random matrices.
which are defined componentwise.
For a random vector $Y$, we write $\Cov[Y] \defeq \E[(Y-\E[Y])(Y^{\transp}-\E[Y]^{\transp})]$ for the covariance matrix of $Y$ with respect to $\Prob$.
$\1_A$ denotes the indicator function of a set $A$ and we write $\E[Y; A]$ for $\E[Y \1_{A}]$.
Further, $\Cov[Y;A]$ is the covariance matrix of the random vector $Y \1_{A}$.

\subsubsection{Relevant groups}
The following groups are of relevance. $\Sim \subseteq GL(d,\R)$ is the group of similarity matrix, i.e., scalar multiples of orthogonal matrices.
$\U$ denotes the smallest closed subgroup of $\Sim$ that covers $\{T_j: j=1,\ldots,N\}$ with probability one.
In the complex case, we identify $\U$ with a subgroup of the multiplicative group $\C^*$ of $\C$. 
Similarly, $\O$ is the smallest closed subgroup of $\Orth$
which contains the random set $\{T_j/\norm{T_j}: j=1,\ldots,N\}$ with probability one.
In the complex case, we identify $\O$ with the smallest multiplicative subgroup of $\S = \{z \in \C: |z|=1\}$
which contains the random set $\{T_j/|T_j|: j=1,\ldots,N\}$ with probability one.
In this case, either $\O = \{e^{2 \pi \imag k / m}: k=0,\ldots,m-1\}$ for some $m \in \N$ or $\O = \S$.
Finally, let $\G$ be the smallest closed multiplicative subgroup of $\Rp$ that covers the random set $\{\norm{T_j}: j=1,\ldots,N\}$ with probability one.
Equivalently, $\G$ is the image of the group $\U$ under the homomorphism $u \mapsto \norm{u}$.
There are three possibilities:
(C) $\G = \Rp$;
(G) $\G = r^{\Z}$ for some $r>1$;
(T) $\G = \{1\}$.
The trivial case (T) is excluded by \eqref{eq:A3}.
We refer to (G) as the geometric or $r$-geometric case and to (C) as the continuous or non-geometric case.

More information about the structure of $\U$ is provided in Section \ref{app:polar coordinates}.
In particular, $\U = \AU \ltimes \CU$ for a one-parameter group $\AU$ which is isomorphic to $\G$, and $\CU=\U \cap \Orth$.
Note that in general, $\CU \subsetneq \O$.

Throughout the paper,
we call a measure $\mu$ on $\R^d$ $G$-(left)-invariant for a group $G$ of matrices
if $\mu(g^{-1}B)=\mu(B)$ for all $g \in G$ and all Borel sets $B \subseteq \R^d$.

\subsubsection{Weighted branching}	\label{subsubsec:weighted branching notation}
We set $\F_n \defeq \sigma((C(v),T(v)): |v|<n)$ and $\F \defeq \sigma(\F_n: n \in \N_0)$.
$(\bC,\bT)$ serves as an abbreviation for the family $((C(v),T(v)))_{v \in \V}$.
We further assume that on the basic probability space, a family $\bX \defeq (X_v)_{v \in \V}$ of i.i.d.\ random variables is defined
which is independent of $\F$. We do not specify the law of $X \defeq X_{\varnothing}$ here, but typically,
$X$ will be a solution to \eqref{eq:FP of ST inhom} or \eqref{eq:FP of ST hom}.

\subsection{Characteristic functions}	\label{subsec:Characteristic functions}

It is natural to approach Equation \eqref{eq:FP of ST inhom} via characteristic functions.
Indeed, the distributional equation \eqref{eq:FP of ST inhom} for an $\R^d$-valued random variable $X$
is equivalent to the functional equation
\begin{equation}	\label{eq:FE of ST inhom}
\phi(x)	~=~	\E\bigg[e^{\imag \scalar{x,C}} \prod_{j \geq 1} \phi(T_j^\transp x)\bigg],	\qquad	x \in \R^d
\end{equation}
for the characteristic function $\phi(x) = \E[e^{\imag \scalar{x,X}}]$ of $X$.
In the homogeneous case, the functional equation takes the simpler form
\begin{equation}	\label{eq:FE of ST hom}
\phi(x)	~=~	\E\bigg[\prod_{j \geq 1} \phi(T_j^\transp x)\bigg],	\qquad	x \in \R^d.
\end{equation}
Solving \eqref{eq:FP of ST inhom} is equivalent to finding all characteristic functions $\phi$ of $\R^d$-valued random variables
which satisfy \eqref{eq:FE of ST inhom}.

\subsection{Proof of Theorem \ref{Thm:solutions to multivariate smoothing equations}: The direct inclusion} \label{subsec:proof direct inclusion}

We are now ready to prove the direct inclusion of Theorem \ref{Thm:solutions to multivariate smoothing equations}.
Here, with the direct inclusion, we mean the assertion that any distribution of a random variable of the form
\eqref{eq:solutions to multivariate smoothing equations} is a solution to \eqref{eq:FP of ST inhom}.

\begin{proof}[of the direct inclusion of Theorem \ref{Thm:solutions to multivariate smoothing equations}]
Consider the situation of Theorem \ref{Thm:solutions to multivariate smoothing equations}
and let $X = W^* + Y_{W} + Z$ where $(Y_{t})_{t \geq 0}$ is a strictly $(\U,\alpha)$-stable L\'evy process
independent of the family $(\bC,\bT)$, in particular independent of $(W^*,W,Z)$.
Denote the characteristic function of $X$ by $\phi$
and the characteristic exponent of $Y_{1}$ by $\Psi$, that is,
$\E[e^{\imag \scalar{x,Y_{t}}}] = \exp(t\Psi(x))$.
By assumption, $\Psi$ satisfies \eqref{eq:(u,alpha)-stable chf} with $b(u) \equiv 0$ for $u \in \U$.
Using the independence of $(W^*,W,Z)$ and $(Y_{t})_{t \geq 0}$, we conclude that
\begin{equation}	\label{eq:ch.f. of solution}
\phi(x)
= \E\big[\exp\big(\imag \scalar{x,W^*} + \imag \scalar{x,Z} + W\Psi(x) \big) \big],	\quad	x \in \R^d.
\end{equation}
As $([W^*]_j,[W]_j,[Z]_j)$ is a copy of $(W^*,W,Z)$, \eqref{eq:ch.f. of solution} still holds
when $(W^*,W,Z)$ is replaced by $([W^*]_j,[W]_j,[Z]_j)$, $j \in \N$.
Furthermore, since $([W^*]_j,[W]_j,[Z]_j)$ is independent of $\F_1$, we conclude that
\begin{align*}
\phi(T_j^{\transp} x)
&= \E\big[\exp(\imag \scalar{T_j^{\transp} x,[W^*]_j} + \imag \scalar{T_j^{\transp} x,[Z]_j} +  [W]_j \Psi(T_j^{\transp} x)) \mid \F_1\big]	\\
&= \E\big[\exp(\imag \scalar{x, T_j [W^*]_j} + \imag \scalar{x,T_j [Z]_j} + \norm{T_j}^{\alpha} [W]_j \Psi(x)) \mid \F_1\big]
\end{align*}
a.s.\ for every $x \in \R^d$, where in the last step we have used \eqref{eq:(u,alpha)-stable chf} and the fact that $T_j \in \U$ a.s.
Consequently, for every $x \in \R^d$, we infer
\begin{align*}
\E\bigg[&e^{\imag \scalar{x,C}} \prod_{j \geq 1} \phi(T_j^\transp x)\bigg]	\\
&= \E\bigg[e^{\imag \scalar{x,C}} \prod_{j \geq 1}
\E\Big[e^{\imag \scalar{x, T_j [W^*]_j} + \imag \scalar{x,T_j [Z]_j} + \norm{T_j}^{\alpha} [W]_j \Psi(x)} \,\big|\, \F_1\Big]\bigg]	\\
&= \E\Big[e^{\imag \scalar{x,C} + \sum_{j \geq 1} \big(\imag \scalar{x, T_j [W^*]_j} + \imag \scalar{x,T_j [Z]_j} + \norm{T_j}^{\alpha} [W]_j \Psi(x)\big)}\Big]	\\
&= \E\Big[e^{\imag \scalar{x,C + \sum_{j \geq 1} T_j [W^*]_j} + \imag \scalar{x, \sum_{j \geq 1} T_j [Z]_j}
+ \sum_{j \geq 1} \norm{T_j}^{\alpha}[W]_j \Psi(x)}\Big]	\\
&= \E\Big[e^{\imag \scalar{x,W^*} + \imag \scalar{x, Z} + W \Psi(x)}\Big]
= \phi(x),
\end{align*}
i.e., $\phi$ solves \eqref{eq:FE of ST inhom}. Hence, $X$ is a solution to \eqref{eq:FP of ST inhom}.
 \qed
\end{proof}

Most of the remainder of this paper is devoted to the proof of the converse inclusion of Theorem \ref{Thm:solutions to multivariate smoothing equations}
and the description of the class of strictly $(\U,\alpha)$-stable L\'evy processes.
We begin with a short section on a common technique in the theory of branching processes,
an exponential change of measure.

\subsection{Exponential change of measure}	\label{subsec:change of measure}

Recall from Section \ref{subsubsec:Weighted branching} the definition of the weighted branching process.
We define the associated random walk $(L_n)_{n \in \N_0}$ on $\Sim$ by the many-to-one formula
\begin{equation}	\label{eq:change of measure}
\E[f(L_0,\ldots,L_n)]	~=~	\E\bigg[\sum_{|v|=n} \norm{L(v)}^{\alpha} f((L(v|_k))_{k=0,\ldots,n})\bigg]
\end{equation}
for all nonnegative Borel-measurable functions $f: \Sim^{n+1} \to \Rnn$.
\eqref{eq:A2} implies that the law of $(L_n)_{n \in \N_0}$ is a proper probability measure.
From this definition, it can be checked that $(L_n)_{n \in \N_0}$ is a multiplicative random walk on the group $\U \subseteq \Sim$.
We define $S_n \defeq -\log \norm{L_n}$ and $O_n \defeq L_n/\norm{L_n}$.
$(S_n)_{n \in \N_0}$ is a standard random walk on $\R$,
while $(O_n)_{n \in \N_0}$ is a multiplicative random walk on $\Orth$.
The step distribution of $(S_n)_{n \in \N_0}$ is given by
\begin{equation}	\label{eq:S_1}
\Prob(S_1 \in \cdot)	~=~	\E\bigg[\sum_{j=1}^N \norm{T_j}^{\alpha} \delta_{-\log \norm{T_j}}(\cdot)\bigg].
\end{equation}
Consequently, when \eqref{eq:A3} holds,
$\E[S_1] = -m'(\alpha) \in (0,\infty)$.

Later on, we will use that \eqref{eq:change of measure} remains valid under certain stopping rules: Considering 
$ \tau(t)\defeq\inf \{{n \in \N_0} \, : \, S_n >t \text{ and } S_k \le t \text{ for all } k <n\}$,
\eqref{eq:change of measure} gives
\begin{align*}
\E&[f(L_0,\ldots,L_n) \1_{\{\tau(t)=n\}}]	\\
&=~	\E\bigg[\sum_{|v|=n} \norm{L(v)}^{\alpha} f((L(v|_k))_{k=0,\ldots,n}) \1_{\{S(v) > t \geq S(v|_k) \, \forall \,  k<n \}}\bigg].
\end{align*}
Summing over all $n \in \N_0$ and defining the {\em coming generation at time} $t \geq 0$,
\begin{equation}	\label{eq:C(t)}
\mathcal{C}(t)	~=~	\{v \in \V: \norm{L(v)}>0 \text{ and } S(v) > t \geq S(v|_k) \text{ for all } k<|v|\},
\end{equation}
we infer that in particular
\begin{equation}	\label{eq:change of measure stopping line}
\E[f(S_{\tau(t)-1},S_{\tau(t)},O_{\tau(t)})] = \E\bigg[\sum_{v \in \mathcal{C}(t)} e^{-\alpha S(v)} f(S(v|_{|v|-1}), S(v), O(v))\bigg]
\end{equation}
for all nonnegative Borel-measurable functions $f: \R^2 \times \Orth \to \Rnn$.
See \cite{Kyprianou:2000} for more information on stopping lines and further references.

\subsection{Multiplicative martingales}	\label{subsec:multiplicative martingales}

Let $X$ be a solution to \eqref{eq:FP of ST inhom} and denote the characteristic function of $X$ by $\phi$.
We will show that $\phi$ is the characteristic function of a random variable of the form \eqref{eq:solutions to multivariate smoothing equations}.
As in previous works on fixed points of smoothing transformations
\cite{Alsmeyer+Biggins+Meiners:2012,Alsmeyer+Meiners:2012,Alsmeyer+Meiners:2013,Biggins+Kyprianou:1997,Biggins+Kyprianou:2005,Iksanov+Meiners:2015a},
we make use of multiplicative martingales.
As this technique is well-known by now,
we keep the presentation short here
and refer to the above references for more detailed expositions.
For $x \in \R^d$, define
\begin{equation}	\label{eq:M_n}
M_n(x)	~\defeq~	\exp(\imag \scalar{x,W_n^*}) \cdot \prod_{|v|=n} \phi \big(L(v)^{\transp} x\big),	\quad	n \in \N_0.
\end{equation}
The fact that $\phi$ solves \eqref{eq:FE of ST inhom} implies that $(M_n(x))_{n \in \N_0}$
is a complex-valued martingale.
Since it is bounded by $1$ in absolute value, it converges a.s.\ and in mean.
We denote its a.s.\ limit by $M(x)$ and note that
\begin{equation}	\label{eq:phi(x)=EM(x)}
\phi(x)	~=~	\E[M(x)],	\quad	x \in \R^d.
\end{equation}
In order to determine $\phi(x)$, it thus suffices to determine $M(x)$.

We begin with a key lemma. We will sometimes write $M(x)=M(\omega,x)$ in order to make more transparent,
whether we consider $M(x)$ as a random variable, or $x \mapsto M(\omega,x)$ as a function of $x$. 

\begin{lemma}	\label{Lem:Caliebe's lemma}
There is an $\F$-measurable set $N$ with $\Prob(N)=0$
such that, for  $\omega \in N^\comp$, $M_n(\omega, \cdot) \to M(\omega,\cdot)$ pointwise on $\R^d$.
Further, $x \mapsto M(\omega,x)$ is the characteristic function of a probability distribution on $\R^d$ for all $\omega \in N^\comp$.
\end{lemma}

An obvious modification of the proof of Theorem 1 in \cite{Caliebe:2003} yields the result.
We refrain from giving any details.

Solving the inhomogeneous equation \eqref{eq:FP of ST inhom} can be reduced to
solving the associated homogeneous equation \eqref{eq:FP of ST hom} along the lines of
\cite[Section 5]{Alsmeyer+Meiners:2013},
a sketch of the reduction argument will be given in Section \ref{subsec:inhomogeneous equation} below.
For now, we restrict our attention to the homogeneous case and assume that $C=0$ a.s.
Then \eqref{eq:M_n} takes the simpler form
\begin{equation}	\label{eq:M_n hom}
M_n(x)	~=~	\prod_{|v|=n} \phi \big(L(v)^{\transp} x\big),	\quad	n \in \N_0.
\end{equation}
We claim that for all $\omega$ from a set of probability one,
$M(\omega,\cdot)$ is the characteristic function of an infinitely divisible law.
Since $\sup_{|v|=n} \norm{L(v)} \to 0$ a.s., we can assume without loss of generality
that the set $N$ from Lemma \ref{Lem:Caliebe's lemma} is such that
$N^\comp \subseteq \{\sup_{|v|=n} \norm{L(v)} \to 0\}$.
Now pick an $\omega \in N^\comp$.
We view $M_n(\omega,\cdot)$ as the characteristic function of $\sum_{|v|=n} L(v) X_v$ conditional given $(L(v))_{v \in \V}$.
Thus, $M(\omega,\cdot)$ is the limit of characteristic functions of the row sums in a triangular array
which is independent and infinitesimal.
Hence, it is the characteristic function of an infinitely divisible law
and we can write $M(x) = \exp(\Psi(x))$ for a random characteristic exponent $\Psi(x)$ which is of the form
\begin{equation}	\label{eq:Psi}
\Psi(x)	~=~	\imag \scalar{W',x} - \frac{x^{\transp} \mathbf{\Sigma} x}{2}
+ \int \left(e^{\imag \scalar{x,y}} - 1 - \imag \scalar{x,y} \1_{[0,1]}(\abs{y}) \right) \nu(\dy)
\end{equation}
for $W'=W'(\bT) \in \R^d$, a covariance matrix $\mathbf{\Sigma}=\mathbf{\Sigma}(\bT)$ and a L\'evy measure $\nu=\nu(\bT)$ on $\R^d$ (see \cite[p.\,290]{Kallenberg:2002}).

That $W', \Sigma$ and $\nu$ are indeed functions of $\bT$ comes from the following representation,
valid for triangular arrays (see \cite[Chapter 15]{Kallenberg:2002}):
Using the convention $\int_{\{h < |x| \leq 1\}} = -\int_{\{1<|x| \leq h\}}$ when $h>1$, it holds on $N^\comp$,
\begin{equation}	\label{eq:W'}
W'	~=~	W^h + \int_{\{h < |x| \leq 1\}} x \, \nu(\dx)
\end{equation}
where $W^h$ is defined by
\begin{equation}	\label{eq:W'^h}
W^h	~\defeq~	\lim_{n \to \infty} \sum_{|v|=n} \E[L(v) X_v; |L(v)X_v| \leq h \mid \F]
\end{equation}
for every $h>0$ with $\nu(\{|x|=h\}) = 0$,
\begin{equation}	\label{eq:Sigma}
\mathbf{\Sigma}	~=~	\mathbf{\Sigma}^h - \int_{\{h < |x| \leq 1\}} x x^\transp \, \nu(\dx)
\end{equation}
where $\mathbf{\Sigma}^h$ is defined by
\begin{equation}	\label{eq:Sigma^h}
\mathbf{\Sigma}^h	~\defeq~	\lim_{n \to \infty} \sum_{|v|=n} \Cov[L(v) X_v; |L(v)X_v| \leq h \mid \F],
\end{equation}
and
\begin{equation}	\label{eq:nu is vague limit}
\int f(x) \, \nu(\dx)	~=~	\lim_{n \to \infty} \sum_{|v|=n} \int f(L(v)x) \, F(\dx)
\end{equation}
for all continuous functions $f$ with compact support on $\overline{\R^d}\setminus\{0\}$
(\,$\overline{\R^d}$ denotes the one-point compactification of $\R^d$). 
Eq. \eqref{eq:nu  is vague limit} also yields that $\nu$ is a random measure (see \cite[Lemma 4.1]{Kallenberg:1986}), 
i.e., the mapping $\bT \mapsto \int f(x) \nu(\bT,\dx)$ is measurable for every nonnegative Borel measurable function $f$ on
$\R^d \setminus \{0\}$.
In order to show that $W'$ and $\mathbf{\Sigma}$ are random variables as well,
we need some more preparation (the problem is to choose $h$ in a measurable way).

Right now, we can use that $W'$, $\Sigma$ and $\nu$ are functions of $\bT$ and apply the shift operator $[\cdot]_v$. 
Arguing as in the proof of Lemma 4.3 in \cite{Alsmeyer+Meiners:2013}, we conclude that on $N^\comp$
\begin{equation}	\label{eq:M}
M(x)	~=~	\prod_{|v|=n} [M]_v(L(v)^{\transp} x)	\quad	\text{for all } x \in \R^d \text{ and } n \in \N_0.
\end{equation}
Now using \eqref{eq:M} in \eqref{eq:Psi}, we conclude that for all $n \in \N_0$, on $N^\comp$
\begin{align}
\Psi(x)
&= \imag \!\sum_{|v|=n} \scalar{[W']_v, L(v)^\transp x} - \sum_{|v|=n} \frac{x^\transp L(v) [\mathbf{\Sigma}]_v L(v)^\transp x}{2} \nonumber \\ 
&\phantom{=}
+\!\sum_{|v|=n}\!\int\!\! \left( e^{\imag \scalar{L(v)^\transp x, y}} - 1 - \imag \scalar{L(v)^\transp x,y} \1_{[0,1]}(\abs{y})  \right) \![\nu]_v(\dy) \nonumber \\
&= \imag\scalar{\sum_{|v|=n}\! L(v) [W']_v, x} - \frac{1}{2}\, x^\transp \sum_{|v|=n} L(v) [\mathbf{\Sigma}]_v L(v)^\transp x \label{eq:FE Psi} \\
&\phantom{=}
+\!\sum_{|v|=n}\! \int \!\! \left( e^{\imag  \scalar{x,L(v)y}} - 1 - \imag \scalar{x,L(v)y} \1_{[0,1]}(\abs{L(v) y})  \right) \![\nu]_v(\dy) \nonumber	\\
&\phantom{=}
-\!\sum_{|v|=n}\! \int \!\! \left( \imag \scalar{x,L(v)y} \1_{[0,1]}(\abs{y})  - \imag \scalar{x,L(v)y} \1_{[0,1]}(\abs{L(v) y})   \right) \![\nu]_v(\dy) \nonumber.
\end{align}
Since the last term contributes to the random shift,
using the uniqueness of the L\'evy triplet, we get:
\begin{align}
\int f(y) \, \nu(\dy)						~&=~	\sum_{|v|=n} \int f(L(v) y) \, [\nu]_v(\dy)	\quad	\text{on } N^\comp	\label{eq:FE nu}	\\
\text{and}	\qquad	\mathbf{\Sigma}	~&=~	\sum_{|v|=n} L(v) [\mathbf{\Sigma}]_v L(v)^{\transp}	\quad	\text{on } N^\comp	\label{eq:FE Sigma}
\end{align}
for all $n \in \N_0$ and all nonnegative Borel-measurable functions $f$ on $\overline{\R^d}\setminus\{0\}$.
We will use \eqref{eq:FE nu} and \eqref{eq:FE Sigma} to determine $\nu$ and $\mathbf{\Sigma}$, respectively.

The next lemma gives an important estimate for $\nu$ and will allow us to infer the measurability of $W'$ and $\mathbf{\Sigma}$.

\begin{lemma}
There is a multiplicatively $r$-periodic (if $\G = r^{\Z}$) or constant (if $\G = \Rp$) function
$\mathfrak{h}:[0,\infty) \to (0,\infty)$ such that $t \mapsto \mathfrak{h}(t)t^{\alpha}$ is nondecreasing
and such that on $N^\comp$,
\begin{equation}	\label{eq:nu evaluated outside balls}
\nu(B_{|x|}^{\comp})	~=~	W \mathfrak{h}(|x|^{-1}) |x|^{-\alpha}	\quad	\text{for all } x \in \R^d \setminus\{0\}.
\end{equation}
\end{lemma}
\begin{proof}
Define $f:[0,\infty) \to [0,1]$ by
\begin{equation*}
f(t)	~\defeq~	\begin{cases}
				\E \big[ \exp \big( -\nu \big( B_{t^{-1}}^{\comp} \big)\big)\big]			&	\text{for } t > 0,	\\
				1																			&	\text{for } t = 0.
				\end{cases}
\end{equation*}
$f$ is decreasing in $t$ and continuous at $0$ since $B_{t^{-1}}^{\comp} \downarrow \emptyset$ as $t \downarrow 0$.
The most important property of $f$, however, is that it solves the functional equation of the smoothing transformation
as studied in \cite{Alsmeyer+Biggins+Meiners:2012}.
Indeed, using \eqref{eq:FE nu} and the fact that $L(v)=\norm{L(v)} O(v)$ for an orthogonal matrix $O(v)$
and $O(v)^{-1} B_{r}^{\comp} = B_{r}^{\comp}$ for all $r>0$,
we get
\begin{eqnarray*}
f(t)
& = &
\E \bigg[ \exp \bigg( - \sum_{|v|=n} [\nu]_v (L(v)^{-1} B_{t^{-1}}^{\comp}) \bigg) \bigg] \\
& = &
\E \bigg[ \prod_{|v|=n} \exp\left( - [\nu]_v B_{(\norm{L(v)}t)^{-1}}^{\comp}) \right)  \bigg]
~=~	\E \bigg[ \prod_{|v|=n} f \left( \norm{L(v)} t\right)  \bigg].
\end{eqnarray*}
Consider the limit $M_f$ of the multiplicative martingale associated with $f$, i.e.,
$M_f(t) = \lim_{n \to \infty} \prod_{|v|=n} f(\norm{L(v)}t)$, $t \geq 0$.
By Theorem 8.3 in \cite{Alsmeyer+Biggins+Meiners:2012}, $M_f(t) = \exp(-W\mathfrak{h}(t)t^{\alpha})$ a.s.~for all $t \geq 0$
and some function $\mathfrak{h}$ with properties as above.
By the arguments given in the proof of \cite[Lemma 4.8]{Alsmeyer+Meiners:2013},
\begin{equation}	
\nu(B_{|x|}^{\comp})	~=~	W \mathfrak{h}(|x|^{-1}) |x|^{-\alpha}	\quad	\text{for all } x \in \R^d \setminus\{0\} \text{ on } N^\comp.
 \qed
\end{equation}
\end{proof}

Since \eqref{eq:nu evaluated outside balls} holds, we can pick some $0 < h \leq 1$ such that $\mathfrak{h}$ is continuous at $h$,
in particular, $\nu(\partial B_h) = 0$ on $N^\comp$.
Using this $h$ in \eqref{eq:W'}, \eqref{eq:W'^h}, \eqref{eq:Sigma} and \eqref{eq:Sigma^h}
implies the asserted measurability statements for $W'$ and $\mathbf{\Sigma}$.

We call a L\'evy triplet $(W',\mathbf{\Sigma},\nu)$  with these measurability properties an $\F$-measurable L\'evy triplet.
We summarize the results of the above discussion in the following proposition.

\begin{proposition}	\label{Prop:Multiplicative martingales}
Assume that \eqref{eq:A1} and \eqref{eq:A2} hold.
Let $X$ be a solution to \eqref{eq:FP of ST hom} and denote its characteristic function and distribution function by $\phi$ and $F$, respectively.
Denote by $(M_n(x))_{n \in \N_0}$, $x \in \R^d$ the multiplicative martingales associated with $X$ and defined by \eqref{eq:M_n hom}.
Then, on an $\F$-measurable set $N^{\comp}$ with $\Prob(N)=0$,
\begin{equation}
M_n(x)	\to	M(x)	\quad	\text{ as } n \to \infty
\end{equation}
for all $x \in \R^d$.
For $\omega \in N^\comp$, $x \mapsto M(\omega,x)$ is a characteristic function 
and possesses a representation $M(x) = \exp(\Psi(x))$ for all $x \in \R^d$
where $\Psi$ is given by \eqref{eq:Psi}.
$(W',\mathbf{\Sigma},\nu)$ is an $\F$-measurable random L\'evy triplet satisfying \eqref{eq:W'}--\eqref{eq:nu is vague limit}.
\end{proposition}

Next, we state some consequences of Proposition \ref{Prop:Multiplicative martingales}
concerning the tail behavior of solutions to \eqref{eq:FP of ST hom}.
These will be useful when determining $\nu$ and $\mathbf{\Sigma}$.

\subsection{Tail estimates}	\label{subsec:Tail estimates}

If $X$ is a solution to \eqref{eq:FP of ST hom}, then
\begin{equation}	\label{eq:tail F}
\limsup_{t \to \infty} t^{\alpha} \Prob(\abs{X} > t)	~<~	\infty.
\end{equation}
If, additionally, the random L\'evy measure $\nu$ of the limit of the multiplicative martingale $M$
associated with $X$ (or its characteristic function $\phi$) vanishes a.s., then the stronger estimate
\begin{equation}	\label{eq:tail F endogenous}
\limsup_{t \to \infty} t^{\alpha} \Prob(\abs{X} > t)	~=~	0
\end{equation}
holds.
The derivation of these estimates can be carried out along the lines of \cite[Lemma 4.7]{Iksanov+Meiners:2015a}
and \cite[Lemma 4.9]{Alsmeyer+Meiners:2013}, we refrain from giving more details here.
As a consequence of \eqref{eq:tail F}, we obtain the following inequality
for 
\begin{equation} \label{eq:L_beta} L_\beta(t) \defeq \E [|X|^\beta ; |X| \leq t] \end{equation} with $\beta > \alpha$:
\begin{align}	\label{eq:L_beta asymptotics}
\limsup_{t \to \infty} t^{\alpha-\beta} L_\beta(t)
& \leq
\limsup_{t \to \infty} t^{\alpha-\beta} \int_0^t \beta x^{\beta-1} \Prob(|X|>x) \, \dx	\\
& \leq
C \limsup_{t \to \infty} t^{\alpha-\beta} \int_0^t x^{\beta-\alpha-1} \, \dx
<	\infty	\notag
\end{align}
where $C$ is a constant depending on $\alpha, \beta$ and the law of $X$ (via \eqref{eq:tail F}).
In the case where the random L\'evy measure $\nu$ vanishes a.s., using \eqref{eq:tail F endogenous} instead of \eqref{eq:tail F}
produces the stronger estimate
\begin{align}	\label{eq:L_beta asymptotics endogenous}
\lim_{t \to \infty} t^{\alpha-\beta} L_\beta(t)	~=~	0.
\end{align}

\subsection{Determining $\nu$}	\label{subsec:nu}

The characterization of $\nu$ is given by the following lemma.

\begin{lemma}	\label{Lem:nu evaluated}
Assume that \eqref{eq:A1}--\eqref{eq:A3} are in force and let $\nu$ be a $\F$-measurable random L\'evy measure.
Then $\nu$ satisfies \eqref{eq:FE nu} if and only if $\nu = W \bar \nu$ a.s.\ for a deterministic $(\U,\alpha)$-invariant L\'evy measure $\bar\nu$, i.e.,
satisfying
\begin{equation}	\label{eq:bar nu (U,alpha)-invariant}
\bar \nu(g B)	~=~	\norm{g}^{-\alpha} \bar \nu(B)
\end{equation}
for all $g \in \U$ and all Borel sets $B \subseteq \R^d \setminus \{0\}$. Further, $\bar{\nu}=0$ if $\alpha \ge 2$.
\end{lemma}
\begin{proof}
We first prove the sufficiency. To this end, suppose $\nu = W \bar\nu$ a.s.\ for a deterministic L\'evy measure $\bar\nu$
satisfying \eqref{eq:bar nu (U,alpha)-invariant}. It suffices to check validity of \eqref{eq:FE nu}
when $f$ is the indicator function of an arbitrary Borel set $B \subseteq \R^d \setminus \{0\}$.
Using \eqref{eq:bar nu (U,alpha)-invariant} and then \eqref{eq:FP tilted a.s.}, we obtain
\begin{eqnarray*}
\sum_{|v|=n} [\nu]_v(L(v)^{-1}B)
& = &
\sum_{|v|=n} [W]_v \bar\nu(L(v)^{-1}B)	\\
& = &
\sum_{|v|=n} \norm{L(v)}^{\alpha} [W]_v \bar\nu(B)
~=~	W\bar\nu(B)
~=~	\nu(B)	\quad	\text{a.s.}
\end{eqnarray*}

For the converse implication, assume that $\nu$ is an $\F$-measurable random L\'evy measure satisfying \eqref{eq:FE nu}.

The idea of the proof is as follows: We define $\bar{\nu}\defeq \E[\nu]$
and prove that $g \mapsto \norm{g}^{-\alpha} \bar{\nu} (g^{-1} B)$ is a constant function in $g$
(this necessitates proving that $(\omega,g) \mapsto \nu(g^{-1} B)$ is product-measurable).
Then we show that the identity $\nu(B)=W \bar{\nu}(B)$ a.s.\ holds simultaneously for enough sets $B$
(namely, a countable generator of the Borel $\sigma$-field which is closed under finite intersections)
on a common $\Prob$-null set. \\

\Step 1: We start by proving the measurability statement:
For every Borel set $B \subseteq \R^d \setminus \{0\}$,
the mapping $(\omega,g) \mapsto \nu(g^{-1}B)$ from $\Omega \times \Sim$ to $\Rnn$
is $\F \otimes \B(\Sim)$-measurable where $\B(\Sim)$ denotes the Borel $\sigma$-field on $\Sim$.
Let $B \subseteq \R^d \setminus \{0\}$ be closed,
$B^{1/k} \defeq \{y \in \R^d: |y-z|<\frac1k \text{ for some } z \in B\}$
and $f_k$ be a continuous function satisfying $0 \leq f_k \leq 1$,
$f_k|_B=1$ and $f_k|_{(B^{1/k})^\comp}=0$ (such a function exists, for instance, by Urysohn's lemma).
Then $f_k \to \1_B$ as $k \to \infty$.
By \eqref{eq:nu is vague limit},
\begin{equation*}
\int f_k(gx) \, \nu(\dx)	~=~	\lim_{n \to \infty} \sum_{|v|=n} \int f_k(L(v)gx) \, F(\dx)
\end{equation*}
By Fubini's theorem, the right-hand side is $\F \otimes \B(\Sim)$-measurable
and by the dominated convergence theorem,
the left-hand side tends to $\nu(g^{-1}B)$ as $k \to \infty$.
Hence $(\omega,g) \mapsto \nu(g^{-1}B)$ is $\F \otimes \B(\Sim)$-measurable.
For every fixed $r>0$, this extends to the Dynkin system generated by all closed sets $B \subseteq B_r^\comp$,
hence to all Borel measurable sets $B \subseteq \R^d\setminus\{0\}$. \\

\Step 2: Now we introduce the sets that will form the countable generator. We define, for $x \in \R^d$ and $\epsilon > 0$,
\begin{equation*}
I_{x}^\epsilon ~=~ \left\{y \in \R^d~:~\abs{y} \geq |x|,~\Big|\frac{y}{\abs{y}}-\frac{x}{\abs{x}}\Big| \leq \epsilon \right\}.
\end{equation*}
We will determine $\nu(I_{x/|x|^2}^{\epsilon})$.
First notice that, for $o \in \Orth$,
\begin{equation*}
\bigg|\frac{oy}{|oy|} - \frac{x}{\abs{x}} \bigg|	~=~	\bigg|o^{-1} \bigg(o \frac{y}{\abs{y}} - \frac{x}{\abs{x}} \bigg)\bigg|
~=~\bigg| \frac{y}{\abs{y}} - o^{-1}\frac{x}{\abs{x}} \bigg|
\end{equation*}
and, therefore, for a similarity $g = \norm{g} o$ with $\norm{g}>0$,
\begin{equation}		\label{eq:I_x^epsilon transformed}
g^{-1} I_{x}^\epsilon
=	\left\{y \in \R^d: |y| \geq \frac{|x|}{\norm{g}},\bigg| \frac{y}{\abs{y}} - o^{-1}\frac{x}{\abs{x}}\bigg| \leq \epsilon \right\}
=	I_{g^{-1}x}^{\epsilon}
= I^{\epsilon}_{g^{\transp} x / \norm{g}^2}	
\end{equation}
where in the last step, we have used that $g^{-1} = \norm{g}^{-1} o^{\transp} = g^{\transp}/\norm{g}^2$.
Using this in \eqref{eq:FE nu} with $f = \1_{I_{x/|x|^2}^{\epsilon}}$ gives
\begin{equation}	\label{eq:nu evaluation}
\nu(I_{x/|x|^2}^{\epsilon})
=	\sum_{|v|=n} [\nu]_v(L(v)^{-1} I_{x/|x|^2}^{\epsilon})
=	\sum_{|v|=n} [\nu]_v(I_{L(v)^{\transp} x/|L(v)^{\transp}x|^2}^{\epsilon})	\text{ a.s.}
\end{equation}

\Step 3:
For $x \in \R^d \setminus \{0\}$, define $\Psi_\epsilon(x)	 \defeq \nu(I_{x/|x|^2}^{\epsilon}) |x|^{-\alpha}$.
By \eqref{eq:nu evaluated outside balls}, $I_{x/|x|^2}^{\epsilon} \subseteq B_{|x|^{-1}}^{\comp}$ implies
\begin{equation}	\label{eq:nu estimate}
\nu(I_{x/|x|^2}^{\epsilon}) \leq \nu(B_{|x|^{-1}}^{\comp}) = W\mathfrak{h}(|x|) |x|^{\alpha}	\quad	\text{a.s.}
\end{equation}
Consequently, with $h^* \defeq \sup_{t > 0} \mathfrak{h}(t)$, we have
\begin{equation}	\label{eq:uniform bound on Psi}
\sup_{x \in \R^d \setminus \{0\}} \Psi_\epsilon(x)	\leq W h^*	\quad	\text{a.s.}
\end{equation}
Further, by \eqref{eq:nu evaluation},
\begin{equation}	\label{eq:Psi_eps}
\Psi_\epsilon(x)	~=~	\sum_{|v|=n} \norm{L(v)}^{\alpha} [\Psi_\epsilon]_v(L(v)^{\transp} x)	\quad	\text{for all } x \in \R^d \setminus \{0\} \text{ a.s.}
\end{equation}
For $x \in \R^d \setminus \{0\}$, define $\psi_{\epsilon,x}: \U \to \Rp$ via
\begin{equation*}
\psi_{\epsilon,x}(g) ~=~ \E[\Psi_\epsilon(g^{\transp} x)],	\quad	g \in \U.
\end{equation*}
$\psi_{\epsilon,x}$ is measurable (since $(\omega,g) \mapsto \nu(g^{-1}B)$ is product-measurable),
nonnegative and bounded (due to \eqref{eq:uniform bound on Psi} and the fact that $\E[W]=1$).
Further, as a consequence of \eqref{eq:Psi_eps} and the change of measure \eqref{eq:change of measure},
we have
\begin{align*}
\psi_{\epsilon,x}(g)
~=~ \E[\Psi_\epsilon(g^{\transp} x)]
~&=~ \E \bigg[\sum_{|v|=n} \norm{L(v)}^{\alpha} [\Psi_\epsilon]_v(L(v)^{\transp} g^\transp x)\bigg]	\\
&=~ \E \bigg[\sum_{|v|=n} \norm{L(v)}^{\alpha} \psi_{\epsilon,x}(gL(v))\bigg]
~=~ \E[\psi_{\epsilon,x}(gL_n)]
\end{align*}
for all $g \in \U$.
In other words, $\psi_{\epsilon,x}$ solves a Choquet-Deny functional equation on the
(possibly non-Abelian) topological group $\U$.
From the Choquet-Deny lemma (see Lemma \ref{Lem:Choquet-Deny} in the appendix)
we infer that $\psi_{\epsilon,x}$ is $\Prob(L_n \in \cdot)$-a.s.\ constant on $\U$.
Consequently, there is a constant $c_{\epsilon,x} \geq 0$ such that, with $U=\{\psi_{\epsilon,x} = c_{\epsilon,x}\}$,
$\Prob(L_n \in U) = 1$ for all $n \in \N_0$.
Notice that this implies
\begin{equation*}
\E\bigg[\sum_{|v|=n} \norm{L(v)}^{\alpha} \1_U(L(v))\bigg]	~=~	\Prob(L_n \in U) = 1
\end{equation*}
and hence $\Prob(L(v) \in U \text{ for all } |v|=n \text{ with } \norm{L(v)}>0) = 1$.
This together with the martingale convergence theorem and \eqref{eq:Psi_eps} yields
\begin{eqnarray*}
\Psi_\epsilon(x)
& = &
\lim_{n \to \infty} \E [\Psi_\epsilon(x)|\F_n]
~=~	\lim_{n \to \infty} \sum_{|v|=n} \norm{L(v)}^{\alpha} \E [[\Psi_\epsilon]_v(L(v)^{\transp} x) | \F_n]	\\
& = &
\lim_{n \to \infty} \sum_{|v|=n} \norm{L(v)}^{\alpha} \psi_{\epsilon,x}(L(v))
~=~	W c_{\epsilon,x}	\quad	\text{a.s.}
\end{eqnarray*}
In terms of $\nu(I_{x/|x|^2}^{\epsilon})$, this reads
\begin{equation*}
\nu(I_{x/|x|^2}^{\epsilon})	~=~	W c_{\epsilon,x} |x|^{\alpha}	\quad	\text{a.s.}
\end{equation*}
Setting $\bar \nu(\cdot) \defeq \E[\nu(\cdot)]$ and taking expectations in the above equation gives
$\bar \nu(I_{x/|x|^2}^{\epsilon}) = c_{\epsilon,x} |x|^{\alpha}$
and hence
\begin{equation}	
\nu(I_{x/|x|^2}^{\epsilon})	~=~	W \bar \nu(I_{x/|x|^2}^{\epsilon}) \quad	\text{a.s.}
\end{equation}

\Step 4: Using \eqref{eq:I_x^epsilon transformed}, we conclude that, for every $g \in \U$,
\begin{equation*}\label{eq:nu of  I_x^epsilon transformed}
\nu(g^{-1} I_{x/|x|^2}^\epsilon)
= \nu(I_{g^\transp x/|g^\transp x|^2}^\epsilon)
= W \psi_{\epsilon,x}(g^\transp) |g^\transp x|^\alpha
= \norm{g}^{\alpha} \nu(I_{x/\abs{x}^2}^\epsilon) \quad \text{a.s.}
\end{equation*}
and hence, upon taking expectation, we arrive at the transformation formula \eqref{eq:bar nu (U,alpha)-invariant} for $\bar \nu$,
valid for all $g \in \U$ and sets $I_{x/|x|^2}^\epsilon$.
\eqref{eq:bar nu (U,alpha)-invariant} extends to finite intersections of the form
$I = I_{x_1/|x_1|^2}^{\epsilon_1} \cap \ldots \cap I_{x_n/|x_n|^2}^{\epsilon_n}$
for $\vec x = (x_1,\ldots,x_n) \in (\R^d \setminus \{0\})^n$ and $\vec \epsilon = (\epsilon_1,\ldots,\epsilon_n) \in \Rp^n$.
To be more precise, first notice that in such an intersection, one can assume that $|x_1|=\ldots=|x_n|$
(otherwise, let $r \defeq \min_{j=1}^n |x_j|$ and notice that $I$ does not change when $x_j$ is replaced by $r x_j/|x_j|$, $j=1,\ldots,n$).
\eqref{eq:nu estimate} then carries over with $I_{x/|x|^2}^{\epsilon}$ replaced by $I$ and $|x|$ replaced by $|x_1|$.
The obvious counterparts for $I$ of \eqref{eq:I_x^epsilon transformed} and \eqref{eq:nu evaluation} hold true
and one can define $\Psi_{\vec \epsilon}(\vec x) \defeq \nu(I) |x_1|^{-\alpha}$.
The remaining arguments apply almost without changes and give
\begin{align}
\nu(I)	~&=~	W \bar \nu(I) \quad	\text{a.s.}		\label{eq:nu=W times bar nu}
\end{align}
and the counterpart of \eqref{eq:bar nu (U,alpha)-invariant}.
Now notice that the finite intersections of sets $I_{x/|x|^2}^{\epsilon}$ with $x \in \Q^d \setminus \{0\}$ and rational $\epsilon > 0$
form a countable generator of the Borel $\sigma$-field on $\R^d$
and on a set of probability one, $\nu(I) = W \bar \nu(I)$ for all sets from this generator.
It follows by an application of the uniqueness theorem from measure theory that $\nu(B)= W \bar \nu(B)$ for all Borel sets $B \subseteq \R^d \setminus \{0\}$.
Analogously, \eqref{eq:bar nu (U,alpha)-invariant} extends to all Borel measurable $B \subseteq \R^d \setminus \{0\}$. \\

\Step 5: It remains to show that $\bar \nu$ is a L\'evy measure and that $\bar \nu = 0$ if $\alpha \geq 2$.
For the former, notice that $\bar \nu$ is a measure on $\R^d \setminus \{0\}$ and that
\begin{equation*}
\int (|x|^2 \wedge 1) \, \bar \nu(\dx) < \infty
\quad	\text{iff}	\quad
\int (|x|^2 \wedge 1) \, \nu(\dx) < \infty		\text{ a.s.}
\end{equation*}
by \eqref{eq:nu=W times bar nu}. For the proof of the second claim, notice that taking expectations in \eqref{eq:nu evaluated outside balls}
gives $\bar \nu(B_t^\comp) = \mathfrak{h}(t^{-1}) t^{-\alpha}$ for all $t>0$ for a multiplicatively $\G$-periodic function $\mathfrak{h}$
such that $\mathfrak{h}(t) t^{\alpha}$ is nondecreasing in $t$.
Because of these properties we have $h_* \defeq \inf_{t > 0} \mathfrak{h}(t) = 0$ iff $h \equiv 0$.
Now assume that $\alpha \geq 2$. We show that then $h_* = 0$.
Indeed, since $\bar \nu$ is a L\'evy measure, integration by parts gives
\begin{eqnarray*}
\infty
& > &	\int_{\{|x|^2 \leq 1\}} |x|^2 \, \bar \nu(\dx)
~=~
\int_0^1 2t \bar \nu(B_t^\comp \setminus B_1^\comp) \, \dt	\\
& = &
\int_0^1 2t \mathfrak{h}(t^{-1}) t^{-\alpha} \dt - \mathfrak{h}(1)
~\geq~	2h_* \int_0^1 t^{1-\alpha} \dt - \mathfrak{h}(1)
\end{eqnarray*}
which implies that $h_*=0$. \qed
\end{proof}

Any L\'evy measure $\bar \nu$ satisfying the invariance property \eqref{eq:bar nu (U,alpha)-invariant}
can be factorized into a radial and spherical part, similar to the decomposition valid for L\'evy measures of stable laws.
The details are given in Proposition \ref{Prop:structure of bar nu} in Section \ref{sec:invariant matrices and measures} below.

\subsection{Determining $\mathbf{\Sigma}$}	\label{subsec:Sigma} 

Now we use the identity
\begin{equation*}	\tag{\ref{eq:FE Sigma}}
\mathbf{\Sigma}	~=~	\sum_{|v|=n} L(v) [\mathbf{\Sigma}]_v L(v)^{\transp}	\quad	\text{a.s.,}
\end{equation*}
to determine $\Sigma$.

\begin{lemma}	\label{Lem:Sigma=W times deterministic Sigma}
Assume that \eqref{eq:A1}--\eqref{eq:A3} hold
and let $\mathbf{\Sigma}$ be a $\F$-measurable random covariance matrix. Then $\mathbf{\Sigma}$ satisfies \eqref{eq:FE Sigma}
if and only if
$\mathbf{\Sigma} = W \Sigma$ a.s.\ for a deterministic covariance matrix $\Sigma$ satisfying
\begin{equation}	\label{eq:Sigma=oSigmao^T}
\Sigma = o \Sigma o^{\transp}
\end{equation}
for all $o \in \O$.
Further, $\Sigma = 0$ if $\alpha \not = 2$.
\end{lemma}
\begin{proof}
Our first observation is that $\trace(\mathbf{\Sigma})$, the trace of $\mathbf{\Sigma}$,
is a nonnegative endogenous fixed point of the one-di\-men\-sional smoothing transform with weights $(\norm{T_j}^2)_{j \geq 1}$.
To see this, notice that it follows from \eqref{eq:FE Sigma} and the fact that the trace is invariant under conjugations with orthogonal transformations
that
\begin{eqnarray*}
\trace(\mathbf{\Sigma})
& = &
\trace \bigg(\sum_{|v|=n} L(v) [\mathbf{\Sigma}]_v L(v)^{\transp} \bigg)
~=~	\sum_{|v|=n} \norm{L(v)}^2 \trace(O(v) [\mathbf{\Sigma}]_v O(v)^{\transp})	\\
& = &
\sum_{|v|=n} \norm{L(v)}^2 [\trace(\mathbf{\Sigma})]_v
\quad	\text{a.s.}
\end{eqnarray*}
From Theorem 6.2 in \cite{Alsmeyer+Biggins+Meiners:2012}, we conclude that $\trace(\mathbf{\Sigma}) = cW$ for some constant $c \geq 0$
if $\alpha = 2$, and $\trace(\mathbf{\Sigma}) = 0$ a.s., otherwise.

Now recall that all eigenvalues of a covariance matrix, i.e., a positive semi-definite, symmetric matrix are nonnegative
and that there is a basis of $\R^d$ consisting only of eigenvectors of that matrix.
Further, the trace of the matrix equals the sum of its eigenvalues.

In the case $\alpha \not = 2$, this implies that all eigenvalues of $\mathbf{\Sigma}$ equal $0$ a.s.,
which means that the matrix itself vanishes a.s.
This, in turn, proves the assertion in the case where $\alpha \not = 2$.

In the case $\alpha = 2$, we can conclude that $\trace(\mathbf{\Sigma})$ has finite expectation since $\E[W]=1$.
This implies that $\mathbf{\Sigma}_{ij}$ is integrable for $i,j=1,\ldots,d$ since
\begin{equation*}
|\mathbf{\Sigma}_{ij}|
= |e_i \mathbf{\Sigma} e_j|
\leq \norm{\mathbf{\Sigma}}
\leq \trace(\mathbf{\Sigma})
\end{equation*}
where for the last inequality, we have used that the norm of $\mathbf{\Sigma}$ is
the largest eigenvalue of $\mathbf{\Sigma}$ while the trace of $\mathbf{\Sigma}$ is the sum of its eigenvalues.
Consequently, $\Sigma \defeq \E[\mathbf{\Sigma}]$ is a symmetric $d \times d$ matrix.
We claim that $\Sigma = o \Sigma o^{\transp}$ for all $o \in \O$.
To prove this, notice that since $\Sigma$ is symmetric and positive semi-definite,
it has $n \leq d$ distinct real eigenvalues $\lambda_1> \ldots > \lambda_n \geq 0$
such that $d_1+\ldots+d_n=d$ where $d_j$ denotes the dimension of the eigenspace $E_{\lambda_j}(\Sigma)$
corresponding to the eigenvalue $\lambda_j$, $j=1,\ldots,n$.
For arbitrary $x \in E_{\lambda_1}(\Sigma)$,
first apply \eqref{eq:FE Sigma} for $n=1$ and then condition w.r.t.~$\F_1$ to obtain
\begin{equation}	\label{eq:x^T Sigma x}
\lambda_1 |x|^2
~=~	x^{\transp} \Sigma x	~=~	\E  \bigg[\sum_{j=1}^N \norm{T_j}^2 x^{\transp} O_j \Sigma O_j^{\transp} x\bigg].
\end{equation}
On the other hand, $x^{\transp} O_j \Sigma O_j^{\transp} x \leq \lambda_1 |x|^2$ for all $j=1,\ldots,N$ a.s.\ since
$\lambda_1$ is the largest eigenvalue of $\Sigma$.
Using $m(2)=1$ and \eqref{eq:x^T Sigma x}, we conclude that $x^{\transp} O_j \Sigma O_j^{\transp} x = \lambda_1 |x|^2$ for all $j=1,\ldots,N$ a.s.
Since $x \in E_{\lambda_1}(\Sigma)$ was arbitrary, we infer that, a.s.,
the restriction of $O_j$ to $E_{\lambda_1}(\Sigma)$ is an automorphism of $E_{\lambda_1}(\Sigma)$, $j=1,\ldots,N$.
Using this fact, the above argument can be repeated consecutively for the remaining eigenvalues of $\Sigma$ (in decreasing order)
to conclude that, a.s., each $O_j$, $j=1,\ldots,N$ is an automorphism of $E_{\lambda_k}(\Sigma)$, $k=2,\ldots,n$.
Since every vector $x \in \R^d$ can be written as a linear combination of normed eigenvectors of the eigenvalues $\lambda_1,\ldots,\lambda_n$,
we conclude that $\Sigma x = O_j \Sigma O_j^{\transp} x$ for $j=1,\ldots,N$ a.s.\ and
hence $\Sigma = O_j \Sigma O_j^{\transp}$ for $j=1,\ldots,N$ a.s.
Standard arguments then yield that $\Sigma = o \Sigma o^{\transp}$ for all $o \in \O$.
Using this, the martingale convergence theorem, \eqref{eq:FE Sigma} and the convergence $W_n \to W$ a.s.,
we conclude that
\begin{eqnarray*}
\mathbf{\Sigma}
& = &
\lim_{n \to \infty} \E [\mathbf{\Sigma} | \F_n]	~=~	\lim_{n \to \infty} \sum_{|v|=n} L(v) \Sigma L(v)^{\transp}	\\
& = &
\lim_{n \to \infty} \sum_{|v|=n} \norm{L(v)}^2 O(v) \Sigma O(v)^{\transp}
~=~	W \Sigma	\quad	\text{a.s.}
\end{eqnarray*}
 \qed
\end{proof}

\subsection{The proof of Proposition \ref{Prop:convergence Z_n w}}

In contrast to Eqs. \eqref{eq:FE nu} and \eqref{eq:FE Sigma}, which allowed to determine $\nu$ and $\Sigma$, there is in general no such equation for $W'$ due to additional contributions coming from the L\'evy measure, see Eq. \eqref{eq:FE Psi}. Therefore, we have to postpone the identification of $W'$ until the next section, and will prove Proposition \ref{Prop:convergence Z_n w} first. This is closely related to determining $W'$:
Indeed, suppose that $Z$ is a solution to \eqref{eq:FP of ST hom a.s.},
in particular, $Z$ is $\F$-measurable.
Iteration of \eqref{eq:FP of ST hom a.s.} yields
\begin{equation}	\label{eq:endogenous hom FP}
Z	~=~	\sum_{|v|=n} L(v) [Z]_v	\quad	\text{a.s.}
\end{equation}
for all $n \in \N_0$. Then consider the multiplicative martingales $(M_n(x))_{n \in \N_0}$ associated with $Z$ (and its characteristic function $\phi$).
As in the proof of \cite[Proposition 4.17]{Alsmeyer+Meiners:2013}, we infer from the martingale convergence theorem,
\begin{align*}
M(x)
&=	\lim_{n \to \infty} \prod_{|v|=n} \phi(L(v)^\transp x)
=	\lim_{n \to \infty} \E\bigg[ \exp\bigg(\sum_{|v|=n} \imag \scalar{L(v)^\transp x, [Z]_v} \bigg) \,\Big|\, \F_n\bigg]	\\
&=	\lim_{n \to \infty} \E\bigg[ \exp\bigg(\imag \Big\langle x,\sum_{|v|=n} L(v) [Z]_v \Big\rangle \bigg) \,\Big|\, \F_n\bigg]
=	\lim_{n \to \infty} \E[\exp(\imag \scalar{x,Z}) | \F_n]	\\
&=	\exp(\imag \scalar{x,Z})	\quad	\text{a.s.}
\end{align*}
Hence, $Z = W'$ a.s.\ in \eqref{eq:Psi} with $\Psi$ denoting the characteristic exponent of $M$ (in this case, $\Sigma$ and $\nu$ vanish a.s.).
Therefore, by Proposition \ref{Prop:Multiplicative martingales},
$Z$ can be written in the form
\begin{equation}	\label{eq:Z alternative}
Z	~=~	\lim_{n \to \infty} \sum_{|v|=n} \E[L(v) [Z]_v; |L(v)[Z]_v| \leq 1 \mid \F_n]	\quad	\text{a.s.}
\end{equation}
Further, \eqref{eq:tail F endogenous} gives
\begin{equation}	\label{eq:tail of Z}
\Prob(|Z|>t)	~=~	o(t^{-\alpha})	\text{ as } t \to \infty.
\end{equation}
This and the estimate \eqref{eq:L_beta asymptotics endogenous} for $L_\beta$ (see Eq. \eqref{eq:L_beta})
from Section \ref{subsec:Tail estimates}
is everything we need to determine $Z$ when $\alpha \not = 1$.

\begin{proof}[Proof of Proposition \ref{Prop:convergence Z_n w} in the case $\alpha \not = 1$] 
(a)
Let $0<\alpha<1$. Then, using \eqref{eq:Z alternative}, \eqref{eq:L_beta asymptotics endogenous}, \eqref{eq:Biggins:1998}
and $W_n \to W$ a.s., we infer
\begin{eqnarray*}
|Z|
& = &
\lim_{n \to \infty} \bigg|\sum_{|v|=n} \E\big[L(v) [Z]_v; |L(v)[Z]_v| \leq 1 \mid \F_n\big]\bigg|	\\
& \leq &
\limsup_{n \to \infty} \sum_{|v|=n} \norm{L(v)} \E\big[|[Z]_v|; |[Z]_v| \leq \norm{L(v)}^{-1} \mid \F_n\big]	\\
& = &
\limsup_{n \to \infty} \sum_{|v|=n} \norm{L(v)} L_1(\norm{L(v)}^{-1})
~=~	0	\quad	\text{a.s.}
\end{eqnarray*}

(c)  Suppose that $1 < \alpha <2$, that $Z$ satisfies \eqref{eq:endogenous hom FP}
and that $\Prob(Z \not = 0) > 0$.
Notice that \eqref{eq:tail of Z} implies that $w \defeq \E[Z]$ is finite, moreover, $Z \in \mathcal{L}^s$ for all $s < \alpha$.
By standard martingale theory and \eqref{eq:endogenous hom FP}, for any $1 \leq s < \alpha$,
\begin{eqnarray*}
Z
& = &
\lim_{n \to \infty} \E[Z|\F_n]
~=~ \lim_{n \to \infty}	\E \bigg[ \sum_{|v|=n} L(v) [Z]_v \,\Big|\, \F_n\bigg]	\\
& = &
\lim_{n \to \infty} \sum_{|v|=n} L(v) w
~=~ \lim_{n \to \infty} Z_n w	\quad	\text{a.s.\ and in } \L^s.
\end{eqnarray*}
Hence, $Z=Z^w$ and thus $w \not = 0$.
Further, taking expectations in \eqref{eq:endogenous hom FP}, we obtain
\begin{equation*}
w	~=~	\E[Z]	~=~	\E\bigg[\sum_{j \geq 1} T_j [Z]_j \bigg]	~=~	\E[Z_1] w,
\end{equation*}
i.e., $w$ is an eigenvector corresponding to the eigenvalue $1$ of $\E[Z_1]$.

Conversely, suppose that $w = \E[Z]$ exists and is an eigenvector corresponding to the eigenvalue $1$ of $\E[Z_1]$
and that $(Z_nw)_{n \in \N_0}$ is uniformly integrable.
Then $Z_nw \to Z^w$ a.s.\ as $n \to \infty$ by the martingale convergence theorem and $\E[Z^w] = w$.
Since $w$ is an eigenvector, we have $w \not = 0$ and thus $\Prob(Z^w \not = 0) > 0$.
One can then check that $Z^w$ satisfies \eqref{eq:endogenous hom FP}.

To finish the proof in the case $1<\alpha<2$,	
we have to show that $\E[|Z_1w|^{\beta}]<\infty$ and $m(\beta)<1$ is sufficient for the $\L^\beta$-boundedness of $(Z_nw)_{n \in \N_0}$.
This, however, is a standard application of the Topchi{\u\i}-Vatutin inequality for martingales.
For details, we refer to \cite[p.\,182]{Alsmeyer+Kuhlbusch:2010} and \cite{Roesler+Topchii+Vatutin:2000}.	

(d) Next, let $\alpha \geq 2$. If $w \in \R^d$ is an eigenvector corresponding to the eigenvalue $1$ of $Z_1$ a.s.,
then $Z_n w = w$ a.s.\ for all $n \in \N_0$ and $Z^w = Z$ solves \eqref{eq:endogenous hom FP}.
For the converse implication, assume that $Z$ solves \eqref{eq:endogenous hom FP} and that $\Prob(Z\not=0)>0$.
By \eqref{eq:tail of Z}, $Z \in \L^s$ for all $s \in [1,\alpha)$. Pick some $s \in (1,2)$ if $\alpha = 2$ and $s=2$ if $\alpha > 2$
and use the lower bound in the Burkholder-Davis-Gundy inequality \cite[Theorem 11.3.1]{Chow+Teicher:1997}
to derive lower bounds for $\E[|Z-w|^s]$ as in the proof of Theorem 2.3 in \cite{Iksanov+Meiners:2015a}.
It follows that $\E[|Z-w|^s] = \infty$ unless $Z_1w=w$ a.s.
We refrain from providing more details. \qed
\end{proof}

The proof for the case $\alpha=1$ is much more involved, some parts of it will be deferred to the Appendix.
As a preparation, we show that it constitutes no loss of generality
to assume the stronger assumption \eqref{eq:A5} instead of \eqref{eq:A4}.

\begin{lemma}	\label{Lem:embedded ladder process}
Suppose that \eqref{eq:A1}--\eqref{eq:A3} and either \eqref{eq:A4} or \eqref{eq:A4'} hold.
Let $\phi$ be a solution to \eqref{eq:FE of ST hom}.
Then there is a weight sequence $\bT'\defeq (T_j')_{j=1}^{N'}$ with $N' < \infty$ a.s.\ such that \eqref{eq:A1}--\eqref{eq:A4}
resp. \eqref{eq:A1}--\eqref{eq:A4'} persist to hold for $\bT'$,
$\phi$ is a solution to equation \eqref{eq:FE of ST hom} associated with $\bT'$ as well
and the martingale limit $W$ defined in \eqref{eq:Biggins' martingale} is the same for $\bT$ and $\bT'$.
Moreover, $\norm{T_j'} < 1$ a.s. for all $1 \leq j \le N'$ and the group $\G$ generated by $\bT'$
is the same as the one generated by $\bT$. If $\bT$ satisfies \eqref{eq:A4'}, then also the group $\O$ remains the same.
If $\bT$ satisfies \eqref{eq:A4}, then \eqref{eq:A5} holds for $\bT'$.
\end{lemma}
\begin{proof}
We first deal with the case that \eqref{eq:A4'} holds. Set $\bT'\defeq (T(v))_{v \in \mathcal{C}(0)}$
(recall that $\norm{T(v)} < 1$ a.s.\ for all $v \in \mathcal{C}(0)$).
Then \cite[Proposition 3.7]{Iksanov+Meiners:2015b} states that if $(T(v))_{|v|=1}$
satisfies the assumptions \eqref{eq:A1}--\eqref{eq:A4'}, then so does the family $(T(v))_{v \in \mathcal{C}(0)}$.
That $W$ is the same for the sequences $\bT$ and $\bT'$ follows from an application of \cite[Theorem 9]{Kyprianou:2000}.
The additional point which we have to take care of here is that we need to know
that the closed multiplicative subgroup of $\Orth$ generated by the $O(v)$, $v \in \mathcal{C}(0)$,
which we denote by $\O^>$, equals $\O$.
Clearly, $\O^> \subseteq \O$.
Conversely, let $t > 0$ and $o \in \Orth$ be such that $to$ is in the support of $T_j$ for some $j \in \N$.
If we can show that $o \in \O^>$, then it follows that $\O \subseteq \O^>$.
By \eqref{eq:A2}, there are $t'\in (0,1)$ and $o' \in \Orth$ such that $t'o'$ is in the support of some $T_k$ for some $k \in \N$.
Then $o' \in \O^>$ by the definition of $\mathcal{C}(0)$ and $\O^>$.
For the minimal $m \in \N_0$ such that $t(t')^m < 1$, we have $o(o')^m \in \O^>$,
which implies $o \in \O^>$.

Now we turn to the case that \eqref{eq:A4} holds.
Here, we obtain $\bT'$ in two successive steps.
We start by considering as before  the sequence $\tilde{\bT} \defeq (T(v))_{v \in \mathcal{C}(0)}$.
As above, it follows that
(referring to the proof of \cite[Proposition 3.7]{Iksanov+Meiners:2015b} for the moment assumptions in \eqref{eq:A4})
that \eqref{eq:A1}--\eqref{eq:A3}
as well as the spread-out assumption and the moment assumptions in \eqref{eq:A4} carry over.
Also, as above, $W$ is the same for $\bT$ and $\tilde{\bT}$.
That $\eta\defeq \E[\sum_{v \in \mathcal{C}(0)} \norm{T(v)}^{\alpha} \delta_{T(v)}(\cdot)]$ is spread-out
can be obtained by arguments similar to those given in the proof of \cite[Lemma 1]{Araman+Glynn:2006}.
Hence there is an $n \in \N$ such that the $n$-fold convolution $\eta^{*n}$ has a component
which is continuous w.r.t.\ the Haar measure on $\Sim$.
Denote the corresponding density by $f$ and assume without loss of generality that $f$ has compact support
and that $c \leq f \leq d$
for constants $0<c<d$.
Then $\eta^{*2n}$ has a component with density $f^{*2}$, which is continuous \cite[p.\,295]{Hewitt+Ross:1963}.
By the Fubini formula for the Haar measure on $\Sim = \Rp \times \Orth$
\cite[Proposition 1.5.5]{Deitmar+Echterhoff:2009}, 
\begin{equation}	\label{eq:etan}
\eta^{*2n} \geq \tilde{\gamma} {H_{\Orth}}_{|D} \otimes {H_{\Rp}}_{|I}
\end{equation}
for some $\tilde{\gamma}>0$ and open sets $D \subseteq \Orth$, $I \subseteq \Rp$.
Replacing $2n$ by $4n$ if neccessary, we may assume that $D \subseteq \SOd$.
Since the latter group is compact and connected, we may invoke (the proof of) \cite[Theorem 3]{Bhattacharya:1972}
which gives that there is a $k \in \N$ such that $({H_{\Orth}}_{|D})^{*k} \ge \epsilon H_{\SOd}$ for some $\epsilon >0$.
Then \eqref{eq:etan} yields that
\begin{equation*}
\mu(A,B) \defeq  \eta^{*(2kn)}(A \times e^{-B}) , \qquad A \subseteq \Orth, B \subseteq \R \text{ measurable}
\end{equation*}
satisfies $(M)$. On the other hand, 
\begin{equation*}
\mu = \E\bigg[\sum_{|v|=2kn} \big\|\tilde{L}(v)\big\|^{\alpha} \delta_{(\tilde{O}(v), - \log (\|\tilde{T}(v)\|))}(\cdot) \bigg],
\end{equation*}
where $(\tilde{L}(v))_{v \in \V}$ is the weighted branching process associated with $\tilde{\bT}$.
It is readily checked that assumptions \eqref{eq:A1}--\eqref{eq:A3}
and the moment assumptions of \eqref{eq:A4} persist to hold for $\bT'\defeq (\tilde{L}(v))_{|v|=2kn}$
and that also the martingale limit $W$ is the same for $\bT$ and $\tilde{\bT}$ and thus the assertion follows.	\qed
\end{proof}

\begin{remark} \label{Rem:U under A4}
Under \eqref{eq:A5}, it holds that $\U=\R_> \times \SOd$ or $\U=\R_> \times \Orth$.
Since the group $\U'$ generated by $\bT'$ will always be a subgroup of the one generated by $\bT$, it follows that under \eqref{eq:A4},
these are the only two possible cases for $\U$. Note that $\U=\R_> \times \Orth$ while $\U'=\R_> \times \SOd$ is possible
(for example, if $\norm{T_j} < 1$ and $\det(T_j)=-1$ for all $j=1, \dots, N$).
This is not a problem since both the $(\U,\alpha)$- and $(\U',\alpha)$-stable laws are just the $d$-dimensional rotational invariant $\alpha$-stable laws.
\end{remark}

\begin{proof}[Proof of Proposition \ref{Prop:convergence Z_n w} in the case $\alpha=1$]
By Lemma \ref{Lem:embedded ladder process}, we may assume throughout the proof, that $\norm{T_j} \le 1$ a.s.\ for all $1 \le j \le N$,
and that  \eqref{eq:A5} holds if \eqref{eq:A4} holds, i.e., we work with the sequence $\bT'$ instead of $\bT$ (but drop the superscript).

We begin with the converse implication of (b)
and suppose that $w \in \R^d$ is an eigenvector corresponding to the eigenvalue $1$ of $\E[Z_1]$.
Then
\begin{equation*}
|w| = \big|\E[Z_1] w\big| \leq \big\|\E[Z_1] \big\| \cdot |w| \leq \E\big[\norm{Z_1}\big] \cdot |w|	\leq \E \bigg[\sum_{j \geq 1} \norm{T_j}\bigg] \cdot |w|	= |w|.
\end{equation*}
Hence, equality must hold in this chain of inequalities
and hence $T_j w = \norm{T_j}w$ for all $j \in \N$ a.s.
Consequently, $Z_nw = W_n w \to Ww$ a.s.\ as $n \to \infty$ and $Ww$ satisfies \eqref{eq:endogenous hom FP}.

We are left with proving the direct implication in (b). Thus, assume that $Z$ is $\F$-measurable and satisfies \eqref{eq:endogenous hom FP}.
First, we write $\R^d = V_+ + V_+^\perp$ where $V_+ = \{x \in \R^d: ox=x \text{ for all } o \in \O\}$.
Both, $ V_+$ and $ V_+^\perp$ are invariant subspaces of $\R^d$ for every $o \in \O$.
Write $Z^{ V_+}$ and $Z^{ V_+^\perp}$ for the orthogonal projection of $Z$ onto $ V_+$ and $ V_+^\perp$, respectively.
We have
\begin{eqnarray*}
Z^{V_+} + Z^{V_+^\perp}
& = &
Z	~=~	\sum_{|v|=n} L(v) [Z]_v	~=~	\sum_{|v|=n} L(v) [Z^{V_+}]_v + \sum_{|v|=n} L(v) [Z^{V_+^\perp}]_v	\\
& = &
\sum_{|v|=n} \norm{L(v)} [Z^{V_+}]_v + \sum_{|v|=n} L(v) [Z^{V_+^\perp}]_v.
\end{eqnarray*}
From linear independence (the left sum is in $V_+$, the right sum is in $V_+^\perp$),
we conclude that
\begin{align*}
Z^{V_+}	~&=~	\sum_{|v|=n} \norm{L(v)} [Z^{V_+}]_v	\quad	\text{a.s.\ for all } n \in \N_0	\\
\text{and}	\quad
Z^{V_+^\perp}	~&=~	\sum_{|v|=n} L(v) [Z^{V_+^\perp}]_v	\quad	\text{a.s.\ for all } n \in \N_0.
\end{align*}
The equation for $Z^{V_+}$ can be reduced to one-dimensional equations
and it follows from \cite[Theorem 4.13]{Alsmeyer+Meiners:2013} that $Z^{V_+} = Ww$ for some $w \in V_+$.
Taking expectations, we conclude that $w=\E[Z^{V_+}]$.
Now notice that $V_+ = E_1(\E[Z_1])$ and thus $w$ is an eigenvector of $\E[Z_1]$ or zero.

Observe that if $\SOd \subseteq \O$, $w$ has to be zero: Using the change of measure \eqref{eq:change of measure} with $f_k(L_1)\defeq(O_1 w)_k$ for $1 \le k \le d$, we infer that
\begin{equation*}
w = \E \bigg[\sum_{j \ge 1} \norm{T_j} (O_j w)\bigg] = \E [O_1 w],
\end{equation*}
which implies $w = O_1 w$ a.s.
Hence, $\SOd \subseteq \O$ implies $w=0$.
Referring to Remark \ref{Rem:U under A4},
we see that \eqref{eq:A4} particularly implies $\SOd \subseteq \O$.

It remains to show that $Z^{V_+^\perp}=0$ a.s.
We drop the ${V_+^\perp}$ in the superscript and write $Z$ for $Z^{V_+^\perp}$.
We will use a variant of \eqref{eq:Z alternative}. 
To formulate it, consider the coming generation at time $t \geq 0$, which is  $\mathcal{C}(t)$, defined by Eq. \eqref{eq:C(t)}.
The idea is to switch from genealogical generations in the branching process
to particles living roughly at the same time (with $S(v)$ interpreted as the time of birth of particle $v$).
The gain is that there is a better control over the birth times $S(v)$ when $v$ ranges over $\mathcal{C}(t)$
than when it ranges of $\{|v|=n\}$.
Instead of considering the multiplicative martingales $(M_n(x))_{n \in \N_0}$ for $Z$
(defined via the characteristic function $\phi$ of $Z$),
we consider
\begin{equation*}
M_{\mathcal{C}(t)}(x)	~=~	\prod_{v \in \mathcal{C}(t)} \phi(L(v)^\transp x),	\quad	x \in \R^d.
\end{equation*}
It can be checked along the lines of \cite[Lemma 8.7(b)]{Alsmeyer+Biggins+Meiners:2012}
and \cite[Lemma 4.4]{Alsmeyer+Meiners:2013} that $M_{\mathcal{C}(t)}(x) \to M(x)$ for all $x \in \R^d$
with $M(x) = \lim_{n \to \infty} M_n(x)$ a.s.
Arguing as in the proof of \cite[Lemma 3.6]{Iksanov+Meiners:2015a} gives
\begin{equation}	\label{eq:Z^h via M_C(t)}
Z^h	~=~	\lim_{t \to \infty} \sum_{v \in \mathcal{C}(t)} \E[L(v) [Z]_v; |L(v)[Z]_v| \leq h \mid \F_{\mathcal{C}(t)}]	\quad	\text{a.s.}
\end{equation}
for every $h>0$ with $\nu(\{|x|=h\}) = 0$ a.s.
Here, $\F_{\mathcal{C}(t)}$ is the $\sigma$-field that makes everything measurable in the weighted branching model defined by $(\bC,\bT)$
that is born up to and including time $t$, see the proof of Lemma 8.7 in \cite{Alsmeyer+Biggins+Meiners:2012} for a rigorous definition.
In the given situation, $\nu=0$ a.s.\ and hence $Z=Z^h$ and $h>0$ can be chosen arbitrarily, hence $h=1$ is the most convenient choice.
The key equation for us thus is
\begin{align}	\label{eq:Z via M_C(t)}
Z
& =
\lim_{t \to \infty} \sum_{v \in \mathcal{C}(t)} L(v) I(\norm{L(v)}^{-1})	\\
& =
\lim_{t \to \infty} \bigg( \sum_{v \in \mathcal{C}(t)} L(v) I(e^t) + \sum_{v \in \mathcal{C}(t)} L(v) \big(I(\norm{L(v)}^{-1})	- I(e^t)\big)\bigg)
\text{ a.s.}	\notag
\end{align}
where $I(t) \defeq \E[Z; |Z| \leq t]$, $t \geq 0$.
For $0 \leq s < t$, we have
\begin{eqnarray*}
|I(t)-I(s)|
& \leq &
\E[|Z|; s < |Z| \leq t]	\\
& = &
\int_s^t \Prob(|Z|>x) \, \dx - t \Prob(|Z|>t) + s \Prob(|Z|>s)	\\
& \leq &
\int_s^t \Prob(|Z|>x) \, \dx + s \Prob(|Z|>s).
\end{eqnarray*}
Using this for the last summand in the last line of \eqref{eq:Z via M_C(t)}, we infer
\begin{align}
\bigg|&\sum_{v \in \mathcal{C}(t)} L(v) \big(I(\norm{L(v)}^{-1})	- I(e^t)\big)\bigg|	\nonumber \\
&~=~	\sum_{v \in \mathcal{C}(t)} \norm{L(v)} \int_{e^t}^{\norm{L(v)}^{-1}} 
\Prob(|Z|>x) \, \dx
+ \sum_{v \in \mathcal{C}(t)} \norm{L(v)} e^t \Prob(|Z|>e^t). \label{eq:decomposition I}
\end{align}
The second sum tends to $0$ a.s.\ as $t \to \infty$ by \eqref{eq:tail of Z} and the fact
that $\sum_{v \in \mathcal{C}(t)} \norm{L(v)} = \E[W | \F_{\mathcal{C}(t)}] \to W$ a.s.
Regarding the first sum, using \eqref{eq:tail of Z} and $\norm{L(v)} = \exp(-S(v))$, we obtain that, for arbitrary $\epsilon > 0$,
\begin{equation*}
\sum_{v \in \mathcal{C}(t)} \norm{L(v)} \int_{e^t}^{\norm{L(v)}^{-1}} \!\!\! \Prob(|Z|>x) \, \dx
~\leq ~ \epsilon \sum_{v \in \mathcal{C}(t)} \norm{L(v)} (S(v)-t)
\end{equation*}
for all sufficiently large $t$.
By the law of large numbers for (single-type) general branching processes
\cite[Theorem 3.1]{Nerman:1981}, the last sum converges to a constant multiple of $W$,
see \cite[p.\,191]{Alsmeyer+Meiners:2013} for a detailed argument.
Since $\epsilon > 0$ was arbitrary, it remains to show that in \eqref{eq:Z via M_C(t)}
\begin{equation}	\label{eq:cancellation term for Z}
\sum_{v \in \mathcal{C}(t)} L(v) I(e^t) ~\to~	0	\quad	\text{ in probability as } t \to \infty.
\end{equation}
Now notice that, for all $t \geq 1$, $I(e^t) \in {V_+^\perp}$ and, by \eqref{eq:tail of Z},
\begin{equation*}
|I(e^t)|
~\leq~	1 + \int_1^{e^t} \Prob(|Z|>x) \dx
~\leq~	\mathrm{const} \cdot t.
\end{equation*}
Hence, \eqref{eq:cancellation term for Z} follows from Lemma \ref{Lem:rate of convergence}. \qed
\end{proof}

Let us note one consequence out of the last step of the proof, which constitutes the analogue of \cite[Lemma 4.9]{Iksanov+Meiners:2015a}:

\begin{lemma}\label{lem:W^h bounded}
Let $X$ be a solution to \eqref{eq:FP of ST hom} and let $W^h$ be defined as in \eqref{eq:W'^h}.
For each $h >0$, there is a finite constant $K >0$ such that 
\begin{equation} \label{eq:Wh bounded by W}
|W^h| ~\leq~ KW \quad \text{a.s.}
\end{equation} 
\end{lemma}
\begin{proof}
Fix $h >0$. It can be checked that the identity \eqref{eq:Z^h via M_C(t)} holds with $Z^h$ replaced by $W^h$.
The subsequent estimates leading to Eq.\ \eqref{eq:decomposition I} remain valid (also without restricting to $V_+^\perp$).
Subsequently, one has to use the tail estimate \eqref{eq:tail F} rather than \eqref{eq:tail F endogenous} and hence obtains that 
\begin{equation*}
\bigg|W^h - \lim_{t \to \infty} \sum_{v \in \mathcal{C}(t)} L(v) I(e^t)\bigg| ~\leq~ K' W \quad \text{a.s.}
\end{equation*}
for some $K' \ge 0$. If now $\limsup_{t \to \infty} |I(e^t)| =\infty$, then $|W^h|=\infty$ on the set of survival,
which yields a contradiction.
Hence $\limsup_{t \to \infty} |I(e^t)| <\infty$ and the assertion follows.
\end{proof}

\begin{lemma}	\label{Lem:rate of convergence}
Assume that \eqref{eq:A1}--\eqref{eq:A3}, $\norm{T_j} \le 1$ a.s.\ for all $1 \le j \le N$ and that either \eqref{eq:A4'} or \eqref{eq:A5} hold.
Let $V_+ = \{x \in \R^d: ox=x \text{ for all } o \in \O\}$
be the space of $\O$-invariant vectors and let $V=V_+^\perp$.
Further, denote by $(x_t)_{t \geq 0}$ a bounded sequence in $V$.
Then
\begin{equation}	\label{eq:rate of convergence}
\lim_{t \to \infty} t \bigg(\sum_{v \in \mathcal{C}(t)} L(v)\bigg) x_t
~=~	0	\quad	\text{in probability as } t \to \infty.
\end{equation}
\end{lemma}

\begin{proof}
Without loss of generality, we assume that $\sup_{t \geq 0} |x_t| \leq 1$.

Next notice that $V_+$ is an $\O$-invariant subspace of $\R^d$.
Thus also its orthogonal complement $V$ is an $\O$-invariant subspace of $\R^d$.

\Step 1: First assume that \eqref{eq:A4'} holds.
In particular, $\O$ is finite. Let $o_1,\ldots,o_p$ denote the elements of $\O$ with $o_1 = 1_{\Orth}$.
Then $((S(v),O(v))_{|v|=n})_{n \geq 0}$ defines a multi-type branching random walk with type space $\O$.
Pick an arbitrary $x \in V$. Then
\begin{equation}	\label{eq:matrices cancel}
\sum_{j=1}^p o_j x	~=~	0.
\end{equation}
Indeed, for $o \in \O$, $o\O = \O$ and, hence,
$o (o_1 x + \ldots + o_p x) = o_1x + \ldots + o_p x$.
Consequently,  $o_1x + \ldots + o_p x \in V_+$.
On the other hand, $o_1x+\ldots+o_px \in V$ since $V$ is $\O$-invariant.
This implies \eqref{eq:matrices cancel},
which in turn implies
\begin{equation*}
\frac{W}{p}  \sum_{j=1}^p o_j x_t	~=~	0
\end{equation*}
for every $t \geq 0$. Consequently,
\begin{eqnarray*}
t \bigg|\sum_{v \in \mathcal{C}(t)} L(v) x_t\bigg|
& = &
t \bigg|\sum_{j=1}^p \bigg(\sum_{v \in \mathcal{C}(t):O(v)=o_j} e^{-S(v)} O(v) - \frac{W}{p} o_j \bigg) x_t \bigg|	\\
& \leq &
\sum_{j=1}^p t \bigg| \bigg(\sum_{v \in \mathcal{C}(t):O(v)=o_j} e^{-S(v)} - \frac{W}{p} \bigg) o_j x_t \bigg|	\\
& \leq &
\sum_{j=1}^p t \bigg|\sum_{v \in \mathcal{C}(t):O(v)=o_j} e^{-S(v)} - \frac{W}{p} \bigg|.
\end{eqnarray*}
The result now follows from \cite[Theorem 2.14]{Iksanov+Meiners:2015b}
once it has been checked that the assumptions of the cited theorem are satisfied in the present situation.
The kind of checking that is needed to verify the assumptions can be found in the proof of \cite[Lemma 3.14]{Iksanov+Meiners:2015a}, it is here where the additional assumptions that $\norm{T_j} \leq 1$ a.s., $1 \le j \le N$, enters.

\smallskip
\Step 2: Next, we turn to the more delicate case where \eqref{eq:A5} holds, hence $\O$ is infinite.
Suppose that $I_{t,1},\ldots,I_{t,p_t}$ is a partition of $\O$
and, for $j=1,\ldots,p_t$, define $o_j \defeq h_{t,j}^{-1} \int_{I_{t,j}} o \, H_{\O}(do)$ where $H_{\O}$ is the normed Haar measure on $\O$
and $h_{t,j} = H_{\O}(I_{t,j})$.
Due to the translation invariance of the Haar measure $H_{\O}$,
the matrix $h_{t,1}o_1+\ldots+h_{t,p_t}o_{p_t}$ is invariant under multiplication with elements from $\O$ and
hence $(h_{t,1}o_1+\ldots+h_{t,p_t}o_{p_t})x = 0$ for all $x \in V$.
Therefore,
\begin{align}
t \bigg|\sum_{v \in \mathcal{C}(t)} L(v) x_t\bigg|
& ~=~ 
t \bigg|\sum_{v \in \mathcal{C}(t)} L(v) x_t - W \sum_{j=1}^{p_t} h_{t,j} o_j x_t \bigg|		\label{eq:rate, 1st decomposition}	\\
& ~\leq~
t \bigg| \sum_{j=1}^{p_t} \sum_{v \in \mathcal{C}(t): O(v) \in I_{t,j}} e^{-S(v)} (O(v)-o_j) x_t \bigg|	\notag	\\
& \hphantom{=}~
+t \bigg| \sum_{j=1}^{p_t} \bigg(\sum_{v \in \mathcal{C}(t):O(v) \in I_{t,j}} e^{-S(v)} -  h_{t,j} W\bigg) o_j x_t \bigg|.		\notag
\end{align}
Now assume that the partition is uniformly fine in the sense that
\begin{equation}	\label{eq:uniformly fine partition}
\max_{j=1,\ldots,p_t} \sup_{o \in I_{t,j}} \norm{o-o_j} = o(t^{-1})	\quad	\text{ as }	t \to \infty.
\end{equation}
Since $\O \subseteq \Orth$ and $\Orth$ is a $d(d-1)/2$-dimensional smooth manifold,
it has box-counting dimension $d(d-1)/2$, see e.g.\ \cite[Section 3.2]{Falconer:1990}.
Hence, for $0 < \epsilon < \frac\delta2 \wedge 1$, with $\delta$ as in \eqref{eq:A4},
we can choose the family of partitions in such a way that
${ p_t = o(t^{d(d-1)/2+\epsilon}) = o(t^{\ell-1+\epsilon})}$ as $t \to \infty$,
this will be assumed as well.

Then the summand in the second line of \eqref{eq:rate, 1st decomposition} vanishes as $t \to \infty$.
We thus centre our attention on the summand in the third line.
We will prove that
\begin{equation}	\label{eq:continuous rate of convergence}
t \sum_{j=1}^{p_t} \bigg|  \sum_{v \in \mathcal{C}(t):O(v) \in I_{t,j}} e^{-S(v)} -  h_{t,j} W\bigg|	~\to~	0	\quad	\text{in probability as } t \to \infty.
\end{equation}
We reformulate this problem in terms of a general branching process with type space $\O$
(see \cite{Jagers:1989} for definition and details).
For a measurable set $I \subseteq \O$ define
\begin{equation*}
\phi_I(t) ~=~ e^{t} \sum_{j \geq 1} e^{-S_j(\varnothing)} \1_{\{O_j(\varnothing) \in I\}} \1_{[0,S_j(\varnothing))}(t) ,	\quad	t \geq 0.
\end{equation*} 
For $c >0$, let $\phi_I^c(t) \defeq \1_{[0,c]}(t) \phi_I(t)$ and notice that
\begin{equation*}
[\phi_I^c]_v(t) ~=~ \1_{[0,c]}(t) e^{t} \sum_{j \geq 1} e^{-S_j(v)} \1_{\{O(v)O_j(v) \in I\}} \1_{[0,S_j(v))}(t) ,	\quad	t \geq 0.
\end{equation*}
The general branching process counted with characteristic $\phi_I^c$ is defined as
\begin{equation*}
\mathcal{Z}^{\phi_I^c}(t) \defeq \sum_{v \in \V} [\phi_I^c]_v(t-S(v)) = e^t \sum_{v \in \mathcal{C}(t): S(v|_{|v|-1}) \geq t-c} e^{-S(v)} \1_{\{O(v) \in I\}},
\quad	t \geq 0.
\end{equation*}
We write $m_t^{\phi_I^c} \defeq e^{-t} \E [\mathcal{Z}^{\phi^c_I}]$, which is finite since it is bounded by $\E[\mathcal{Z}^{\phi_\O}]=1$,
see \eqref{eq:change of measure stopping line}, and define
\begin{equation}	\label{eq:def h_I^c}
h_{I}^c ~\defeq~ \frac{\E[S_1 \wedge c]}{\E [S_1]} H_\O(I).
\end{equation} 
For $0<c<t$, we can then rewrite the left-hand side of \eqref{eq:continuous rate of convergence} as follows:
\begin{align}
t & \sum_{j=1}^{p_t} \bigg| \sum_{v \in \mathcal{C}(t):O(v) \in I_{t,j}} e^{-S(v)} -  h_{t,j} W\bigg|
~=~t  \sum_{j=1}^{p_t} \bigg| e^{-t} \mathcal{Z}^{\phi_{I_{t,j}}}(t) - h_{t,j} W \bigg| \notag \\
& \leq~
t \sum_{j=1}^{p_t} \Big(e^{-t} \mathcal{Z}^{\phi_{I_{t,j}}}(t) - e^{-t} \mathcal{Z}^{\phi^c_{I_{t,j}}}(t)\Big) + t \sum_{j=1}^{p_t}   \big(h_{t,j} - h_{I_{t,j}}^c \big)W	\label{eq:rate, 2nd decomposition}  \\
& \hphantom{\leq~} + t \sum_{j=1}^{p_t} \bigg| e^{-t} \mathcal{Z}^{\phi^c_{I_{t,j}}}(t) - h_{I_{t,j}}^c W \bigg| .
	\notag
\end{align}
The first sum in \eqref{eq:rate, 2nd decomposition} tends to $0$ in $\L^1$ as $t \to \infty$ when $c=t/4$.
Indeed, observe that $\O = \bigcup_{j=1}^{p_t} I_{t,j}$ implies that
$\phi^c_{I_{t,1}}(t)+\ldots+	\phi^c_{I_{t,p_t}}(t) = \phi^c_\O(t)$, $t \geq 0$
and, hence,
\begin{equation}	\label{eq:sum over Z^phi_I_t,j}
\sum_{j=1}^{p_t} \mathcal{Z}^{\phi^c_{I_{t,j}}}(t)	~=~	\mathcal{Z}^{\phi^c_\O}(t),	\quad	t \geq 0.
\end{equation} Then, taking expectations in the first sum in \eqref{eq:rate, 2nd decomposition} and using Eq. \eqref{eq:change of measure stopping line} as well as $\E[W]=1$ leads to
\begin{align*}
& t\,  \E\bigg[ e^{-t} \sum_{j=1}^{p_t} \mathcal{Z}^{\phi_{I_{t,j}}}(t) - e^{-t} \sum_{j=1}^{p_t}\mathcal{Z}^{\phi^c_{I_{t,j}}}(t)\bigg] \\
&~=~
t\, \E\bigg[\sum_{v \in \mathcal{C}(t)} e^{-S(v)} (1-\1_{\{S(v|_{|v|-1}) \ge t-c \}}) \bigg]\\
&~=~
t\, \Prob(S_{\tau(t)-1} < t-c)
~\leq~
t\, \Prob(S_{\tau(t)}-S_{\tau(t)-1} > c) \to 0	\text{ as } t \to \infty
\end{align*}
by Lemma \ref{Lem:no big jump over t} if $c$ grows linearly with $t$.
From now on, we fix $c=t/4$.
The expectation of the second sum in \eqref{eq:rate, 2nd decomposition} is bounded by
\begin{equation*}
t \,  \frac{\E[S_1] - \E[S_1 \wedge c]}{\E [S_1]} ~=~ \frac{1}{\E[S_1]} \, t \, \int_{t/4}^\infty \Prob(S_1 > s) \, \ds ~\to~0
\end{equation*}
as $t \to \infty$ since $\E[S_1^2] < \infty$ (which implies $P(S_1>s)=o(s^{-2})$).

The remaining sum in \eqref{eq:rate, 2nd decomposition} requires much more attention. 
We begin the proof by observing that $[\phi^c_{I_{t,j}}]_v(t-S(v)) = 0$ for all $v$ with $S(v) \leq \frac t2$ due to the convention $c=t/4$. Consequently, for $j=1,\ldots,p_t$,
\begin{equation*}
e^{-t} \mathcal{Z}^{\phi^c_{I_{t,j}}}(t)	~=~	\sum_{v \in \mathcal{C}(\frac t2)} e^{-S(v)} e^{-(t-S(v))}[\mathcal{Z}^{\phi^c_{I_{t,j}}}]_v(t-S(v)).
\end{equation*}
Using this, we can estimate as follows: 
\begin{align*}
t  & \sum_{j=1}^{p_t} \bigg| e^{-t} \mathcal{Z}^{\phi^c_{I_{t,j}}}(t) - h_{I_{t,j}}^c W\bigg|	\\
& \leq~
t \sum_{j=1}^{p_t} \sum_{v \in \mathcal{C}(\frac t2)} e^{- S(v)} \bigg| e^{-(t-S(v))}  [\mathcal{Z}^{\phi^c_{I_{t,j}}}]_v(t-S(v)) - m_{t-S(v)}^{\phi^c_{I_{t,j}}} \bigg|	\\
& \hphantom{\leq~}
+ t \sum_{j=1}^{p_t} \sum_{v \in \mathcal{C}(\frac t2): S(v) \le at} e^{-S(v)} \bigg| m_{t-S(v)}^{\phi_{I_{t,j}}^c} - h_{I_{t,j}}^c \bigg| \\
& \hphantom{\leq~}
+ t \sum_{j=1}^{p_t} \sum_{v \in \mathcal{C}(\frac t2): S(v) > at} e^{-S(v)} \bigg| m_{t-S(v)}^{\phi_{I_{t,j}}^c} - h_{I_{t,j}}^c \bigg|\\
& \hphantom{\leq~} 
+ t \sum_{j=1}^{p_t} h_{I_{t,j}}^c \bigg| \sum_{v \in \mathcal{C}(\frac t2)} e^{-S(v)} - W\bigg| ~\eqdef~\sum_{k=1}^4 J_k(t) 
\end{align*}
for some (fixed) $a \in (\frac 12, 1)$. We consider each term separately. Using \eqref{eq:def h_I^c}, the sum over $(h^c_{I_{t,j}})_j$ in $J_4(t)$ is uniformly bounded by one, and \cite[Proposition 4.3]{Iksanov+Meiners:2015b} gives that $\lim_{t \to \infty} | \sum_{v \in \mathcal{C}(\frac t2)} e^{-S(v)} - W | =0$ a.s. For $J_3(t)$, we change the order of the summation (the number of summands being finite a.s.) and use  \eqref{eq:sum over Z^phi_I_t,j} and the change of measure \eqref{eq:change of measure stopping line} to obtain
\begin{align*}
\E [J_3(t)] ~&\le~ t \E \bigg[ \sum_{v \in \mathcal{C}(\frac t2):S(v) > at } e^{-S(v)} \sum_{j =1}^{p_t} \big( m_{t-S(v)}^{\phi_{I_{t,j}}^c} + h_{I_{t,j}}^c \big)   \bigg] \\
&\le~  2 \, t \, \Prob \big( S_{\tau(t/2)}-t/2 > (a-1/2)t \big), 
\end{align*}
which vanishes by \cite[Lemma A.3]{Iksanov+Meiners:2015b}.
Turning to $J_2(t)$, using that $p_t =o(t^{\ell-1+\epsilon})$ and $\E[\sum_{v \in \mathcal{C}(\frac t2)} e^{-S(v)}]=1$
(by \eqref{eq:change of measure stopping line}), we have that
\begin{align}
\E[J_2(t)] ~&=~t
\E\bigg[\sum_{v \in \mathcal{C}(\frac t2):S(v) \le at} e^{-S(v)} \sum_{j=1}^{p_t} \big|  m_{t-S(v)}^{\phi_{I_{t,j}}^c} - h_{I_{t,j}}^c  \big| \bigg] \nonumber	\\
 &\le~ t \sup_{s \ge (a-1/2)t} \sum_{j=1}^{p_t} \big|  m_{s}^{\phi_{I_{t,j}}^c} - h_{I_{t,j}}^c  \big|  \nonumber \\
 &\le~\sup_{s \ge (a-1/2)t} \Big( \frac{s}{a-1/2} \Big)^{\ell+\epsilon} \max_{1 \le j \le p_t} \big|  m_{s}^{\phi_{I_{t,j}}^c} - h_{I_{t,j}}^c  \big|.
 \label{eq:rate, renewal}
 \end{align}
To proceed further, we need a rate-of-convergence result from Markov renewal theory.
Indeed, applying the change of measure, 
\begin{align}
&m_t^{\phi_{I_{t,j}^c}} = e^{-t} \E  \bigg[ \sum_{v \in \V} [\phi_{I_{t,j}^c}]_v(t-S(v)) \bigg]	\nonumber \\
&= e^{-t} \, \E \bigg[  \sum_{v \in \V} e^{t-S(v)} \1_{[0,c]}(t-S(v))\sum_{i \ge 1} e^{-S_i(v)} \1_{I_{t,j}}( O(vi))  \1_{[0, S_i(v)]}(t-S(v))  \bigg] \nonumber\\
&=  \sum_{n=0}^\infty \E \bigg[ \sum_{\abs{v}=n} e^{-S(v)} \, f_{t,j}^c(O(v),t-S(v)) \bigg] 
= \E \bigg[ \sum_{n=0}^\infty  f_{t,j}^c (O_n,t-S_n) \bigg], \nonumber 
\end{align}
with 
\begin{equation*}
f_{t,j}^c(o,r) ~\defeq~ \1_{[0,c]}(r)\E\big[ \1_{I_{t,j}}(oO_1)\1_{[0,S_1)}(r) \big] ~\le~\1_{[0,\infty)}(r) \Prob(S_1 > r) ~\eqdef~g(r).
\end{equation*}
By \eqref{eq:A4}, the function $g$ satisfies $g(r)=o(r^{\ell+\delta+1})$.
From \eqref{eq:def h_I^c}, we conclude
\begin{equation*}
h_{I_{t,j}}^c ~=~\frac{1}{\E[S_1]} \int_\R \int_\O f_{t,j}^c(o,r) H_\O(\ddo) \dr
\end{equation*}
and hence we can apply Proposition \ref{prop:uniform renewal} to deduce
that the last term in \eqref{eq:rate, renewal} tends to zero as $t \to \infty$.
 
We finally consider $J_1(t)$ and proceed as in \cite[pp.\,735--736]{Iksanov+Meiners:2015b}.
For fixed $t$, define
\begin{equation*}
Z_{v,j}~\defeq~e^{-(t-S(v))}[\mathcal{Z}^{\phi_{I_{t,j}}^c}]_v(t-S(v))
\end{equation*}
and similarly $Z_v$ with $I_{t,j}$ replaced by $\O$, i.e., $Z_v=\sum_{j=1}^{p_t} Z_{v,j}$.
Conditioned upon $\mathcal{F}_{\mathcal{C}(t/2)}$, the $Z_v$, $v \in \mathcal{C}(t/2)$ are independent.
Let $Z_{v,j}' \defeq Z_{v,j}\1_{\{Z_v \le e^{S(v)}\}}$, $m_{v,j}'\defeq\E[Z_{v,j}']$, and $J_1'(t)$ as $J_1(t)$,
but with $Z_{v,j}$ and $m_{t-S(v)}^{\phi_{I_{t,j}}^c}$ replaced by $Z_{v,j}'$ and $m_{v,j}'$, respectively.
On the set $\big\{ Z_v \le e^{S(v)} \text{ for all } v \in \mathcal{C}(t/2)\big\}$,
\begin{align*}
J_1(t) ~&=~J_1'(t) + t \sum_{j=1}^{p_t} \sum_{v \in \mathcal{C}(\frac t2)} e^{-S(v)} \big(m_{t-S(v)}^{\phi_{I_{t,j}}^c} - m_{v,j}' \big) \\
 ~&=~J_1'(t) + t \sum_{v \in \mathcal{C}(\frac t2)} e^{-S(v)} \big(m_{t-S(v)}^{\phi_{\O}^c} - m_{v}' \big).
\end{align*}
We want to prove that $J_1(t) \to 0$ in probability.
To this end, we use the above decomposition and obtain for arbitrary $\eta >0$,
\begin{align*}
\Prob\big( |J_1(t)| \ge \eta\big) ~&=~ \E  \big[  \Prob( |J_1(t)| \ge \eta | \mathcal{F}_{\mathcal{C}(\frac t2)})   \big] \\
&\le~ \E \bigg[ \sum_{v \in \mathcal{C}(\frac t2)} \Prob( Z_v > e^{S(v)} | \mathcal{F}_{\mathcal{C}(\frac t2)} ) \bigg] ~+~ \E \big[  \Prob( |J_1'(t)| \ge \eta/2 | \mathcal{F}_{\mathcal{C}(\frac t2)})   \big] \\
&\phantom{\le~} + \frac{2 t}{\eta} \E \bigg[ \sum_{v \in \mathcal{C}(\frac t2)} e^{-S(v)} \big(m_{t-S(v)}^{\phi_{\O}^c} - m_{v}' \big) \bigg].
\end{align*}
The first and the last term can be dealt with as the corresponding terms in \cite[pp.\,735--736]{Iksanov+Meiners:2015b}.
It remains to consider the middle term.
\begin{align*}
\E \big[\Prob&( |J_1'(t)| \ge \eta/2 | \mathcal{F}_{\mathcal{C}(\frac t2)})   \big] \\
~\le&~\sum_{j=1}^{p_t} \E \bigg[  \Prob \bigg( t \sum_{v \in \mathcal{C}(\frac t2)} e^{-S(v)}|Z_{v,j}'-m_{v,j}' | \ge \eta/(2 p_t)
\, \Big| \, \mathcal{F}_{\mathcal{C}(\frac t2)} \bigg)   \bigg] \\
\le&~\sum_{j=1}^{p_t} \frac{4 p_t^2 t^2}{\eta^2} \E \bigg[ \sum_{v \in \mathcal{C}(\frac t2)} e^{-2 S(v)} \mathrm{Var}[Z_{v,j}' |  \mathcal{F}_{\mathcal{C}(\frac t2)}] \bigg] \\
\le&~ \frac{4 p_t^2 t^2}{\eta^2} \E \bigg[ \sum_{v \in \mathcal{C}(\frac t2)} e^{-2 S(v)}  \sum_{j=1}^{p_t} \E[(Z_{v,j}')^2 |  \mathcal{F}_{\mathcal{C}(\frac t2)}] \bigg] \\ 
~\le&~ \frac{4 p_t^2 t^2}{\eta^2} \E \bigg[ \sum_{v \in \mathcal{C}(\frac t2)} e^{-2 S(v)}   \E[Z_{v}^2 \1_{\{Z_v \le e^{S(v)} \}} |  \mathcal{F}_{\mathcal{C}(\frac t2)}] \bigg] \\
=&~\frac{4 p_t^2 t^2}{\eta^2} \E \bigg[ \sum_{v \in \mathcal{C}(\frac t2)} e^{-2 S(v)}   \E\bigg[h_{2\ell+\delta}(Z_v) \frac{Z_{v}^2}{h_{2\ell+\delta}(Z_v)} \1_{\{Z_v \le e^{S(v)} \}} |  \mathcal{F}_{\mathcal{C}(\frac t2)} \bigg] \bigg] \\
\le&~\frac{4 p_t^2 t^2}{\eta^2} \E \bigg[ \sum_{v \in \mathcal{C}(\frac t2)} e^{- S(v)}  \frac{e^{S(v)}}{h_{2\ell+\delta}(e^{S(v)})} \bigg] \sup_{s \ge 0}
\E\bigg[h_{2\ell+\delta} \big(e^{-s} \mathcal{Z}^{\phi_\O}(s)\big)  \bigg] \\
\le&~\frac{4}{\eta^2} \frac{e^{t/2} t^{2 \ell + 2 \epsilon}}{h_{2\ell+\delta}(e^{t/2})} \sup_{s \ge 0} \E\bigg[h_{2\ell+\delta} \big(e^{-s} \mathcal{Z}^{\phi_\O}(s)\big)  \bigg] ~\to 0 \quad \text{ as } t \to \infty.
\end{align*}
Here we used the triangular inequality, the independence of $Z_{v,j}'$ and $\mathcal{F}_{\mathcal{C}(\frac t2)}$,
Chebyshev's inequality and the facts that $t \mapsto t^2/h_{2\ell+\delta}(t)$ and $t \mapsto t/ h_{2\ell+\delta}(t)$ are increasing and decreasing, respectively, for large $t$,
$t (\log t)^{2 \ell + 2\epsilon} / h_{2\ell+\delta}(t) \to 0$ as $t \to \infty$,
and the finiteness of the supremum, which follows from \eqref{eq:A4}
and is proved as in \cite[Lemma 3.14]{Iksanov+Meiners:2015a}.	\qed
\end{proof} 
%
%
%

\subsection{Computing $\Psi$}

In this section, we finish the determination of the $\F$-measurable L\'evy triplet $(W',\bf \Sigma, \nu)$
of solutions of the homogeneous equation.
In order to so, we will make use of some results proved in Section \ref{sec:invariant matrices and measures} below,
which are postponed since their proofs do not require probabilistic tools like branching processes that are used in this section.

As a by-product, we prove Proposition \ref{Prop:(U,alpha)-stable laws}.
We remind the reader of the definition of the functions $\eta_1^\alpha$, $\eta_2^\alpha$, $\eta^1$
and the vector $\gamma^1$ in Eqs.\ \eqref{eq:Definition eta1}--\eqref{eq:Definition gamma}.
$\eta_1^\alpha$, $\eta_2^\alpha$ are defined in terms of a $(\U,\alpha)$-invariant L\'evy measure $\nu^\alpha$,
i.e., satisfying \eqref{eq:bar nu (U,alpha)-invariant};
$\eta^1$ and $\gamma^1$ are defined in terms of a measure $\rho$ on $\Sd$. Note that \eqref{eq:A4'} implies $\G=\R$, hence $\U=\{t^Q \, : \, t >0 \} \times C_\U$, see Proposition \ref{Prop:Buraczewski et al}.
Here $t^Q \defeq  e^{(\ln t) Q}$, and we choose $Q$ such that $\norm{t^Q}=t$.

\begin{proposition}\label{Prop:solving Psi}
Assume \eqref{eq:A1}--\eqref{eq:A3}, let $\phi$ be a solution to \eqref{eq:FE of ST hom}
and let $\Phi=\exp(\Psi)$ be the limit of the multiplicative martingales, given by Proposition \ref{Prop:Multiplicative martingales}. 
\begin{itemize}
	\item[(a)]	Let $\alpha \in (0,1)$. There is a $(\U,\alpha)$-invariant L\'evy measure $\nu^\alpha$
				such that a.s.,
				\begin{equation}	\label{eq:Psi alpha smaller 1}
				\Psi(x) ~=~- W |x|^\alpha \eta_1^\alpha(x) + \imag W |x|^\alpha \eta_2^\alpha(x) \qquad \forall\, x \in \R^d .
				\end{equation}
	\item[(b)]	Let $\alpha =1$.
				\begin{itemize}
					\item[(b1)]	Assume \eqref{eq:A4} in addition.
								There is a $c>0$ such that a.s.,
								\begin{equation}	\label{eq:Psi alpha is 1, continuous case}
								\Psi(x) ~=~- Wc \abs{x} \qquad \forall\, x \in \R^d .
								\end{equation}
					\item[(b2)]	Assume \eqref{eq:A4'} in addition.
								There is a $z \in \R^d$ with $\E[\sum_{j \ge 1} T_j]z=z$ and a finite $\CU$-invariant measure $\rho$ on $\Sd$,
								satisfying $\int \scalar{x,s} \rho(\ds)=0$ for all $x \in E_1(Q^\transp)$,
								such that a.s., 
								\begin{equation}\label{eq:Psi alpha is 1, discrete case}
								\Psi(x) ~=~ \imag W \scalar{z,x} + W \big(\eta^1(x) + \imag\scalar{\gamma^1,x}\big) \qquad \forall\, x \in \R^d .
								\end{equation}
				\end{itemize}
	\item[(c)]	Let $\alpha \in (1,2)$. There is a $(\U,\alpha)$-invariant L\'evy measure $\nu^\alpha$
				such that a.s.,
				\begin{equation}\label{eq:Psi alpha greater 1}
				\Psi(x) ~=~ \imag  \scalar{Z,x} - W|x|^\alpha \eta_1^\alpha(x) + \imag W |x|^\alpha \eta_2^\alpha(x) \qquad \forall\, x \in \R^d .
				\end{equation}
	\item[(d)]	Let $\alpha=2$. Then there is a positive semi-definite $d \times d$ matrix $\Sigma$ satisfying $o^\transp \Sigma o$ for all $o \in \O$ and a  $z \in \R^d$ with $z = \sum_{j \geq 1} T_j z$ a.s.,
				such that a.s., 
				\begin{equation}\label{eq:Psi alpha is 2}
				\Psi(x) ~=~ \imag  \scalar{z,x} - W\frac{x^\transp \Sigma x}{2} \qquad \forall\, x \in \R^d .
				\end{equation}
	\item[(e)]	Let $\alpha >2$. Then there is a $z \in \R^d$ with $z = \sum_{j \geq 1} T_j z$ a.s., such that a.s., 
				\begin{equation}	\label{eq:Psi alpha greater 2}
				\Psi(x) ~=~ \imag  \scalar{z,x} \qquad \forall\, x \in \R^d .
				\end{equation}
\end{itemize}
\end{proposition}

\begin{proof}
By Proposition \ref{Prop:Multiplicative martingales},
$\Psi$ is a L\'evy-Khintchine exponent (see \eqref{eq:Psi}) with an ${\F}$-measurable L\'evy triplet $(W', {\bf \Sigma}, \nu)$.
By Lemma \ref{Lem:nu evaluated}, $\nu=W \bar{\nu}$ a.s.~for a deterministic $(\U,\alpha)$-invariant L\'evy measure $\bar{\nu}$.
By Lemma \ref{Lem:Sigma=W times deterministic Sigma}, $\mathbf{\Sigma} = W \Sigma$ a.s.\ for a deterministic covariance matrix $\Sigma$
satisfying $\Sigma = o \Sigma o^{\transp}$ for all $o \in \O$.
Moreover, $\bar\nu =0$ if $\alpha \geq 2$ and $\Sigma=0$ unless $\alpha=2$.
The burden of the proof is to determine the random shift $W'$.
We now consider the cases separately. 

(a) Let $0 < \alpha < 1$. Using the evaluation of the L\'evy integral in Lemma \ref{Lem:I(tx) evaluated},
we obtain from \eqref{eq:Psi} that for all $x \in \R^d \setminus \{0\}$,
\begin{align}
\Psi(x) ~&=~\imag \scalar{W',x} + W \big(- |x|^\alpha \eta_1^\alpha(x) + \imag |x|^\alpha \eta_2^\alpha(x) + \imag \scalar{\gamma^\alpha,x} \big) \nonumber \\
~&=~\imag  \scalar{W'+\gamma^\alpha W,x} - W  |x|^\alpha \eta_1^\alpha(x) + \imag W |x|^\alpha \eta_2^\alpha(x)  \label{eq:alpha smaller 1,Psi1},
\end{align}
for functions $\eta_1^\alpha$, $\eta_2^\alpha$ defined as  in \eqref{eq:Definition eta1} and \eqref{eq:Definition eta2}, resp.,
in terms of $\nu^\alpha \defeq \bar{\nu}$. Using that $\eta_j^\alpha(u^\transp x)=\eta_j^\alpha(x)$, $j=1,2$, and \eqref{eq:M},
we infer that for all $n \in \N$
\begin{align}
\Psi(x) ~&=~ \sum_{\abs{v}=n} [\Psi]_v(L(v)^\transp x) \nonumber \\
&=~\imag  \sum_{\abs{v}=n} \scalar{L(v) [W'+\gamma^\alpha W]_v,x} - \sum_{\abs{v}=n} \norm{L(v)}^\alpha [W]_v |x|^\alpha \eta_1^\alpha(x) \nonumber \\
&\phantom{=~} + \imag \sum_{\abs{v}=n} \norm{L(v)}^\alpha [W]_v |x|^\alpha \eta_2^\alpha(x)  \quad \text{a.s.}\label{eq:alpha smaller 1,Psi2}
\end{align}
Combining \eqref{eq:alpha smaller 1,Psi1} and \eqref{eq:alpha smaller 1,Psi2} and linear independence of $1$ and $\imag$, we obtain
\begin{align}
\scalar{W'+&\gamma^\alpha W,x} +  W |x|^\alpha \eta_2^\alpha(x) \nonumber \\ 
~=&~  \sum_{\abs{v}=n}  \scalar{L(v) [W'+\gamma^\alpha W]_v,x}
 +  \sum_{\abs{v}=n} [W]_v \norm{L(v)}^\alpha |x|^\alpha \eta_2^\alpha(x). \label{eq:alpha smaller 1,Psi3}
\end{align}
Dividing by $|x|$ and letting $\abs{x} \to \infty$, we obtain
\begin{equation} \label{eq:alpha smaller 1,Psi4}
\scalar{W'+\gamma^\alpha W,y} ~=~ \sum_{\abs{v}=n}  \scalar{L(v) [W'+\gamma^\alpha W]_v,y}
\end{equation}
for all $y \in \Sd$, hence $W'+\gamma^\alpha W$ is a solution to \eqref{eq:FP of ST hom a.s.},
i.e., is an endogenous fixed point. Since $\alpha <1$, Proposition \ref{Prop:convergence Z_n w} yields that $W'+\gamma^\alpha W=0$ a.s.
Then \eqref{eq:Psi alpha smaller 1} follows from \eqref{eq:alpha smaller 1,Psi1}.

(c) For $1 < \alpha <2$, we start from Eq.\ \eqref{eq:alpha smaller 1,Psi3}. Dividing by $|x|$, but considering $\abs{x} \to 0$ this time, we obtain the identity \eqref{eq:alpha smaller 1,Psi4}. Hence, $W'+\gamma^\alpha W$ is an endogenous fixed point.
Proposition \ref{Prop:convergence Z_n w} thus implies \eqref{eq:Psi alpha greater 1}.

(b1) Assumption \eqref{eq:A4} implies $\U=\R_> \times \O$ with $\O=\SOd$ or $\O=\Orth$, see Remark \ref{Rem:U under A4}.
Thus, the r.v.~with L\'evy triplet $(0,0,\bar{\nu})$ is rotation invariant and $1$-stable.
By \cite[Theorem 14.14]{Sato:1999}, its characteristic exponent equals $-c \abs{x}$ for some $c>0$ and we obtain that a.s.\ 
\begin{equation*}
\Psi(x) ~=~\imag \scalar{W',x} - W c \abs{x}.
\end{equation*}
Using \eqref{eq:M} and considering real and imaginary part separately, we obtain that 
\begin{equation*}
W' ~=~\sum_{\abs{v}=n} L(v) [W']_v,
\end{equation*}
and Proposition \ref{Prop:convergence Z_n w} yields that $W'=0$ a.s.\ due to the assumption \eqref{eq:A4}. 

(b2) Let $\rho$ be the spherical component of the L\'evy measure $\bar{\nu}$, given by Proposition \ref{Prop:structure of bar nu}.
We start by proving that $\int \scalar{x,s} \rho(\ds)=0$ for all $x \in E_1(Q^\transp)$. We consider two different cases: 

1. Suppose $x \in E_1(Q^\transp) \cap E_1(\CU) \eqdef V'$, i.e.,
$u^\transp x=\norm{u}x$ for all $u \in \U$.
For this case, we adjust the proof given of \cite[Theorem  4.10 (b1)]{Iksanov+Meiners:2015a} to the present situation. 
Note that on $V'$,
\eqref{eq:FP of ST hom} reduces to a fixed-point equation of a smoothing transformation with nonnegative scalar weights.
Let $e_1, \dots, e_k$ denote an orthonormal basis of $V'$ and write $\tilde{W}_j :=\scalar{e_j, W' + \gamma W}$
where $\gamma$ is as in \eqref{eq:Psi alpha 1 a}.
Further, let $s_j:=\scalar{e_j, s}$ for $j=1,\ldots,k$ and $s \in \R^d$.
Thus, using \eqref{eq:Psi alpha 1 a} and  the linear independence of $\imag$ and $1$, we obtain for all $r \in \Rp$,
$j = 1,\ldots,k$, a.s.
\begin{align}
& r\tilde{W_j} -  W \frac{2}{\pi} \int_{\Sd}  r s_j \log(|r s_j|) \, \rho(\ds) \nonumber  \\
&~=~ r \sum_{v=n} \norm{L(v)} [\tilde{W}_j]_v - \sum_{v=n} \norm{L(v)} \log( \norm{L(v)}) [W]_v \frac2\pi \int_{\Sd} r s_j \,\rho(\ds) \nonumber \\
&\hphantom{~=~ r \sum_{v=n} \norm{L(v)} [\tilde{W}_j]_v}  - \sum_{v=n} \norm{L(v)} [W]_v \frac2\pi \int_{\Sd} r s_j \log (|r s_j|) \, \rho(\ds).	\label{eq:Jej}
\end{align}
Assuming for a contradiction that $\int_{\Sd} s_j \rho(\ds) \neq 0$ for some $j$,
we choose $r > 0$ such that $\int_{\Sd} r s_j \log |r s_j| \, \rho(\ds)$ vanishes. Hence, upon dividing by $r$,
\begin{align} \label{eq:tilde W}
\tilde{W_j} ~=~ \sum_{v=n} \norm{L(v)} [\tilde{W}_j]_v - \sum_{v=n} \norm{L(v)} \log \big( \norm{L(v)} \big) [W]_v \frac2\pi \int_{\Sd}  s_j \,\rho(\ds).
\end{align} 
By Lemma \ref{lem:W^h bounded}, Eqs.\ \eqref{eq:W'} and \eqref{eq:W'^h} and Lemma \ref{Lem:nu evaluated},
we deduce that there is some $K \geq 0$ such that $|\tilde{W}_j| \leq K W$ a.s.,
and consequently also $|[\tilde{W}_j]_v| \le K [W]_v$ a.s.\ for all $v \in \V$.
This together with \eqref{eq:tilde W} yields that a.s.
\begin{equation*}
\bigg|\sum_{v=n} \norm{L(v)} \log(\norm{L(v)}) [W]_v \frac2\pi \int_{\Sd}  s_j \,\rho(\ds) \bigg| ~\le~ 2 K W.
\end{equation*}
The assumption $\int_{\Sd} s_j\ \rho(\ds) \neq 0$ implies that the left-hand side tends to $\infty$ a.s.\ since
$\lim_{n \to \infty} \sup_{|v|=n} \norm{L(v)} = 0$ a.s.\ by \eqref{eq:Biggins:1998}
and $\sum_{|v|=n} \norm{L(v)} [W]_v=W$ a.s. Contradiction!
Hence, $\int_{\Sd} \scalar{x,s} \, \rho(\ds)=0$ for all $x \in V'$.

2. Suppose $x \in E_1(\CU)^\perp$. Since $\rho$ is $\CU$-invariant (see Proposition \ref{Prop:structure of bar nu}),
it follows that $s_0 \defeq \int s\, \rho(\ds)$ is $\CU$-invariant as well, thus $0=\scalar{x,s_0}=\int \scalar{x,s} \rho(\ds)$.

Hence, by 1. and 2. together, 
\begin{equation}	\label{eq:condition for rho}
\int \scalar{x,s} \, \rho(\ds)=0	\quad \text{for all } x \in E_1(Q^\transp).
\end{equation}
By Corollary \ref{Cor:operator stability} below, the infinitely divisible law with L\'evy triplet $(0,0,\bar{\nu})$ is operator-stable with exponent $Q$.
For such laws, \cite[Theorem 13]{Luczak:2010} shows that validity of \eqref{eq:condition for rho}
is equivalent to the existence of $\gamma^1$
(given by the fomula \eqref{eq:Definition gamma}, cf.\ \cite[Proposition 12 and Eq.\ (19)]{Luczak:2010})
such that
\begin{equation*}
\tilde{\Psi}(x)~\defeq~\imag\scalar{\gamma^1,x} + \eta^1(x)
\end{equation*} 
satisfies $\tilde{\Psi} \big((t^Q)^\transp  x\big) = t \tilde{\Psi}(x)$ for all $t>0$.
Moreover, since $t^Q$ commutes with $\CU$ (see Prop. \ref{Prop:Buraczewski et al}),
the $\CU$-invariance of $\rho$ and the definition of $\gamma^1$ imply that $\gamma^1$ is $\CU$-invariant, and so is $\eta^1$.
Hence, $\tilde{\Psi} \big((t^Q)^\transp o^\transp  x\big) = t \tilde{\Psi}(x)$ for all $t>0$, $o \in \CU$, i.e.,
$\tilde{\Psi}$ is the exponent of a strictly $(\U,1)$-stable law.
 
Thus we can write
\begin{equation}
\Psi(x) ~=~\imag \scalar{W'-\gamma^1 W,x} + W \tilde{\Psi}(x), \label{eq:Psi alpha 1 b}
\end{equation}
and this equals, using \eqref{eq:M} and the strict $(\U,1)$-stability of $\tilde{\Psi}$,
\begin{align}
\Psi(x) ~=&~\imag \scalar{\sum_{|v|=n} L(v) [W'-\gamma^1 W]_v,x} + \sum_{|v|=n} \tilde{\Psi}(x) \norm{L(v)}[W]_v \nonumber	\\
~=&~\imag \scalar{\sum_{|v|=n} L(v) [W'-\gamma^1 W]_v,x} + W \tilde{\Psi}(x). \label{eq:Psi alpha 1 c}
\end{align}
Substracting \eqref{eq:Psi alpha 1 c} from \eqref{eq:Psi alpha 1 b}, we infer that $W'-\gamma^1W$ is an endogenous fixed point, hence equals $zW$ for some $z$ with $\E[\sum_{j \ge 1} T_j] z =z$ by Proposition \ref{Prop:convergence Z_n w}. All in all,
\begin{equation*} \Psi(x)=zW + \tilde{\Psi}(x).\end{equation*}

(d) $\alpha=2$ implies $\nu=0$ a.s.~and hence
\begin{equation*}
\Psi(x)~=~\imag \scalar{W',x} - W\frac{x^\transp \Sigma x}{2}.
\end{equation*}
The claimed properties of $\Sigma$ are proved in Lemma \ref{Lem:Sigma=W times deterministic Sigma}. Using again the linear independence of $1$ and $\imag$ and \eqref{eq:M}, we deduce, since $x \in \R^d \setminus \{0\}$ is arbitrary, that
\begin{equation*}
W' ~=~\sum_{\abs{v}=n} L(v) [W']_v,
\end{equation*}
hence $W'$ satisfies \eqref{eq:FP of ST hom a.s.}.
By Proposition \ref{Prop:convergence Z_n w}, either $W'=0$ or $W'=w$ for a deterministic $w \neq 0$
which satisfies $w = \sum_{j \ge 1} T_j w$ a.s.

(e) In this case, $\nu$ and ${\bf \Sigma}$ vanish, $W'$ can be identified as in (d). \qed
\end{proof}

\begin{proof}[Proof of Proposition \ref{Prop:(U,alpha)-stable laws}]
The notion of $(U,\alpha)$-stability implies, using uniqueness of the L\'evy triple $(\gamma, \Sigma, \nu^\alpha)$,
that $\nu^\alpha$ satisfies \eqref{eq:bar nu (U,alpha)-invariant} and  that $o^\transp \Sigma o$ for all $o \in O$.
For $\alpha \neq 1$, we can argue as before, using the identity $\Psi(u^\transp x)=\norm{u}^\alpha\Psi(x)$
and letting $\norm{u} \to \infty$ resp. $\norm{u} \to 0$ to prove that only $\eta_{1,2}^\alpha$ or $\Sigma$ remain.

If $\alpha=1$ and $U=\{t^Q \, : \, t >0\} \times C$,
then a $(U,1)$-stable law is in particular operator-stable with exponent $Q$ (see Section \ref{subsect:operator semistable}),
and necessarily of the form $(\gamma,0,\nu^\alpha)$, where $\int f(x) \nu^\alpha(\dx)=\int_{\R_>} \int_{\Sd} f(t^Qs) \frac{1}{t^2} \, \rho(\ds) \dt$.
Then \cite[Proposition 12 and Theorem 13]{Luczak:2010} give that \eqref{eq:condition for rho} is equivalent to the existence of $\gamma \in \R^d$
such that $(\gamma,0,\nu^\alpha)$ is strictly operator-stable.
In addition, if \eqref{eq:condition for rho} holds, then $(\gamma^1,0,\nu^\alpha)$ is strictly operator-stable,
and $\gamma^1$ inherits $C$-invariance from $\rho$. Thus, $(\gamma^1,0,\nu^\alpha)$ is strictly $(U,1)$-stable
and if $(\gamma,0,\nu^\alpha)$ is strictly $(U,1)$-stable as well, then also $(\gamma^1-\gamma,0,0)$ is $(U,1)$-stable,
which implies that $z\defeq\gamma^1-\gamma$ satisfies $u^\transp z=\norm{u}z$ for all $u\in U$. \qed
\end{proof}

Propositions  \ref{Prop:(U,alpha)-stable laws} and \ref{Prop:solving Psi} together show that all solutions of the homogeneous equation are of the form $Z+Y_W$. It remains to solve the inhomogeneous equation.

\subsection{Proof of Theorem \ref{Thm:solutions to multivariate smoothing equations}: The converse inclusion} \label{subsec:converse inclusion}	\label{subsec:inhomogeneous equation}

Since we have determined all solutions to the homogeneous equation in Proposition \ref{Prop:solving Psi} above,
we can now finish the proof of our main result by proving the converse inclusion in Theorem \ref{Thm:solutions to multivariate smoothing equations}
(the direct inclusion has already been proved in Section  \ref{subsec:proof direct inclusion}).

\begin{proof}[Proof of Theorem \ref{Thm:solutions to multivariate smoothing equations}: The converse inclusion]
Let $X$ be a solution to \eqref{eq:FP of ST inhom} with characteristic function $\phi$.
Write $\Phi_n(x)\defeq\prod_{|v|=n} \phi(L(v)^\transp x)$,
and notice that the multiplicative martingale associated with $\phi$ takes the form
$M_n(x)=\exp(\imag \scalar{x,W_n^*}) \cdot \Phi_n(x)$ (see \eqref{eq:M_n}).
The assumption that $W_n^* \to W^*$ in probability implies 
\begin{equation*}
\Phi_n(x) ~~\to~ M(x)/\exp(\imag\scalar{W^*,x}) ~\eqdef~\Phi(x)
\end{equation*}
in probability. Arguing as in the proof of \cite[Theorem 4.2]{Alsmeyer+Meiners:2012},
it follows that $\psi(x)\defeq\E[\Phi(x)]$ satisfies the functional equation \eqref{eq:FE of ST hom} of the homogeneous smoothing transform
and that $\Phi(x)$ equals the limit of the multiplicative martingale associated with $\psi(x)$,
hence $\Phi(x)=\exp(\Psi(x))$ with $\Psi(x)$ given by Proposition \ref{Prop:solving Psi}. 
We conclude that 
\begin{equation*}
\phi(x)=\E[M(x)]=\E\big[\exp(\imag\scalar{W^*,x})\Phi(x) \big] =\E\big[\exp(\imag\scalar{W^*,x} + \Psi(x)) \big].
\end{equation*}	\qed
\end{proof}

\section{Matrices and measures invariant under actions of similarity groups}	\label{sec:invariant matrices and measures}

In this section, we study the property of $(U,\alpha)$-stability in detail, for arbitrary closed subgroups $U \subseteq \Sim$.
We start by describing the general structure of such groups.
This will allow us to relate $(U,\alpha)$-stable laws to operator semi-stable laws,
and to characterize L\'evy measures and covariance matrices,
satisfying the invariance properties \eqref{eq:bar nu (U,alpha)-invariant} and \eqref{eq:Sigma=oSigmao^T}, respectively.

\subsection{Structure of $U$ and polar coordinates}	\label{app:polar coordinates}	

Let $U \nsubseteq \Orth$ be a closed subgroup of the similarity group $\Sim$, a particular case of which is $\U$,
the closed subgroup  generated by the $T_j$, $j=1,\ldots,N$.
We write $G$ for the image of $U$ under the group homomorphism $u \mapsto \norm{u}$
and distinguish between the discrete case $G=r^\Z$ for some $0<r<1$, and the continuous case $G=\Rp$. As  before, $t^Q\defeq e^{(\ln t) Q}$, and the right-hand side denotes the matrix exponential.

\begin{proposition}	\label{Prop:Buraczewski et al}
Let $\Cu \defeq U \cap \Orth$.
Then there is a subgroup $\Au \subseteq U$ isomorphic to $G$ such that
\begin{equation*}
U ~\simeq~ \Au \ltimes \Cu,
\end{equation*}
in particular, every $u \in U$ has a unique representation $u=a c$ with $a \in \Au, c \in \Cu$. Moreover, in the
\begin{itemize}
\item discrete case: $\Au = \{A^n \, : \, n \in \Z\}$ for some $A \in U$ with $\norm{A}=r$,
\item continuous case: $\Au=\{t^Q \, : \, t \in \Rp\}$ for a $d \times d$-matrix $Q=Q'+c\Id$, where $Q'$ is skew symmetric and $c \neq 0$. $Q$ can be chosen in such a way that $\Au$ and $\Cu$ commute. 
\end{itemize}
\end{proposition}
\begin{proof}
The structure of $U$ is given by Proposition C.1 in \cite{Buraczewski+al:2009}, commutativity in the continuous case is proved in Proposition D.13 in \cite{Buraczewski+al:2009}. It is proved there that (in the continuous case) $Q$ is an element of the Lie Algebra of $U$, i.e., $e^Q \in U \subseteq \Sim$. Any orthogonal matrix is the exponential of a skew symmmetric matrix, hence a similarity matrix $u$ with $\norm{u}\neq 1$ is the product of an orthognal matrix times a scalar multiple of the identity matrix, which we represent by $e^{(\ln c) \Id}$. \qed
\end{proof}

The multiplication $\cdot:(\Au \ltimes \Cu)^2 \to \Au \ltimes \Cu$ is defined via the conjugation action induced on $\Cu$
by elements of $\Au$:
\begin{equation*}
(a_1,c_1) \cdot (a_2,c_2) = (a_1a_2,\, a_2^{-1}c_1a_2 c_2),	\quad	(a_1,c_1), (a_2,c_2) \in \Au \times \Cu.
\end{equation*}
Notice that  $a_2^{-1}c_1a_2 c_2 \in \Cu$ since $\Cu$ is a normal subgroup of $U$.
Further note that when all elements of $\Au$ commute with all elements of $\Cu$,
multiplication on $\Au \ltimes \Cu$ simplifies to
\begin{equation*}
(a_1,c_1) \cdot (a_2,c_2) = (a_1a_2,c_1 c_2),	\quad	(a_1,c_1), (a_2,c_2) \in \Au \times \Cu
\end{equation*}
and hence, in this case, $\Au \ltimes \Cu = \Au \times \Cu$,
that is, $U$ is isomorphic to the direct product $\Au \times \Cu$.

According to the two cases considered, we introduce generalized polar coordinates as follows. Let
\begin{equation}	\label{eq:SL}
\Su ~\defeq~ \begin{cases}
\{ x \in \R^d \, : \, \abs{x}=1 \}=\Sd & \text{ if } G=\Rp, \\
\{x \in \R^d \, : \,  r \le \abs{x} < 1 \} & \text{ if } G= r^{\Z}, \text{  } 0 < r <1.
\end{cases}
\end{equation}
Then any $x \in \R^d\setminus\{0\}$ has a unique representation $x=as$ with $s \in \Su$ and $a \in \Au$.
Notice that in general, $a$ is not a scalar. For example, in the setting of cyclic P\'olya urns, Section \ref{subsubsec:Cyclic Polya urn}, 
\begin{equation*}
\AU=\{ t^\zeta \, : \, t \in \Rp \}, \qquad \CU=\{ \zeta^k \, : \, 0 \le k < b \}, \qquad \SU=\S^1,
\end{equation*}
where $\zeta=\cos(2\pi/b) + \imag \sin(2 \pi/b)$ is a primitive $b$th root of unity.

\subsection{Operator (semi)stable laws} \label{subsect:operator semistable}
An infinitely divisible law on $\R^d$ with characteristic exponent $\Psi$ is called $(A,c)$-{\em operator semistable} if there is a $d \times d$-matrix $A$, $b \in \R^d$ and $c \in (0,1)$ such that
\begin{equation} \label{eq:operator semistable}
\Psi(A^\transp x) ~=~c\Psi(x) + \imag\scalar{x,b} \quad \text{ for all } x \in \R^d,
\end{equation}
see \cite[Definition 1.3.6]{Hazod+Siebert:2001}. 
It is called {\em operator stable} with exponent $Q$ if there is a matrix $Q$ and a mapping $s \mapsto b(s) \in \R^d$ such that 
\begin{equation} \label{eq:operator stable}\Psi((t^Q)^\transp x) ~=~t\Psi(x) + \imag\scalar{x,b(t)} \quad \text{ for all } x \in \R^d,\ t>0,\end{equation}
see \cite[Definition 1.3.11]{Hazod+Siebert:2001}. The law is called {\em strictly} operator (semi)stable if $b=0$ or $b(t) \equiv 0$, respectively.
Recalling the definition of $(U,\alpha)$-stability in \eqref{eq:(u,alpha)-stable chf}, and using the structure of $U$ given in Proposition \ref{Prop:Buraczewski et al}, we obtain the following Corollary.

\begin{corollary}\label{Cor:operator stability}
Let $\eta$ be a (strictly) $(U,\alpha)$-stable law. 
\begin{enumerate}
\item If $\Au=\{A^n \, : \, n \in \Z\}$ with $\norm{A}<1$, then \eqref{eq:operator semistable} holds with $c=\norm{A}^{\alpha}$, i.e., $\eta$ is $(A, \norm{A}^\alpha)$ (strictly) operator semistable.
\item If $\Au=\{e^{sQ} \, : \, s \in \R\}$, then \eqref{eq:operator stable} holds upon rescaling $Q$ such that $\norm{t^Q}^\alpha=t$, i.e., $\eta$ is (strictly) operator stable with exponent $Q$.
\end{enumerate}
\end{corollary}

Notice that $\Cu$ did not play a role in the above considerations, therefore $(U,\alpha)$-stability is more restrictive than operator (semi)stability. 

\subsection{L\'evy measures invariant under similarity transformations}	\label{subsec:U-invariant Levy measures}

Using the generalized polar coordinates introduced in Section \ref{app:polar coordinates} above, we can now describe the structure of L\'evy measures satisfying \eqref{eq:bar nu (U,alpha)-invariant}. Of course, here $\U$ can be any closed subgroup of $\Sim$. We write $\HAU$ for the Haar measures on $\AU \simeq \G$, which is the counting measure in the discrete case, and the image of $\dt/t$ under the map $t \mapsto t^Q$ in the continuous case.

\begin{proposition}	\label{Prop:structure of bar nu} 

Let $\bar \nu$ be a L\'evy measure on $\R^d \setminus \{0\}$.
Then the following assertions are equivalent:
\begin{itemize}
	\item[(i)]
		$\bar\nu$ satisfies \eqref{eq:bar nu (U,alpha)-invariant}.
	\item[(ii)]
		There is a $\CU$-invariant finite measure $\rho$ on $\SU$ such that for all $\bar\nu$-integrable $f : \R^d \setminus\{0\} \to \R$,
		\begin{equation}	\label{eq:structure of bar nu}
		\int f(x) \, \bar\nu(\dx) ~=~  \int_{\AU } \int_{\SU} f(ax) \norm{a}^{-\alpha} \, \rho(\dx)\, \HAU(\da).
		\end{equation}
\end{itemize}
\end{proposition}

Notice that since $\AU$ is not compact, the Haar measure $\HAU$ on $\AU$ is unique up to a positive scaling constant only.
In the discrete case, we stipulate that $\HAU$ is the counting measure.
In the continuous case, we stipulate that $\HAU$ is such that the pushforward measure of $\HAU$ under the map
$\AU \to \Rp$, $a \mapsto \norm{a}$ is $\dt/t$.

\begin{proof}
We consider the discrete and continuous case separately. 

In the discrete case, \eqref{eq:bar nu (U,alpha)-invariant} yields that $\bar{\nu}(A^{-1} \cdot)= \norm{A}^\alpha \bar{\nu}$. 
Recall the definition of $\SU$ from \eqref{eq:SL}.
Setting $\rho\defeq\bar{\nu}(\cdot \cap \SU)$,
we first observe that by \eqref{eq:bar nu (U,alpha)-invariant}, $\bar{\nu}$ is $\CU$-invariant
and hence so is $\rho$.
Further, \cite[Theorem 1.4.5]{Hazod+Siebert:2001}
implies that for any Borel set $B \subseteq \R^d\setminus\{0\}$,
\begin{align*}
\bar{\nu}(B) =& \sum_{n=-\infty}^\infty \norm{A}^{n \alpha} \rho \big( (A^n B) \cap \SU \big)
=	\int_{\AU} \norm{a}^{-\alpha} \rho\big( (a^{-1} B) \cap \SU \big) \, \HAU(\da) \\
=&	\int_{\AU}  \norm{a}^{-\alpha} \int_{\SU} \1_B(ax) \, \rho(\dx) \, \HAU(\da). 
\end{align*}
This proves (i) $\Rightarrow$ (ii) in the discrete case.

For the converse implication, we invoke \cite[Theorem 1.4.4]{Hazod+Siebert:2001}
which gives that each measure satisfying \eqref{eq:structure of bar nu} is a L\'evy measure
with $\bar{\nu}(A^{-n} \cdot) =\norm{A}^{n\alpha} \bar{\nu}$,
i.e., $\bar{\nu}(a^{-1} \cdot)=\norm{a}^\alpha \bar{\nu}$ for all $a \in \AU$.
Let $\U \ni u=a_0c$ with $a_0 \in \AU$, $c \in \CU$ and let $B$ be a Borel subset of $\R^d \setminus \{0\}$.
In general, $\AU$ and $\CU$ do not commute,
but since $(a_0c)^{-1}$ is an element of  $\U$ with norm $\norm{(a_0c)^{-1}}=\norm{a_0}^{-1}$,
there is a $c' \in \CU$ such that $(a_0c)^{-1}=a_0^{-1}c'^{-1}$.
Then
\begin{align*} \label{eq:CU invariance of bar nu}
\bar{\nu}(c^{-1}a_0^{-1} B)
&~=~ \bar{\nu}(a_0^{-1}c'^{-1}B) = \norm{a_0}^\alpha \bar{\nu}(c'^{-1}B) \\
&~=~  \norm{a_0}^\alpha \int_{\AU } \int_{\SU} \1_B(c'ax) \norm{a}^{-\alpha} \rho(\dx) \, \HAU(\da) \\
&~=~  \norm{a_0}^\alpha \int_{\AU } \int_{\SU} \1_B(ac_a'x) \norm{a}^{-\alpha} \rho(\dx) \, \HAU(\da) \\
&~=~  \norm{a_0}^\alpha \int_{\AU } \int_{\SU} \1_B(ax) \norm{a}^{-\alpha} \rho(\dx) \, \HAU(\da)	\\
&~=~ \norm{a_0c}^\alpha \bar{\nu}(B),
\end{align*}
where the $c_a' = a^{-1} c' a \in \CU$.
In the next-to-last line, the $\CU$-invariance of $\rho$ was used.
Thus (ii) $\Rightarrow$ (i) is proved in the discrete case.

Turning to the implication (i) $\Rightarrow$ (ii) in the continuous case, 
choose $Q$ in such a way that $\norm{t^Q}^\alpha=t$.
Then \eqref{eq:bar nu (U,alpha)-invariant} for $(t^{Q})^{-1}$ becomes $\bar{\nu}((t^Q)^{-1} \, \cdot) = t \bar{\nu}(\cdot)$.
Define
\begin{equation} \label{eq:rho continuous case}
\rho(B) ~\defeq~\bar{\nu} \big( \{ t^Q x \, : \, x \in B, \, t \geq 1 \}\big),	\quad	B \subseteq \Sd.
\end{equation}
By Proposition \ref{Prop:Buraczewski et al}, $t^Q$ and $\CU$ commute for every $t >0$.
Thus, we infer from Eqs.\ \eqref{eq:rho continuous case} and \eqref{eq:bar nu (U,alpha)-invariant} that
\begin{equation*}
\rho(cB) = \bar{\nu} \big( c \cdot \{ t^Q x \, : \, x \in \Sd, \, t\ge 1 \}\big) = \rho(B)
\end{equation*}
for all $c \in \CU$ and all Borel sets $B \subseteq \Sd$, i.e., $\rho$ is $\CU$-invariant.
Moreover, \cite[Theorem 1.4.12]{Hazod+Siebert:2001} gives that
\begin{align}
\label{eq:bar nu operator stable} \bar{\nu}(C) ~&=~ \int_{\SU} \int_0^\infty \1_C(t^Q x) t^{-2}   \, \dt \, \rho(\dx) \\
&=~\int_{\Rp} \int_{\SU}  \1_C(t^Q x) \norm{t^Q}^{-\alpha}  \rho(\dx) \, \frac{\dt}{t} \nonumber \\
&=~\int_{\AU} \int_{\SU}  \1_C(a x) \norm{a}^{-\alpha}  \rho(\dx) \, \HAU(\da) . \nonumber
\end{align}
Here we used the particular scaling of $Q$ and the fact that $\dt/t$ is the pushforward measure of $\HAU$
under $a \mapsto \norm{a}$.
Thus the implication (i) $\Rightarrow$ (ii) is proved. 

For the converse implication, we use that each $u \in \U$ is of the form $u=t^Q c$ for some $t>0$ and $c \in \CU$.
\cite[Theorem 1.4.11]{Hazod+Siebert:2001} gives that if $\bar{\nu}$ satisfies \eqref{eq:structure of bar nu},
then $\bar{\nu}$ is a L\'evy measure satisfying $\bar{\nu}((t^Q)^{-1} C) = t \bar{\nu}(C)$.
Validity of \eqref{eq:bar nu (U,alpha)-invariant} for all $u \in \U$ then follows as in the discrete case. \qed
\end{proof}

\subsection{Matrices invariant under orthogonal transformations}	\label{subsec:oSigmao^t}

In this section, we analyze the structure of a positive semi-definite $d \times d$ matrix $\Sigma$
satisfying
\begin{equation}	\label{eq:oSigmao^t}
o \Sigma o^\transp ~=~ \Sigma
\quad \text{ for all } o \in \O.
\end{equation}
We defined $\O$ as the closed subgroup generated by the $(O(j))_{j \ge 1}$, but it can be any closed subgroup of $\Orth$.

The main result of this section is the following proposition.
To formulate it, we recall two notions.
We say that a subspace $V$ of $\R^d$ is $\O$-invariant if $oV=V$ for all $o \in \O$,
and $\O$-indecomposable if it does not contain any nontrivial $\O$-invariant subspace.

\begin{proposition}	\label{Prop:oSigmao^t}
Let $\Sigma$ be a positive semi-definite symmetric matrix, satisfying \eqref{eq:oSigmao^t}. Then there is a decomposition
\begin{equation}	\label{eq:O-invariant subspaces}
\R^d ~=~ V_+ \oplus V_- \oplus V_1 \oplus \ldots \oplus V_l
\end{equation}
into $\O$-invariant orthogonal subspaces with the following properties:
\begin{itemize}
	\item[(i)]
		Every $o \in \O$ is the identity mapping on $V_+$ and minus the identity on $V_-$.
		$V_+$ and $V_-$ are the maximal $\O$-invariant subspaces with these properties.
		Further, $V_+$ and $V_-$ are $\Sigma$-invariant subspaces and the restrictions
		$\Sigma_{|V_\pm}$ are positive semi-definite symmetric matrices.
	\item[(ii)]
		For each $i=1,\ldots,l$, $V_i$ is $\O$-indecomposable and $\Sigma$-invariant,
		and $\Sigma_{|V_i}$ is a nonnegative scalar multiple of the identity mapping on $V_i$.
\end{itemize} 
\end{proposition}

One ingredient in the proof of Proposition \ref{Prop:oSigmao^t} is the following variant of Schur's lemma, see e.g.\ \cite[Corollary XVIII.6.2]{Lang:1993}.

\begin{lemma}\label{Lem:Schur}
Let $F$ be a family of real $d \times d$ matrices,
and suppose that $\{0\}$ and $\R^d$ are the only subspaces of $\R^d$ that are invariant for each matrix in $F$.
If $\Sigma$ is a symmetric matrix that commutes with every matrix in $F$, then $\Sigma = c \Id$ for a constant $c \in \R$. 
\end{lemma}

For the proof of Proposition \ref{Prop:oSigmao^t}, we need another lemma.

\begin{lemma}	\label{Lem:oSigmao^t}
Let $\Sigma$ be a positive semi-definite symmetric $d \times d$ matrix and $o \in \Orth$
such that $\Sigma o = o \Sigma$.
Let 
\begin{equation}	\label{eq:decomposition for o}
\R^d	~=~	V_+ \oplus V_- \oplus V_1 \oplus \dots \oplus V_k
\end{equation}
be a decomposition of $\R^d$ into orthogonal $o$-invariant subspaces
where $V_\pm = E_{\pm1}(o)$ and, for each $i=1,\ldots,k$,
$V_i$ is a $2$-dimensional, $o$-indecomposable subspace on which $o$ acts as a rotation by an angle $\pi \not = \theta_i \in (0,2\pi)$.

Then $V_+, V_-, V_1, \ldots, V_k$ are $\Sigma$-invariant as well.
\end{lemma}
\begin{proof}
We have $o^{-1}=o^\transp$ since $o \in \Orth$ and multiplication of $\Sigma o = o \Sigma$ with $o^\transp$ from the left and right
yields that $o^\transp$ and $\Sigma$ commute.
Hence so does $a \defeq (o + o^\transp)$.
As symmetric matrices, $a$ and $\Sigma$ are diagonalizable.
This implies that the eigenspaces of $a$ are $\Sigma$-invariant and vice versa.
To see this, pick an eigenvector $v$ of $a$ corresponding to the eigenvalue $\lambda$
and notice that
\begin{equation*}
\lambda (\Sigma v) ~=~ \Sigma a v ~=~ a (\Sigma v),
\end{equation*}
i.e., $\Sigma$ maps the eigenspace $E_\lambda(a)$ into itself. 

The eigenvalues of $a$ can easily be computed:
Let $b \in \Orth$ be such that $bob^\transp$ is in normal form, i.e.,
\begin{equation*}
b o b^\transp ~=~
\begin{pmatrix}
I_{d'}	& 			& 				&			& 0 \\
		& -I_{d''}	&				&			& \\
		&			& R_{\theta_1}	&			& \\
		&			&				& \ddots	& \\
0		&			&				&			& R_{\theta_k} \\
\end{pmatrix}
\end{equation*}
for $2 \times 2$-rotation matrices $R_{\theta_i}$, $i=1,\ldots,k$.
Then $b o^\transp b^\transp = (bob^\transp)^\transp$ is in normal form, too, and $b(o + o^\transp)b^\transp$ is a diagonal matrix,
with diagonal entries $+2$, $-2$ and $2 \cos (\theta_i)$, $i=1,\ldots,k$.
The corresponding eigenspaces are $V_+$, $V_-$ and $V_i$, $i=1,\ldots,k$ if all $\theta_i$ are distinct.
In this case, the proof is complete.

Suppose there is an $i \in \{1,\ldots,k\}$ such that $\theta_i = \theta_j$ for some $j \not = i$.
Then $V_i \oplus V_j \subseteq E_{2 \cos (\theta_i)}(a)$. Further,
\begin{equation*}
o \Sigma V_i = \Sigma o V_i = \Sigma V_i,
\end{equation*}
i.e., $\Sigma V_i$ is $o$-invariant.
By the reasoning in the beginning of the proof, $\Sigma V_i \subseteq E_{2 \cos (\theta_i)}(a)$.
Hence $\Sigma V_i = \{0\} \subseteq V_i$ or $V_i=V_j$ for some $j$ with $\theta_i=\theta_j$. 

We show that $\Sigma V_i = V_j$ for $i \neq j$ is impossible.
Choosing a basis of $E_{2 \cos (\theta_i)}(a)$,
$\Sigma$ has to be a symmetric matrix w.r.t.~this basis.
In particular, it has to permute $V_i$ and $V_j$.
Hence $V_i \oplus V_j$ is $\Sigma$-invariant. With respect to a joint basis of $V_i$ and $V_j$, 
\begin{equation*}
\Sigma_{|V_i \oplus V_j} ~\sim~
\begin{pmatrix}
0 & A^\transp \\
A & 0
\end{pmatrix}
\end{equation*}
for a $2 \times 2$-matrix $A$ with non-vanishing determinant.
Thus, $\det(\Sigma_{|V_i \oplus V_j})$ is negative, which violates the Hurwitz criterion for positive semi-definiteness.  
 \qed
\end{proof}

\begin{proof}[Proof of Proposition \ref{Prop:oSigmao^t}]
For each $o \in \O$ there is an individual decomposition of the form \eqref{eq:decomposition for o},
we denote its components by $V_{\pm}(o)$ etc.

Defining $V_+ \defeq \bigcap_{o \in \O} V_+(o)$ and $V_- \defeq \bigcap_{o \in \O} V_-(o)$,
we obtain $\Sigma$- and $\O$-invariant orthogonal subspaces, on which each $o$ acts as the identity
or minus the identity, respectively. $\overline{V}\defeq(V_+ \oplus V_-)^\perp$ is $\Sigma$- and $\O$-invariant as well. 

We consider the set
\begin{equation*} \mathcal{V} ~\defeq~ \{ V_1 \oplus \dots \oplus V_l \, : \,  \{0\} \neq  V_j \subseteq \overline{V} \text{ is  a $\Sigma$- and $\O$-invariant subspace}\}.\end{equation*}
We do not distinguish between decompositions that consist of the same subspaces, but in a different order.
$\mathcal{V}$ is non-empty since it contains $\overline{V}$.
The set $\mathcal{V}$ possesses a partial order $\prec$:
a decomposition is larger than another one if the former is a refinement of the latter. 
Any totally ordered subset of $\mathcal{V}$ has at most $d$ elements
since every strict refinement decreases the dimension of at least one subspace by at least 1.
Now pick a totally ordered subset of $\mathcal{V}$ with maximal number of elements
and within this subset pick the largest element, $V_1 \oplus \dots \oplus V_l$, say.
Clearly, $V_1 \oplus \dots \oplus V_l$ is maximal with respect to $\prec$.
We claim that $V_1,\ldots,V_l$ are $\O$-indecomposable.
If not, then there is a $V_j$ which contains two proper orthogonal subspaces
that are invariant under every $o \in \O$,
and hence also $\Sigma$-invariant by Lemma \ref{Lem:oSigmao^t}.
This contradicts the maximality of the decomposition $V_1 \oplus \dots \oplus V_l$ with respect to $\prec$.

Consequently, we can decompose $\Sigma$ according to a maximal element of $\mathcal{V}$.
In particular, it suffices to solve 
\begin{equation*}
o_{|V} \,  \Sigma_{|V} \,  o^\transp_{|V} ~=~ \Sigma_{|V} \quad	\text{for all } o \in \O
\end{equation*}
separately for $V \in \{V_+, V_-, V_1, \ldots, V_l\}$.

The restriction of $\Sigma$ to $V_+$ or $V_-$ can be any positive semi-definite symmetric matrix,
for conjugation by $o_{|V_\pm}$ is the identity mapping for all $o \in \O$.
To identify $\Sigma_{|V_i}$, where $V_i$ is an $\O$-invariant and indecomposable subspace, we use Schur's lemma (Lemma \ref{Lem:Schur}).
This yields that each $\Sigma_{|V_i}$ is a scale multiple of the identity  on $V_i$, 
and since $\Sigma_{|V_i}$ is positive semi-definite, the scaling factor is nonnegative.
 \qed
\end{proof}
\appendix

\section{The Choquet-Deny lemma}	\label{app:Choquet-Deny}

Given a probability measure $\mu$ on the similarity group $\Sim$,
let $U$ be the closed subgroup generated by the support of $\mu$---if
$\mu$ is the step distribution of the associated multiplicative random walk
$(L_n)_{n \in \N_0}$, see Section \ref{subsec:change of measure}, then $U= \U$.  

\begin{lemma}	\label{Lem:Choquet-Deny}
Let $\psi :U \to \R$ be measurable and bounded. If
\begin{equation}
\int_{U} \psi(ug) \, \mu(\mathrm{d} u) ~=~ \psi(g)
\end{equation}
for all $u \in U$, then $\psi$ is constant $\mu$-a.e.
\end{lemma}

This is a consequence of \cite[Theorem 3]{Guivarch:1973}.
For the reader's convenience, we state that theorem and show how the lemma can be derived from it.

In the following, let $G$ be a locally compact, separable and unimodular group.
A probability measure $\mu$ on $G$ is called \emph{aperiodic},
if the closed subgroup generated by the support of $\mu$ equals $G$.
Write $[G,G]$ for the commutator subgroup, i.e., the group generated by the commutators
$[a,b] \defeq (ba)^{-1} ab$, $a,b \in G$, and $\overline{[G,G]}$ for its closure.
Let $H \subsetneq G$ be a normal subgroup 
of $G$. Then $G$ acts on $H$ by conjugation (inner automorphisms), i.e.,
\begin{equation*}
g.h ~\defeq~ g^{-1} h g,	\quad g \in G, h \in H.
\end{equation*}
For $A \subseteq H$ write
\begin{equation*}
A^G ~\defeq~ \{g.a:\, g \in G, a \in A\}.
\end{equation*}
The action of $G$ on $H$ is said to be \emph{compact} if for each compact $A \subseteq H \setminus \{1_G\}$,
$1_G$ the unit element of $G$, $A^G$ is relatively compact, i.e., has compact closure.
Then \cite[Theorem 3]{Guivarch:1973} reads as follows:

\begin{theorem}	\label{Thm:Guivarch}
Let $\mu$ be an aperiodic probability measure on $G$.
If $\overline{[G,G]}$ is Abelian or compact and if the action of $G$ on $\overline{[G,G]}$ is compact,
then the only bounded, measurable functions $\psi$ satisfying
\begin{equation*}
\int \psi(u g) \, \mu(\mathrm{d}u) ~=~ \psi(g)	\quad	\text{for all } g \in G
\end{equation*}
are the $\mu$-almost everywhere constant functions.
\end{theorem}

Following the proof of \cite[Theorem A.1]{Buraczewski+al:2009},
we show how Theorem \ref{Thm:Guivarch} applies to the situation here, i.e.,
$G=U$ is the closed subgroup generated by the support of $\mu$. 
Then $\mu$ is aperiodic on $U$ by the very definition of $U$.
Referring to Proposition \ref{Prop:Buraczewski et al}, there is a closed subgroup $\Au$ of $U$,
which is isomorphic to a closed subgroup of the multiplicative group $\R_>$,
and a normal compact subgroup $\Cu=U \cap \Orth$, such that $U/\Cu\simeq\Au$.
The groups $\Au$ and $\Cu$ (as a compact group, see \cite[Theorem 1.4.1]{Deitmar+Echterhoff:2009}) are unimodular and hence,
by the Fubini formula for the Haar measure on $U=\Au \Cu$, \cite[Proposition 1.5.5]{Deitmar+Echterhoff:2009},
$U$ is unimodular as well.

Clearly, the commutator subgroup of $U$ is a subgroup of $\Orth$, hence its closure is compact.
Moreover, for any compact $A \subseteq \overline{[U,U]} \setminus \{\Id\}$, $A^U$ is again a subset of $\Orth$
since 
\begin{equation*}
u.[a,b] ~=~ u^{-1} [a,b]\, u ~=~ o^{-1} [a,b]\, o,
\end{equation*}
where $u = \norm{u} o$ with $o \in \Cu \subseteq \Orth$. Hence, $A^U$ is relatively compact as a subset of a compact set.

\section{Evaluating the L\'evy integrals}\label{Sect:Levy integrals}

In this section, we compute
\begin{equation}	\label{eq:I(x)}
I(x) = \int_{\R^d \setminus \{0\}} \left(e^{\imag \scalar{x,y}} - 1 - \imag \scalar{x,y} \1_{[0,1]}(\abs{y}) \right) \bar\nu(\dy),	\quad	x \in \R^d
\end{equation}
for a deterministic $(U,\alpha)$-invariant L\'evy measure $\bar\nu$, i.e.\ satisfying \eqref{eq:bar nu (U,alpha)-invariant}.

\begin{lemma}	\label{Lem:I(tx) evaluated}
Let $\bar\nu$ be a deterministic L\'evy measure satisfying \eqref{eq:bar nu (U,alpha)-invariant} for some $1 \neq \alpha \in (0,2)$,
and define $I(x)$ via \eqref{eq:I(x)}. Then, for $0 \not = x \in \R^d$, 
\begin{equation}	\label{eq:I(tx) evaluated}
I(x)	~=~	
				- |x|^\alpha \eta_1^\alpha(x) + \imag |x|^\alpha \eta_2^\alpha(x) + \imag \scalar{x,\gamma^\alpha},
\end{equation}
with functions $\eta_1^\alpha,\eta_2^\alpha$ defined in \eqref{eq:Definition eta1} and \eqref{eq:Definition eta2}, respectively.
$\eta_1^\alpha$ and $\eta_2^\alpha$ are bounded real functions
satisfying $\eta_j^\alpha(u^\transp x) = \eta_j^\alpha(x)$ for all $u \in \U$, $x \in \R^d\setminus\{0\}$, $j=1,2$,
and $\eta_1^\alpha$ is nonnegative.
The vector $\gamma^\alpha$ satisfies $o \gamma^\alpha = \gamma^\alpha$ for all $o \in \CU$.
\end{lemma}
\begin{proof}
Fix $0 \not = x \in \R^d$ and notice that $I(x)$ is finite since $\bar\nu$ is a L\'evy measure.
Further, according to Proposition \ref{Prop:structure of bar nu},
there is a $\CU$-invariant finite measure $\rho$ on $\SU$ such that \eqref{eq:structure of bar nu} holds.

We set
\begin{equation*}
\gamma^\alpha
\defeq
\begin{cases}
-\int	y \1_{[0,1]}(|y|) \bar\nu(\dy)
= - 	\int_{\AU} \int_{\SU} ax \1_{[0,1]}(|ax|) \norm{a}^{-\alpha} \rho(\dx) \, \HAU(\da)	&	\text{ if } \alpha < 1	\\
- \int_{\AU} \int_{\SU} ax \1_{(1,\infty)}(|ax|) \norm{a}^{-\alpha} \rho(\dx) \, \HAU(\da),	&	\text{ if } \alpha \in (1,2).
\end{cases}
\end{equation*}
The asserted $\CU$-invariance follows from the $\CU$-invariance of $\rho$.

Recalling the definitions
\begin{align*}
\eta_1^\alpha(x) ~&=~ \frac{1}{|x|^\alpha} \int \big(1 - \cos (\scalar{x,y}) \big) \nu^\alpha(\dy) \\
\eta_2^\alpha(x) ~&=~ \frac{1}{|x|^\alpha} \int \big(\sin (\scalar{x,y}) - \1_{\{\alpha >1\}}\scalar{x,y} \big) \nu^\alpha(\dy) ,
\end{align*}
\eqref{eq:I(tx) evaluated} holds, and it remains to prove boundedness and invariance properties of $\eta_i^\alpha$, $i=1,2$.
Let $u \in \U$, then, using \eqref{eq:bar nu (U,alpha)-invariant},
\begin{align*}
\eta_1^\alpha(u^\transp x) ~&=~ \frac{1}{\norm{u}^\alpha|x|^\alpha} \int \big(1 - \cos (\scalar{x,uy}) \big) \nu^\alpha(\dy) \\
~&=~ \frac{\norm{u}^\alpha}{\norm{u}^\alpha|x|^\alpha} \int \big(1 - \cos (\scalar{x,y}) \big) \nu^\alpha(\dy) ~=~\eta_1^\alpha(x),
\end{align*} 
and the invariance of $\eta_2^\alpha$ is proved along the same lines. This implies in particular that the continuous functions $\eta_i^\alpha$ are determined by their respective values on the (relative) compact set $\SU$, hence the asserted boundedness follows.
 \qed
\end{proof}

For $\alpha=1$, we can compute a meaningful expression for $\eta^1(x)$ only for $x \in E_1(Q^\transp)$. Notice that this implies $x \in E_1(t^Q)$ for all $t \ge 0$. 
Hence, using formula \eqref{eq:bar nu operator stable} for $\bar{\nu}$,  we obtain
\begin{align}
\eta^1(x)
~&=~\int_{\Sd} \int_{\R_>} \bigg( e^{\imag \scalar{(t^{Q})^\transp x,s}} - 1 - \imag\scalar{(t^Q)^\transp x,s}\1_{\{|(t^Q)^\transp s| \le 1\}}  \bigg) t^{-2}\, \dt\, \rho(\ds) \nonumber \\
&=~\int_{\Sd} \int_{\R_>} \bigg( e^{\imag \scalar{tx,s}} - 1 - \imag\scalar{tx,s}\1_{\{t \le 1\}}  \bigg) t^{-2}\, \dt\, \rho(\ds) \nonumber \\
&=~-\int_{\Sd} \bigg( \abs{\scalar{x,s}} + \imag \frac2\pi \scalar{x,s} \log|\scalar{x,s}| \bigg) \rho(\ds) + \imag\scalar{\gamma,x}  \label{eq:Psi alpha 1 a}
\end{align}
for a suitable $\gamma \in \R^d$, see \cite[Theorem 14.10]{Sato:1999} for details.

\section{Auxiliary results from (Markov) renewal theory}

\begin{lemma}	\label{Lem:no big jump over t}
Let $(S_n)_{n \in \N_0}$ be a random walk with i.i.d.\ increments, $S_0=0$, $\E[S_1]>0$ and $\E[(S_1^+)^2]<\infty$.
Let $\tau(t) \defeq \inf\{n \in \N_0: S_n > t\}$, $t \geq 0$.
Then, for every $0<a<1$,
\begin{equation}	\label{eq:no big jump over t}
t \Prob(S_{\tau(t)-1} < at)	~\to~	0	\quad	\text{as } t \to \infty.
\end{equation}
\end{lemma}
\begin{proof}
$\E[(S_1^+)^2]<\infty$ implies $\lim_{t \to \infty} t^2\Prob(S_1 > t) = 0$.
Further, it is known from standard random walk theory that $\lim_{t\to\infty} t^{-1} \E[\tau(t)] = \E[S_1]^{-1}$.
Consequently, setting, for $n \in \N$, $A_n \defeq  \{S_0 \leq t, \ldots, S_{n-2} \leq t, S_{n-1} < at\}$, we have
\begin{eqnarray*}
t \Prob(S_{\tau(t)-1} < at)
& = &
t \sum_{n \geq 1} \Prob(A_n \cap \{S_{n+1}>t\})
~\leq~
t \sum_{n \geq 1} \Prob(A_n) \Prob(S_{1}>(1-a)t) \\
& \leq &
t^2 \Prob(S_{1}>(1-a)t)  \cdot \frac1t \sum_{n \geq 1} \Prob(\tau(t) \geq n)	\\
& = &
t^2 \Prob(S_1 > (1-a)t) \cdot \frac{\E[\tau(t)]}{t}
~\to~	0	\qquad	\text{as } t \to \infty.	\hfill
\end{eqnarray*}	\qed
\end{proof}

\subsubsection*{A rate-of-convergence result in Markov renewal theory}

Throughout this section, let $((O_n, S_n))_{n \in \N_0}$ be a random walk on $\Orth \times \R$ with increment law $\mu$, say.
The components $O_n$ and $S_n$ may be dependent.
We assume that $\mu$ satisfies the minorization condition \eqref{eq:minorization}
and that
\begin{equation}	\label{eq:moments Markov random walk}
\E[S_1] > 0 \text{ and } \E\big[ |S_1|^{\ell + 1 + \delta} \big] < \infty
\end{equation}
for some $\ell > 0$.
Let $\O$ be the closed subgroup of $\Orth$ generated by the support of $O_1$.

\begin{proposition}	\label{prop:uniform renewal}
Let $\ell >0$ and $g: \R \to [0,\infty)$ be a measurable function that is decreasing on $[0,\infty)$, with $g(t)=0$ for all $t<0$
and $\lim_{t \to \infty} t^{\ell+1+\epsilon} g(t) =0$ for some $0<\epsilon< (\delta \wedge 1)$. Then 
\begin{equation*}
\lim_{t \to \infty} \sup_{\abs{f} \leq g}
~t^{\ell+\epsilon} \abs{\E \bigg[ \sum_{n=0}^\infty f(O_n, t-S_n) \bigg] ~-~ \frac{1}{\E[S_1]} \int_0^\infty \int_\O f(o, r) \, H_\O(\ddo) \, \dr}
~=~0.
\end{equation*} 
\end{proposition}
Here and below, $\sup_{|f| \leq g}$ means the supremum over all measurable functions
$f : \O \times \R \to \R$ satisfying $ \sup_{o \in \O} |f(o,x)| \leq g(x)$ for all $x \in \R$.

The main ingredient in the proof will be the use of regeneration techniques for general state space Markov chains
as developed in \cite{Athreya+Ney:1978a,Nummelin:1978}. We sum up what is needed in the subsequent lemma.

\begin{lemma}	\label{lem:regeneration}
There is a measurable space $(\Omega, \mathcal{G})$ together with a family of probability measures $(\Prob_{o,r})_{o \in \O, r \in \R}$
and sequences of random variables $((M_n, R_n))_{n \geq 0}$ and $(\tau_n)_{n \geq 1}$ with
\begin{equation}	\label{eq:MnRn}
\Prob_{o,r}\big(((M_n, R_n))_{n \geq 0} \in \cdot  \big) ~=~ \Prob\big( ((oO_n, r + S_n))_{n \geq 0} \in \cdot \big)
\end{equation}
for all $o \in \O$, $r \in \R$. Further, the following properties hold:
\begin{enumerate}
	\item[(i)]
		There is a filtration $(\mathcal{G}_n)_{n \in \N}$ such that $((M_n, R_n))_{n \ge 0}$ is Markov adapted to $(\mathcal{G}_n)_{n \in \N}$,
		and $(\tau_n)_{n \geq 1}$ is a sequence of predictable $(\mathcal{G}_n)_{n \in \N}$-stopping times, i.e., $\{\tau_n=k\} \in \mathcal{G}_{k-1}$.
	\item[(ii)]
		There are probability measures $\nu$ on $\O$ and $\eta$ on $\R$, $\eta$ having a bounded Lebesgue density,
		such that for all $n \geq 1$ and $o \in \O$, under $\Prob_o$,
		$((M_{\tau_n+k}, R_{\tau_n+k}-R_{\tau_n-1}))_{0 \leq k \leq \tau_{n+1}-\tau_n-1}$
		is independent of $(M_0, R_0, \dots, M_{\tau_n-1}, R_{\tau_n -1})$
		and has law $\Prob_{\nu \otimes \eta}(((M_{k}, R_{k}))_{k=0}^{\tau_{1}-1} \in \cdot)$.
	\item[(iii)]
		For each bounded measurable function $f:\O \to \R$,
		\begin{equation}	\label{eq:occupation measure}
		\E_{\nu \otimes \eta} \bigg[ \sum_{n=0}^{\tau_1-1} f(M_n) \bigg] ~=~ \E_{\nu \otimes \eta}[\tau_1] \, \int_\O f(o) \, H_\O(\ddo).
		\end{equation}
	\item[(iv)]
		There are $C, \lambda>0$ such that $\Prob_{o,r}(\tau_1 > n) \le C e^{-\lambda n}$ for all $o \in \O$, $r \in \R$ and $n \in \N$.
\end{enumerate}
\end{lemma}

Here and below, we use the shorthand $\Prob_{\nu \otimes \eta} ~=~ \int_{\O} \int_\R \Prob_{o,r} \, \nu(\ddo) \, \eta(\dr)$,
the same notation for expectations, and sometimes omit the initial value $R_0$ if it is irrelevant.

Now we turn to the proof of Proposition \ref{prop:uniform renewal}. 

\begin{proof}[Proof of Proposition \ref{prop:uniform renewal}]
In order to prove this result, we will combine methods from \cite{Athreya+McDonald+Ney:1978} and \cite{Nummelin+Tuominen:1983}.
Below, we describe the steps of the proof and defer the technicalities to several lemmata.
By splitting $f$ into its positive and negative part, it suffices to consider nonnegative functions which are bounded by $g$.

Let $((M_n, R_n))_{n \geq 0}$ and $(\tau_n)_{n \geq 0}$ be as in Lemma \ref{lem:regeneration}.
Define $V_0 \defeq 0$ and
\begin{equation*}
V_n ~\defeq~ \sum_{k=1}^n (R_{\tau_{k+1}-1}-R_{\tau_k-1}) ~=~ R_{\tau_{n+1}-1} - R_{\tau_1-1}.
\end{equation*}
Then, under each $\Prob_{o,r}$, $(V_n)_{n \geq 1}$ is a random walk with i.i.d.\ increments and
increment law $\Prob_{\nu \otimes \eta}(R_{\tau_1-1} \in \cdot)$
which is absolutely continuous since the law $\eta$ of $R_0$ is absolutely continuous.
Since $\E[S_1]>0$ and
\begin{equation*}
1~=~\Prob(S_n/n \to \E [S_1] \text{ as } n \to \infty) ~=~ \Prob_{\Id,0}(R_n/n \to \E[S_1] \text{ as } n \to \infty),
\end{equation*}
we deduce that $\E_{\nu \otimes \eta}[V_1] = \E_{\nu \otimes \eta}[\tau_1]\E[S_1]$.

Let $f$ be nonnegative and set
\begin{equation*}
\hat{f}(t) ~\defeq~ \E_{\nu \otimes \eta} \bigg[ \sum_{k=0}^{\tau_1-1} f(M_k, t- R_k) \bigg].
\end{equation*}
Then we can proceed as in \cite[Section 4]{Athreya+McDonald+Ney:1978} to obtain
\begin{align}
\E \bigg[ & \sum_{n=0}^\infty f(O_n, t-S_n) \bigg]
= \E_{\Id,0} \bigg[ \sum_{n=0}^\infty f(M_n, t-R_n) \bigg]	\nonumber \\
&= \E_{\Id,0} \bigg[ \sum_{n=0}^{\tau_1 -1} f(M_n, t-R_n) \bigg] \\
&\hphantom{=} + \E_{\Id,0} \bigg[ \sum_{n=1}^\infty \,
\E \bigg[\sum_{k=\tau_n}^{\tau_{n+1}-1} f(M_k, t-(R_k-R_{\tau_{n}-1}) - R_{\tau_{n}-1}) \, \Big| \, \mathcal{G}_{\tau_n-1} \bigg] \bigg]  \nonumber \\
&= \E_{\Id,0} \bigg[ \sum_{n=0}^{\tau_1 -1} f(M_n, t-R_n) \bigg]
+ \E_{\Id,0} \bigg[ \sum_{n=0}^\infty \hat{f}(t-R_{\tau_1-1}-V_n) \bigg]
\eqdef E_1(t) + E_2(t). \nonumber
\end{align}
We show in Lemma \ref{lem:rateE1} that $t^{\ell+\epsilon} E_1(t)$ tends to zero, uniformly over all $f$ with $\abs{f}\le g$.
We rewrite
\begin{equation*}
E_2(t) ~=~ \int_\R \E_{\nu \otimes \eta} \bigg[\sum_{n=0}^\infty \hat{f}(t-s-V_n)\bigg] \, \Prob_{\Id,0}(R_{\tau_1-1} \in \ds)
\end{equation*}
and use \eqref{eq:occupation measure} to infer that (recall that $f \geq 0$)
\begin{align*}
\frac{1}{\E_{\nu\otimes\eta} [V_1]} \int_\R \hat{f}(r) \, \dr
~&=~	\frac{1}{\E_{\nu\otimes\eta} [V_1]} \E_{\nu \otimes \eta} \bigg[ \sum_{k=0}^{\tau_1-1} \int_\R f(M_k, r-R_k) \, \dr \bigg]	\\
&=~	\frac{1}{\E_{\nu\otimes\eta} [V_1]} \E_{\nu\otimes\eta} \bigg[ \sum_{k=0}^{\tau_1-1} \int_\R f(M_k, r) \, \dr \bigg] \\
&=~	\frac{1}{\E[S_1]} \int_0^\infty \int_{\O} f(o,r) \, H_\O (\ddo) \, \dr.
\end{align*}
Then the claimed convergence rate holds, if
\begin{equation*}
\int_\R  \, \sup_{\abs{f} \leq g} t^{\ell+\epsilon} \abs{\E_{\nu \otimes \eta} \bigg[ \sum_{n=0}^\infty \hat{f}(t-s-V_n) \bigg]
 - \frac{1}{\E_{\nu\otimes\eta}[V_1]}\int_0^\infty \hat{f}(r) \, \dr} \, \Prob_{\Id,0}(R_{\tau_1-1} \in \ds)
\end{equation*}
tends to 0 as $t \to \infty$. This result will be established in Lemma \ref{lem:1D convergence rate}.
\qed
\end{proof}

\subsubsection*{Lemmata needed in the proof of Proposition \ref{prop:uniform renewal}}

\begin{proof}[Proof of Lemma \ref{lem:regeneration}]
Observe that $((oO_n, r+S_n))_{n \in \N_0}$ is indeed a Markov chain on $\O \times \R$
and that the increments of $S_n-S_{n-1}$ are independent of the past. 
Since $\mu$ satisfies the minorization condition \eqref{eq:minorization}, we have that for all $o \in \SOd$, $B \in \SOd$,
\begin{equation}	\label{eq:recurrence1}
\Prob(oO_{1} \in B) = \Prob(O_{1} \in o^{-1}B) \geq \gamma H_{\O}((o^{-1}B) \cap \SOd) \geq \frac\gamma2 H_{\SOd}(B).
\end{equation}
If $\O=\SOd$, this shows that $(oO_n)_{n \in \N}$ is a Doeblin chain on $\O$
(see \cite[Section 16.2]{Meyn+Tweedie:1993} for the definition). 

If $\O$ contains elements with determinant $-1$ as well, then readily $\O=\Orth$, for $\SOd \subseteq \O$,
and the product of two matrices with negative determinant has a positive determinant.
Moreover, this necessitates $\Prob(\det(O_1)=-1) >0$.
Then, for all $o \in \O \setminus \SOd$ and $B \subseteq \SOd$,
\begin{align}
\Prob(oO_{2} \in B)
~\geq&~ \int \1_{\{\det(o')=-1 \}} \Prob(O_{1} \in (oo')^{-1}B) \, \Prob(O_1 \in \ddo') \nonumber \\
\geq&~ \Prob(\det(O_1)=-1) \, \frac\gamma2 H_{\SOd}(B). \label{eq:recurrence2}
\end{align}
Thus, $(oO_n)_{n \in \N}$ is a Doeblin chain in the case $\O=\Orth$, too. 
Its unique invariant probability measure is given by the normalized Haar measure $H_\O$ on $\Orth$.

Again by the minorization condition for $\mu$,
it follows that there is an absolutely continuous measure $\eta(\dx)\defeq  \1_I(x) dx$ such that
\begin{equation}	\label{eq:minorization for MRW}
\Prob(oO_1 \in A, S_1 \in B) ~\geq~ \frac\gamma2 H_{\SOd}(A) \eta(B)
\end{equation}
for all $o \in \SOd$ and all measurable $A, B$. 

\eqref{eq:recurrence1} and \eqref{eq:recurrence2} yield that there is some $q <1$ such that, for all $o \in \O$,
\begin{equation}	\label{eq:geometric rate}
\Prob(oO_k \notin \SOd \text{ for } k=1,2) \leq q.
\end{equation}
It follows that $(oO_n)_{n \in \N}$ is $(\SOd,\gamma/2,\nu,1)$-recurrent in the sense of \cite{Athreya+McDonald+Ney:1978},
with $\nu\defeq H_{\SOd}$.
Then, $((M_n, R_n))_{n \in \N_0}$ can be constructed along the same lines as in \cite[Section 3]{Athreya+McDonald+Ney:1978}:
Under $\Prob_{o,r}$, let $((M_n, R_n))_{n \in \N}$ have the same transitions as $(oO_n,r+ S_n)_{n \in \N}$,
but whenever $O_n$ enters $\SOd$, an independent $B(1,\gamma/2)$-distributed coin is flipped.
If 1 shows up, then $(M_{n+1}, R_{n+1}-R_n)$ is generated according to $\nu \otimes \eta$,
this event we call a {\em regeneration}; if 0 shows up,
then $(M_{n+1}, R_{n+1}-R_n)$ is generated according to $(1-\frac\gamma2)^{-1}\big(\Prob((M_nO_1, S_1) \in \cdot) - \frac\gamma2 \nu \otimes \eta \big)$.
Thus, the total transition probabilities are still equal to that of $(oO_n, r+S_n)_{n \in \N}$.
Let $\tau_0=0$ and $\tau_n$ be the $n$th regeneration time,
see \cite{Athreya+Ney:1978a,Nummelin:1978} for details.
This gives Assertions 1 and 2, while Assertion 3 is proved in \cite[Theorem 6.1]{Athreya+Ney:1978a}. 
The construction together with \eqref{eq:geometric rate} show that at least every third step,
there is a uniform positive chance for regeneration, this yields Assertion 4.
\qed
\end{proof}

\begin{lemma}	\label{lem:rateE1}
Let $g$ be as in Proposition \ref{prop:uniform renewal} and $((M_n,R_n))_{n \geq 0}$ as in Lemma \ref{lem:regeneration}.
Then 
\begin{equation}	\label{eq:rateE1}
\lim_{t \to \infty} \sup_{\abs{f}\le g}  t^{\ell+\epsilon+1} \bigg|\E_{\Id,0} \bigg[\sum_{n=0}^{\tau_1-1} f(M_n, t-R_n) \bigg] \bigg| ~=~0.
\end{equation}
In particular, $\lim_{t \to \infty} t^{\ell+\epsilon+1} \hat{g}(t)=0$. Moreover,
\begin{equation}	\label{eq:rateRtau}
\lim_{t \to \infty} t^{\ell+\epsilon+1} \, \Prob_{\Id,0}(R_{\tau_1-1}>t/2) ~=~0.
\end{equation}
\end{lemma}
\begin{proof}
In order to prove \eqref{eq:rateE1}, we assume that $g(0) \leq 1$ and fix some $|f| \leq g$.
Recall that, by Lemma \ref{lem:regeneration}(iv),
$\Prob_{\Id,0}(\tau_1 >n) \leq C e^{-\lambda n}$ for some $C,\lambda > 0$.
Define $n_t=  (\log t) (\ell+2)/\lambda$, $t>0$.
Then
\begin{align*}
t^{\ell+\epsilon+1} \bigg| & \E_{\Id,0} \bigg[\sum_{n=0}^{\tau_1-1} f(M_n, t-R_n) \bigg] \bigg|
~\leq~  t^{\ell+\epsilon+1} \E_{\Id,0} \bigg[\sum_{n=0}^{\tau_1-1} g(t-R_n) \bigg]  \\
~\le&~ t^{\ell+\epsilon+1}\E_{\Id,0} \bigg[\1_{\{\tau_1 \le n_t\}}\sum_{n=0}^{\lfloor n_t \rfloor}  g(t-R_n)\big(\1_{\{R_n \le t/2\}} + \1_{\{R_n > t/2\}} \big) \bigg] \\
&~ +  t^{\ell+\epsilon+1}\E_{\Id,0} \bigg[ \1_{\{\tau_1 > n_t\}} \sum_{n=0}^{\tau_1-1} g(t-R_n) \bigg]
~\eqdef~ I_1(t) + I_2(t) + I_3(t).
\end{align*}
Recall that $g$ is decreasing on $[0, \infty)$ and $\lim_{t \to \infty} t^{\ell+1+\delta} g(t)=0$ for some $\delta>0$. This gives 
\begin{equation*}
I_1(t)
~=~  t^{\ell+\epsilon+1}\E_{\Id,0} \bigg[\1_{\{\tau_1 \le n_t\}} \sum_{n=0}^{\lfloor n_t \rfloor}  g(t-R_n) \1_{\{R_n \leq t/2\}} \bigg]
~\leq~ t^{\ell+\epsilon+1} (\lfloor n_t \rfloor + 1) g(t/2) ~\underset{t \to \infty}{\to}~0.
\end{equation*}
By Jensen's inequality, $\E[\abs{S_n}^\kappa] \le n^{k-1} \E[\abs{S_1}^\kappa]$ for all $\kappa \geq 1$.
Thus, applying Markov's inequality,
\begin{align*}
I_2(t) ~&\leq~	t^{\ell+\epsilon+1} \sum_{n=0}^{\lfloor n_t \rfloor} \Prob(S_n > t/2)
~\leq~		t^{\ell+\epsilon+1} \sum_{n=0}^{\lfloor n_t \rfloor} \frac{\E[ \abs{S_n}^{\ell+1+\delta}]}{(t/2)^{\ell+1+\delta}}	\\
~&\leq~	\frac{2^{\ell+1+\delta}}{t^{\delta-\epsilon}} \sum_{n=0}^{\lfloor n_t \rfloor} n^{\ell+\delta} \E[\abs{S_1}^{\ell+1+\delta}]
~\leq~		2^{\ell+1+\delta} \frac{n_t^{\ell+1+\delta}}{t^{\delta-\epsilon}} \E[\abs{S_1}^{\ell+1+\delta}],
\end{align*}
which tends to zero as $t \to \infty$. Finally,
\begin{align*}
I_3(t)
~&\leq~ t^{\ell+\epsilon+1} \E_{\Id,0}\big[ \tau_1 \1_{\{\tau_1 > n_t\}} \big] \\
&\leq~ t^{\ell+\epsilon+1} \bigg( n_t \Prob_{\Id,0}(\tau_1 > n_t) + \int_{n_t}^\infty \Prob_{\Id,0}(\tau_1 > r) \, \dr \bigg) \\
&\leq~ t^{\ell+\epsilon+1} \bigg( n_t C e^{-\lambda n_t} + \frac{C}{\lambda} e^{-\lambda n_t} \bigg) 	\\
&\leq~ C \frac{t^{\ell+\epsilon+1} (n_t +1/\lambda)}{t^{\ell+2}}
~\to~	0
\end{align*}
as $t \to \infty$. Thus we have proved the first and second assertion.
\eqref{eq:rateRtau} follows from \eqref{eq:rateE1} with $g(s)=\1_{[0,t/2)}(s)$.
\qed
\end{proof}

\begin{lemma}\label{lem:1D convergence rate}
It holds that
\begin{equation}	\label{eq:1D convergence rate}
\lim_{t \to \infty}
\int_\R  \, \sup_{\abs{f} \le g} t^{\ell + \epsilon}
\abs{\E_{\nu \otimes \eta} \bigg[ \sum_{n=0}^\infty \hat{f}(t-s-V_n) \bigg] - \frac{\int_0^\infty \hat{f}(r) \, \dr}{\E_{\nu \otimes \eta}[V_1]}} \, \Prob_{\Id,0}(R_{\tau_1-1} \in \ds) ~=~0.
\end{equation}
\end{lemma}
\begin{proof}
We start by proving that the functions $\hat{f}$ are uniformly directly Riemann-integrable over $\abs{f} \leq g$.
Write
\begin{equation*}
\hat{f}(t)	~=~	\int \E_{\nu \otimes \delta_0} \bigg[\sum_{k=0}^{\tau_1-1} f(M_k,t-s-R_k) \bigg] \, \eta(\ds),
\end{equation*}
then observe that
\begin{align*}
\int \bigg| \E_{\nu \otimes \delta_0} \! \bigg[\sum_{k=0}^{\tau_1-1} f(M_k,t-R_k) \bigg] \bigg| \dt
\leq	\E_{\nu \otimes \delta_0} \bigg[\sum_{k=0}^{\tau_1-1} \int g(t-R_k) \bigg] \dt
=	\E_{\nu \otimes \delta_0}[\tau_1] \int g(t) \dt
\end{align*}
and recall that $\eta$ has a bounded Lebesgue density.
Consequently, $\hat{f}$ is the convolution of a Lebesgue integrable function with a bounded function
and hence continuous, see \cite[Lemma VII.1.2]{Asmussen:2003}.
Further, arguing as in \cite[Eq.\ (4.4)]{Athreya+McDonald+Ney:1978},
\begin{align*}
\sum_{l\in \Z} \sup_{t \in [l h,(l+1)h]}  |\hat{f}(t)|
~&\leq~  \E_{\nu \otimes \eta}[\tau_1] \int_\O \sum_{l \in \Z}  \sup_{t \in [l 2h,(l+1)2h]}  \abs{f(o,t)} \, H_\O(\ddo) \\
&\leq~	\E_{\nu \otimes \eta}[\tau_1] \sum_{l \in \Z}  \sup_{t \in [l 2h,(l+1)2h]}  g(t)	~\eqdef~	C_g.
\end{align*}
Hence, it suffices to show that $g$ is directly Riemann-integrable.
The latter is clear since $g$ is monotone and Lebesgue-integrable (see \cite[Proposition V.4.1(v)]{Asmussen:2003}).
We have the uniform bound (cf. \cite[Theorem V.2.4(iii)]{Asmussen:2003})
\begin{equation*}
\sup_{t \in \R} \sup_{\abs{f} \le g} \, \E_{\nu\otimes\eta} \bigg[ \sum_{n=0}^\infty |\hat{f}(t-V_n)| \bigg]
~\leq~		C_g \E_{\nu\otimes\eta} \bigg[ \sum_{n=0}^\infty \1_{[-4h,4h]}(V_n) \bigg]
~\eqdef~	D_g~<~ \infty.
\end{equation*}
Now we decompose the integral inside the limit in \eqref{eq:1D convergence rate} according to the set $\{R_{\tau_1-1}> t/2\}$
to obtain the following upper bound
\begin{align*}
\sup_{s \leq t/2} \sup_{\abs{f} \leq g}  t^{\ell + \epsilon}
\bigg|\E_{\nu\otimes\eta} \bigg[ \sum_{n=0}^\infty \hat{f}(t-s-V_n) \bigg] - \frac{1}{\E_{\nu\otimes\eta} [V_1]}\int_0^\infty \hat{f}(r) \dr \bigg|&	\\
+ 2 D_g \, t^{\ell + \epsilon} \, \Prob_{\Id,0}(R_{\tau_1-1} > t/2)&.
\end{align*}
The second term tends to zero by \eqref{eq:rateRtau}.
For the first term, we invoke \cite[Theorem 4.2(ii)]{Nummelin+Tuominen:1983} (with $G=\delta_0$),
which gives (note that $\abs{f} \le g$ implies $|\hat{f}| \le \hat{g}$)
\begin{equation*}
\lim_{t \to \infty} t^{\ell + \epsilon} \sup_{\abs{\hat{f}} \le \hat{g}} \abs{ \E_{\nu\otimes\eta} \bigg[ \sum_{n=0}^\infty \hat{f}(t-s-V_n) \bigg]
~-~\frac{1}{\E_{\nu\otimes\eta} [V_1]}\int_0^\infty \hat{f}(r) dr } ~=~ 0
\end{equation*}
as soon as $V_1$ has positive drift (here, $\E_{\nu \otimes \eta}[V_1] = \E_{\nu \otimes \eta}[\tau_1]\E[S_1] > 0$ by the proof of Proposition \ref{prop:uniform renewal}),
a spread-out law (here, the law of $V_1$ is even absolutely continuous)
and $\E_{\nu\otimes\eta}[|V_1|^{\ell+\epsilon+1}]<\infty$
(which is true by \eqref{eq:moments Markov random walk} and Lemma \ref{lem:regeneration}(iv))
and $\hat{g}$ is bounded, Lebesgue-integrable and satisfies
\begin{equation}	\label{eq:conditions g}
t^{\ell + \epsilon} \int_t^{2t} \hat{g}(r) \dr \to 0 \quad \text{ and } \quad t^{\ell + \epsilon} \sup_{r \ge t} \hat{g}(r) \to 0 \quad \text{ as } t\to\infty.
\end{equation}
Lemma \ref{lem:rateE1} gives $\lim_{t \to \infty} t^{\ell+\epsilon+1} \hat{g}(t) = 0$,
which is sufficient for \eqref{eq:conditions g} to hold.
\qed
\end{proof}

\begin{acknowledgements}
The research has partly been carried out during visits of the authors to the Institute of Mathematical Statistics in M\"unster.
The authors would like to express their gratitude for the hospitality. We further thank Vincent Vargas for interesting discussions on the subject.
\end{acknowledgements}

\bibliographystyle{plain}      
\bibliography{Fixed_points}   

\end{document}